\documentclass[11pt,reqno,a4paper]{amsart}

\usepackage[T1]{fontenc}
\usepackage[margin=1in]{geometry}
\usepackage{amsmath,amssymb,amsthm,mathrsfs,mathtools}
\usepackage{xcolor}
\usepackage[hidelinks]{hyperref}

\numberwithin{equation}{section}
\newtheorem{theorem}{Theorem}[section]
\newtheorem{proposition}[theorem]{Proposition}
\newtheorem{lemma}[theorem]{Lemma}
\newtheorem{corollary}[theorem]{Corollary}
\newtheorem{lettertheorem}{Theorem}

\theoremstyle{definition}
\newtheorem{definition}[theorem]{Definition}
\newtheorem{notation}[theorem]{Notation}

\newtheorem{example}[theorem]{Example}
\theoremstyle{remark}
\newtheorem{remark}[theorem]{Remark}
\newcommand{\Fq}{\mathbb F_q}
\newcommand{\Cinf}{\mathbb C_\infty}
\newcommand{\OO}{\mathcal O}
\newcommand{\sfm}{\operatorname{sf}}
\newcommand{\Res}{\operatorname{res}}
\newcommand{\res}{\operatorname{res}}
\newcommand{\Ninf}{\mathrm{N}}
\newcommand{\one}{\mathbf I}
\newcommand{\two}{\mathbf {I\!I}}
\newcommand{\Tr}{\operatorname{Tr}}
\newcommand{\Cl}{\operatorname{Cl}}
\newcommand{\divisor}{\operatorname{div}}
\newcommand{\LL}{\mathbb L}
\newcommand{\LT}{\operatorname{LT}}
\newcommand{\Rcal}{\mathcal R}
\newcommand{\Hcal}{\mathcal H}
\newcommand{\Gcal}{\mathcal G}
\newcommand{\Lcal}{\mathcal L}
\newcommand{\Ia}{\mathfrak a}
\newcommand{\Ib}{\mathfrak b}

\title{Pellarin Values from Cauchy Coordinates on Arbitrary Curves}

\author{Chuangqiang Hu}
\address{School of Mathematics, Sun Yat-Sen University, Guangzhou, 510275, P.~R.~China}
\email{ \href{mailto:huchq5@mail.sysu.edu.cn}{huchq5@mail.sysu.edu.cn}}

\author{Stephen S.-T. Yau$^*$}
\thanks{ $^*$ Corresponding author.}
\address{Department of Mathematical Sciences, Tsinghua University, Beijing, 100084, P.~R.~China}
\email{ \href{mailto:yau@uic.edu}{yau@uic.edu}}

\author{Lishan Yu}
\address{Beijing Institute of Mathematical Sciences and Applications, Beijing, 101408, P.~R.~China}
\email{ \href{mailto:lishany@bimsa.cn}{lishany@bimsa.cn}}

\begin{document}
\begin{abstract}
 Let $A$ be the ring of functions regular outside a closed
point of degree $\mathrm{N}$ on a smooth projective curve over $\mathbb{F}_q$. We express the normalized weight-one ideal Pellarin value of a rank-one Drinfeld $A$-module as an ordinary Hilbert-class trace of rational Cauchy differentials. A pairing on projective period lattices identifies the differential associated with each ideal
class and computes its coefficients by finite torsion traces. For a nonfree lattice, the formula uses an isogeny with arbitrary
nonzero derivative to a module with free period lattice.
Each ideal-class subseries extends regularly to the characteristic
point with value one, so the full series takes the value $|\operatorname{Cl}(A)|$.
For every rank-one module, the quotient of the Anderson generating
differential by the Drinfeld exponential is a meromorphic lattice
sum whose value at zero is the Cauchy differential. Its Taylor
coefficients give the higher-weight lattice differentials.
For sign-normalized Hayes modules with free period lattice,
an ordinary Hilbert-class trace after projection to powers
divisible by $q^{\mathrm{N}}-1$ generates ideal values at
$n\equiv1\pmod{q^{\mathrm{N}}-1}$.
An algebraic parameter rescaling gives a trace formula for
character-twisted values at $n\equiv1\pmod{q-1}$.
The exact coefficient fields explain the difference between
these weight ranges. For sign-normalized modules, decomposition by ideal classes also relates the element-type and ideal-type series at every positive weight. Explicit computations include the Hu--Huang--Yau module with nonprincipal period lattice and an elliptic curve with an infinite point of degree three, including
the holomorphic term of its Cauchy differential. At rational
infinity, the elliptic formula recovers the Green--Papanikolas
identity. 
\end{abstract}


\normalcolor\maketitle
\normalcolor\tableofcontents
\section{Introduction}\label{sec:introduction}

 Pellarin's weight-one formula for $\Fq[t]$ expresses an element
sum in terms of the Carlitz period and the Anderson--Thakur special
function \cite[Theorem 1]{Pellarin12}; the residue normalization of
the special function comes from \cite[Section 2.5]{AT90}.
Over a general function ring, the corresponding ideal sum involves
the ideal class group. Decomposing it by classes leads to element
sums on inverse ideals, whose associated period lattices need not
be free. The problem is to identify the contribution of each class
and assemble these contributions into an explicit ideal-value
formula. 

 We express the weight-one ideal value as an ordinary
Hilbert-class trace of rational Cauchy differentials, for a curve
of arbitrary genus and a closed infinite point of arbitrary degree.
The formula applies to every rank-one Drinfeld module through an
isogeny to a module with free period lattice.
On each projective period lattice, the pairing identifies a
rational differential by its divisor and characteristic residue.
Finite torsion traces compute its coefficients over the module's
field of definition. This differential is the constant term of
a meromorphic generating family.
For sign-normalized modules with free period
lattice, the coefficient fields of this family determine which
higher-weight values admit an ordinary trace formula and which
require a character twist. 

 At rational infinity, Green--Papanikolas give an explicit
shtuka-coordinate formula for the ideal value on elliptic curves
\cite[Theorem 7.3]{GP18}, and Ferraro proves functional identities
for ideal lattices and an Anderson zeta trace identity in arbitrary
genus \cite[Theorems 7.27 and 8.9]{Ferraro24}.
Hu--Yau--Yu identify the rank-one pairing for a principal period
lattice with a Cauchy coordinate and compute it by finite torsion
traces \cite[Theorems 8.8 and 8.11]{HuCauchy}.
Their free-target isogeny construction gives element-value formulas
for arbitrary source modules
\cite[Proposition 8.15 and Theorem 8.16]{HuCauchy}.

The ideal sum requires a pairing for the ideal-isogeny targets,
whose lattices need not be free even when the source lattice is
free. We extend the Cauchy pairing to these projective lattices
and use the ideal-isogeny cocycle to assemble an ordinary
Hilbert-class trace.
The differential formulation also allows the module of regular
differentials to be nonprincipal.

 The ideal action is supplied by Drinfeld's and Hayes's theories
\cite{Drinfeld74,Hayes92}; see also \cite[Chapters 4, 7, and 8]{Goss96}.
The shtuka description goes back to Thakur \cite{Thakur93}.
Related arithmetic developments include functional identities and
Gauss--Thakur sums \cite{AP14,AP15}, values at one for Drinfeld
modules over Tate algebras \cite{APTR16}, and rationality of twisted
$L$-series over general curves \cite[Section 5]{ANDTR17}.
Pellarin's recent account places these identities in the theory
of zeta functions over curves
\cite[Sections 5--7 and the appendix]{Pellarin26}. 

\subsection{Notation and conventions}
Let $X/\Fq$ be a smooth projective geometrically connected
curve, let $\infty$ be a closed point of degree $\Ninf$, and put
$A=\Gamma(X\setminus\{\infty\},\OO_X)$ and $K=\Fq(X)$.
Write $K_\infty$ for the completion at $\infty$ and
$\Cinf$ for its completed algebraic closure, and fix an isometric
$\Fq$-embedding $\iota:K_\infty\hookrightarrow\Cinf$.
Write $\chi(a)=1\otimes a$ for the geometric copy of $a\in K$ and
$\Xi$ for the characteristic point, so that
$\chi(a)(\Xi)=\iota(a)$.
Throughout, $\varphi$ denotes a general rank-one Drinfeld
$A$-module, with exponential $\exp_\varphi$ and period lattice
$\Lambda_\varphi=\ker(\exp_\varphi)$.
We reserve $\psi$ for a module with free period lattice.

 For a nonzero ideal $\Ia\subset A$, let $\varphi_\Ia$ be the
monic ideal isogeny. Denote 
\[
u_\Ia^\varphi=\frac{\varphi_\Ia(\omega_\varphi)}{\omega_\varphi},
\]
where $\omega_\varphi$ is a nonzero special function. 
Let $ d_\Ia^\varphi=\partial\varphi_\Ia $ be the coefficient of $\tau^0$ in
$\varphi_\Ia$, equivalently its derivative as an additive polynomial. The ideal-type Pellarin series is defined by
\[
\LL_\varphi(n)=\sum_{0\ne\Ia\subset A}
                 \frac{u_\Ia^\varphi}{(d_\Ia^\varphi)^n}.
\]
The ratios $u_\Ia^\varphi$ are independent of this choice.
The series $\LL_\varphi(n)$ converges at $n=1$ for every $\varphi$,
and at every $n\ge1$ for sign-normalized $\varphi$.
The weight-one series is invariant under scalar isomorphisms.

 Let $\Hcal_\varphi$ be the Anderson generating differential
of \cite[Definition 4.19]{HuCauchy}, normalized by
$\Res_\Xi\Hcal_\varphi(\lambda)=\lambda$ for
$\lambda\in\Lambda_\varphi$.
Write $C_\varphi=K_\varphi(\Xi,\cdot)$, with
$\Res_\Xi C_\varphi=1$.
Here $K_\varphi$ is the differential coordinate of the normalized
Szeg\H{o} kernel attached to $\varphi$
\cite[Section 4.3]{HuCauchy}; we call it the Cauchy coordinate.
The geometric data and the normalization are specified in
Sections~\ref{sec:ideal-series} and~\ref{sec:projective-pairing}. 

\subsection{The ideal-value formula}
Let $H_A/K$ be the ordinary Hilbert class field, in which
$\infty$ splits completely, and view it inside $\Cinf$ over
$\iota(K)$. Then
$\operatorname{Gal}(H_A/K)\simeq\Cl(A)$.
For $\Lambda_\psi=\tilde\pi\iota(A)$, the free-lattice case of
our main formula reads
\[
\LL_\psi(1)\frac{\Hcal_\psi(\tilde\pi)}{\tilde\pi}
 =\Tr_{H_A/K}(C_\psi).
\]
For a general $\varphi$, choose a separable isogeny datum
\[
J:\varphi\longrightarrow\psi,
\qquad \Lambda_\psi=\tilde\pi\iota(A),
\]
with underlying isogeny $\eta_J$ and motive multiplier
$h_J=\eta_J(\omega_\varphi)/\omega_\varphi$. Put
\[
 \partial\eta_J=h_J(\Xi),\qquad
g_J=h_J/\partial\eta_J,\qquad z_J=\tilde\pi/\partial\eta_J.
\]
Then $\Lambda_\varphi=z_J\iota(\mathfrak i_J)$, where
$\mathfrak i_J$ is the annihilator of $\ker\eta_J$.
The transport formulas of \cite[Sections 6.2 and 6.4]{HuCauchy} give
\[
K_J(\Xi,\cdot)=\frac{C_\psi}{g_J},\qquad
\Gcal_J(z_J)=\frac{\Hcal_\psi(\tilde\pi)}{h_J},\qquad
\Res_\Xi\Gcal_J(z_J)=z_J.
\]
\begin{lettertheorem}\label{thm:intro-main}
For $J$ as above, $K_J(\Xi,\cdot)$ has coefficients in
$H_A$, and
\begin{equation}\label{eq:intro-main}
\LL_\varphi(1)\Gcal_J(z_J)
 =z_J\Tr_{H_A/K}\bigl(K_J(\Xi,\cdot)\bigr).
\end{equation}
\end{lettertheorem}
 The trace acts on arithmetic coefficients and has
$|\Cl(A)|$ terms. The formula holds for every such $J$, with
arbitrary nonzero derivative $\partial\eta_J$ and without sign normalization
of either module. Taking $J=\mathrm{id}_\psi$ recovers the
free-lattice formula above.
For a general source lattice, $z_J$ need not be a period.
At an actual source period $\iota(a)z_J$, $a\in\mathfrak i_J$,
the transported differential satisfies
$\Hcal_\varphi(\iota(a)z_J)=\chi(a)\Gcal_J(z_J)$.
These statements are proved in Theorem~\ref{thm:general-value}
and Corollary~\ref{cor:hayes-trace}. 

The residue normalization gives a characteristic value
for each class separately.
\begin{lettertheorem}\label{thm:intro-characteristic}
Each ordinary ideal-class contribution extends regularly
to $\Xi$ and takes the value one. In particular,
\begin{equation}\label{eq:intro-characteristic}
\LL_\varphi(1)|_\Xi=|\Cl(A)|\cdot1.
\end{equation}
\end{lettertheorem}
 This is Corollary~\ref{cor:class-number}; the class number
is interpreted in the coefficient field. The regular value refers
to the meromorphic continuation of the ideal series.

\subsection{Lattice pairings and finite traces}
The passage from element sums to ideal values requires a
pairing for every projective rank-one period lattice.
Special functions and period lattices are related by
\cite[Theorems 3.3 and 3.11]{GM21}; tensoring with the module
$\Omega_A=\Omega^1_{A/\Fq}$ gives the canonical residue comparison of
\cite[Theorem 5.8]{GM24}.
Poonen's exponential duality \cite[Theorem 10]{Poonen96}
underlies the lattice pairing, whose rationality in arbitrary
rank is proved in \cite[Theorem 7.3]{Ferraro25}.
In rank one, we identify the pairing explicitly with the Cauchy coordinate and obtain a finite computation.

\begin{lettertheorem}\label{thm:intro-pairing}
Let $\varphi$ be defined over $H\subset\Cinf$, with
$\iota(K)\subset H$. Choose $0\ne\tilde\pi\in\Lambda_\varphi$
and write $\Lambda_\varphi=\tilde\pi\iota(\mathfrak j)$,
where $\mathfrak j$ is a fractional ideal containing $A$.
Then
\[
C_\varphi=-\tilde\pi^{-1}
       \left(\sum_{0\ne c\in\mathfrak j}
                    \frac{\chi(c)}{\iota(c)}\right)
       \Hcal_\varphi(\tilde\pi).
\]
For $b\in A\setminus\Fq$ and $c\in A$, its residue-dual
evaluation is the finite torsion trace
\[
\langle c/b,C_\varphi\rangle _\infty
 =-\iota(b)^{-1}\Tr_{\mathscr T_b^\varphi/H}
          \left(\frac{\varphi_c(Z)}{Z}\right),
\]
where $
\mathscr T_b^\varphi=H[Z]/(\varphi_b(Z)/Z)$.
\end{lettertheorem}
The pairing $\langle\cdot,\cdot\rangle _\infty$ is the residue
duality on $K_\infty/A$ recalled in
Section~\ref{sec:projective-pairing}; the algebra
$\mathscr T_b^\varphi$ is finite \'etale.
Theorem~\ref{thm:projective-pairing} identifies the differential
by its divisor and characteristic residue, and
Proposition~\ref{prop:torsion-trace} gives the displayed trace.
After clearing the characteristic pole, finitely many such
evaluations reconstruct $C_\varphi$ in a Riemann--Roch basis
(Corollary~\ref{cor:finite-algorithm}).
Lemma~\ref{lem:dual-denominator} bounds the degree of a suitable $b$.
Applying this pairing to the ideal-isogeny targets and using the
ideal-isogeny cocycle proves Theorem~A.

\subsection{Generating functions and higher-weight traces}
Generating functions connect Drinfeld modules with their
periods and logarithms \cite{Anderson86,EP14}.
For the Carlitz module, Perkins's identity generates Pellarin
values from the Anderson generating function
\cite[Theorem 1.1 and Corollary 1.3]{Perkins14}; related functional
identities are developed in \cite{PP16}.
Here the Cauchy differential of Theorem~C is the value at zero
of a meromorphic generating family for every rank-one module.

\begin{lettertheorem}[{The Perkins' generating identity}]\label{thm:intro-generating}
For every rank-one module $\varphi$, one has the meromorphic identity
\begin{equation}\label{eq:intro-generating}
\mathscr C_\varphi(w):=
\frac{\Hcal_\varphi(w)}{\exp_\varphi(w)}
 =\sum_{0\ne\lambda\in\Lambda_\varphi}
       \frac{\Hcal_\varphi(\lambda)}{w-\lambda}.
\end{equation}
The singularity at zero is removable, with value $C_\varphi$.
For $|w|<r_\varphi=\min_{0\ne\lambda\in\Lambda_\varphi}|\lambda|$,
\[
\mathscr C_\varphi(w)=\sum_{n\ge1}C_{\varphi,n}w^{n-1},
\qquad
C_{\varphi,n}=-\sum_{0\ne\lambda\in\Lambda_\varphi}
                      \frac{\Hcal_\varphi(\lambda)}{\lambda^n}.
\]
\end{lettertheorem}
 Theorem~\ref{thm:lattice-generating} proves this identity
without a freeness or sign-normalization hypothesis.
For $\Lambda_\psi=\tilde\pi\iota(A)$, division by
$\Hcal_\psi(\tilde\pi)$ gives the arbitrary-curve analogue of
Perkins's $s=1$ identity (Corollary~\ref{cor:element-generating}).
The coefficients $C_{\varphi,n}$ are rational differentials;
Proposition~\ref{prop:higher-weight-recurrence} computes them
from $C_\varphi$, the shtuka function, and finitely many
coefficients of $\exp_\varphi$. 

To pass from this family to ideal values, let $\psi$ be
a sign-normalized Hayes module with free period lattice.
Put $Q=q^{\Ninf}$, let $\mathbb F_\infty\simeq\mathbb F_Q$ be the
residue field at infinity, identified with its lift in $K_\infty$,
and write
$G_0=\operatorname{Gal}(H_A^+/H_A)$, where $H_A^+$ is the sign
class field. For $n\ge1$ with $n\equiv1\pmod{q-1}$, the differential
$C_{\psi,n}$ is nonzero, and its smallest coefficient field
$H_n$ between $H_A$ and $H_A^+$ satisfies
(Proposition~\ref{prop:weight-field})
\[
[H_n:H_A]=\frac{Q-1}{\gcd(n-1,Q-1)}.
\]
Thus $C_{\psi,n}$ descends to $H_A$ precisely for
$n\equiv1\pmod{Q-1}$.
At the other weights $1\pmod{q-1}$, multiplication by a scalar
with the inverse $G_0$-character restores descent.

For $\mu\in(H_A^+)^\times$, the corresponding twisted ideal
series is
\[
\LL_\psi^\mu(n)=\sum_{0\ne\Ia\subset A}
        \sigma_\Ia(\mu)\frac{u_\Ia^\psi}{(d_\Ia^\psi)^n},
\]
where $\sigma_\Ia$ is the Artin automorphism of $\Ia$ on $H_A^+$.
Projection and parameter rescaling realize these two descent
procedures on the whole generating family.

\begin{lettertheorem}\label{thm:intro-generating-traces}
Let $\psi$ be a sign-normalized Hayes module with
$\Lambda_\psi=A\tilde\pi$, and put $Q=q^{\Ninf}$.
\begin{enumerate}
\item[(i)] The values at the weights $1\pmod{Q-1}$ are generated by
the Hilbert-class trace of the projected series:
\[
\sum_{k\ge0}\LL_\psi\bigl(1+k(Q-1)\bigr)z^{k(Q-1)}
 =\frac{\tilde\pi}{\Hcal_\psi(\tilde\pi)}
   \Tr_{H_A/K}\bigl((\mathcal P_Q\mathscr C_\psi)(w)\bigr)
   \Big|_{w=\tilde\pi z},
\]
where $(\mathcal P_QB)(w)=\frac1{Q-1}\sum_{\varepsilon\in
\iota(\mathbb F_\infty^\times)}B(\varepsilon w)$ retains the
powers divisible by $Q-1$;
\item[(ii)] let $\gamma\in(H_A^+)^\times$ have $G_0$-character
inverse to that of $C_{\psi,q}$, and choose
$\beta\in\Cinf^\times$ with $\beta^{q-1}=\gamma$.
The rescaled series
$\mathscr C_{\psi,\gamma}(w):=\mathscr C_\psi(\beta w)$ gives
\[
\sum_{k\ge0}\LL_\psi^{\gamma^k}\bigl(1+k(q-1)\bigr)z^{k(q-1)}
 =\frac{\tilde\pi}{\Hcal_\psi(\tilde\pi)}
   \Tr_{H_A/K}\bigl(\mathscr C_{\psi,\gamma}(w)\bigr)
   \Big|_{w=\tilde\pi z}.
\]
\end{enumerate}
Both identities hold near $z=0$. The traces act on the
Taylor coefficients, fixing $w$, before the substitution
$w=\tilde\pi z$. Finite sums over the ideal classes give
meromorphic continuations of the two generating series on the left.
\end{lettertheorem}
These are Theorem~\ref{thm:ideal-generating-trace} and
Corollary~\ref{cor:twisted-generating-trace}.
Their constant terms recover the free-lattice case of Theorem~A.
Comparing coefficients gives the untwisted higher-weight formula
of Theorem~\ref{thm:higher-weight-trace} and the twists
$\mu=\gamma^k$.
Theorem~\ref{thm:twisted-weight-trace} allows any scalar with the
required inverse character.
For $q=2$, the twisted generating formula covers every positive
weight at every degree of infinity.

For a fixed sign function, the element-type Pellarin series attached to
a fractional ideal $\mathfrak j$ is
$L_+(\mathfrak j;n)=\sum_{c\in\mathfrak j,\,\operatorname{sgn}(c)=1}
\chi(c)/\iota(c)^n$.
For sign-normalized $\varphi$, decomposition by narrow ideal
classes expresses $\LL_\varphi(n)$ as a finite linear combination
of $L_+(\Ib^{-1};n)$ at every positive integer weight
(Corollary~\ref{cor:narrow-weight-comparison}).
At rational infinity, $L_+(A;n)$ is precisely the principal-class
contribution. For higher-degree infinity, the signs of generators
enter the comparison (Proposition~\ref{prop:all-weight-comparison}).

For the module $\psi$ of Theorem~E at rational infinity,
$\mathcal P_q$ is the identity on
$\mathscr C_\psi$, and
\[ 
C_{\psi,n}=L_+(A;n)\Hcal_\psi(\tilde\pi)/\tilde\pi^n\]
for $n\equiv1\pmod{q-1}$.
The ordinary trace then directly relates the two Pellarin series,
extending the polynomial-ring rationality of
\cite[Theorem 2]{Pellarin12}.

The examples make the trace coordinates and their
normalizations explicit. On $\mathbb P^1$, we give closed
weight-one values for every $\Ninf$, including the
Hu--Huang--Yau module with nonprincipal period lattice, and
compute a scalar $\gamma$ for the higher-weight twists.
On elliptic curves, we recover the Green--Papanikolas weight-one
identity at rational infinity and work out
$y^2+y=x^3+x+1$ over $\mathbb F_2$ with an infinite point of
degree three. The latter computation includes the holomorphic
term of the Cauchy differential and a residue-dual basis for
the finite trace algorithm.

\subsection{Organization of the paper}
We now describe the structure of the paper.
 Section~\ref{sec:ideal-series} compares the two Pellarin
series by ideal classes.
Section~\ref{sec:projective-pairing} proves the projective pairing
and the weight-one results, Theorems~A--C.
Section~\ref{sec:general-value} develops the generating family,
its coefficient fields, and the trace formulas of Theorems~D
and~E.
Sections~\ref{sec:scalar} and~\ref{sec:elliptic} give the
projective-line and elliptic computations, respectively.

\section{Two Pellarin series}\label{sec:ideal-series}

Let $X/\Fq$ be a smooth, projective, geometrically connected curve of
genus $g$, let $\infty$ be a closed point of degree $\Ninf\ge1$, and put
\[
U=X\setminus\{\infty\},\qquad A=\Gamma(U,\OO_X).
\]
Let $p=\operatorname{char}(\Fq)$, $K=\operatorname{Frac}(A)$, and
$\Omega_A=\Omega^1_{A/\Fq}$; put $\Omega_A^{\mathrm{rat}}=K\otimes_A\Omega_A$.
Write $K_\infty$ for the completion at $\infty$ and
$\Cinf=\widehat{\overline{K_\infty}}$ for its completed algebraic closure.
Let $\mathfrak m_\infty$ be the maximal ideal of the valuation ring
of $K_\infty$ and $\mathbb F_\infty\simeq\mathbb F_{q^{\Ninf}}$ its residue field.
We identify $\mathbb F_\infty$ with its coefficient-field lift
in $K_\infty$.
Fix an isometric $\Fq$-embedding $\iota:K_\infty\hookrightarrow\Cinf$;
it satisfies $|\iota(a)|=q^{-\Ninf v_\infty(a)}$ for $a\in K$.
Write $\chi(a)=1\otimes a$ for the geometric copy of $a\in K$.
The characteristic point $\Xi\in X(\Cinf)$ is determined by
$\chi(a)(\Xi)=\iota(a)$ for $a\in A$.
For a fractional ideal $\mathfrak j\subset K$, we abbreviate
$\tilde\pi\iota(\mathfrak j)$ as $\tilde\pi\mathfrak j$.
When $\Lambda_\varphi=\tilde\pi\mathfrak j$, we identify
$\tilde\pi^{-1}\Lambda_\varphi$ with the fractional ideal $\mathfrak j$
via $\iota^{-1}$.
Let $\tau$ act by the $q$th power on coefficients and fix the geometric
copy; for $i\in\mathbb Z$, a superscript $(i)$ denotes
the $i$th coefficient Frobenius twist, using $\tau^{-1}$ when $i<0$. Put
$\mathbb T_A=\Cinf\widehat\otimes_{\Fq}A$, with its Gauss norm.
Following \cite[Section 6.9]{HuCauchy}, put
\[
\mathbb M_A=\mathbb T_A\otimes_A\Omega_A
\simeq\Cinf\widehat\otimes_{\Fq}\Omega_A.
\]
Since $\Omega_A$ is finite projective, realize it as a direct summand
of $A^s$ and give $\mathbb M_A$ the norm induced from $\mathbb T_A^s$.
This module is complete, coefficient Frobenius is continuous, and
$\|\chi(a)\eta\|\le\|\eta\|$ for $a\in A$ and $\eta\in\mathbb M_A$.
We view $\mathbb M_A$ in the one-dimensional
$\operatorname{Frac}(\mathbb T_A)$-vector space
$\operatorname{Frac}(\mathbb T_A)\otimes_A\Omega_A$.
For differentials $\eta_1,\eta_2$ in this space with $\eta_2\ne0$,
the quotient $\eta_1/\eta_2$ denotes the unique scalar $r$
satisfying $\eta_1=r\eta_2$.

We use $\varphi$ for a general rank-one Drinfeld module and reserve
$\psi$ for a module whose period lattice is free of rank one over $A$.
Ideal-isogeny targets of $\psi$ are denoted by $\varphi_i$; their
lattices need not be free.

Let $H\subset\Cinf$ be a field containing $\iota(K)$ and
the coefficients of the Drinfeld modules under consideration.
We enlarge $H$ as
needed to include the coefficients of finitely many isogenies.
For the trace formulas, we use sign-normalized scalar models over
the sign class field to establish descent to the ordinary Hilbert
class field.
For the basic analytic theory and ideal action, we use
\cite[Sections 3--9]{Hayes92} and \cite[Chapter 4]{Goss96}; the two
copies of $K$, Frobenius twists, and differential notation follow
\cite[Section 2.1 and Appendix A]{HuCauchy}.

For $\iota(K)\subset E\subset\Cinf$, write
$E(X)=\operatorname{Frac}(E\otimes_{\Fq}A)$ and let
$\Omega^1_{E(X)/E}$ denote its space of rational differentials.
This space also embeds in
$\operatorname{Frac}(\mathbb T_A)\otimes_A\Omega_A$.
An $\iota(K)$-embedding $\sigma:E\hookrightarrow\Cinf$ acts by
\[
\sigma(c\otimes a)=\sigma(c)\otimes a,\qquad
\sigma(d\chi(a))=d\chi(a).
\]
In particular, coefficient Galois action fixes the geometric functions
and differentials.
We regard the class fields $H_A\subset H_A^+$ as subfields
of $\Cinf$ over $\iota(K)$. In their Galois
groups and field traces, the base field $K$ means its arithmetic
copy $\iota(K)$.

\subsection{Isogeny data and motive multipliers}

For each rank-one module and its ideal-isogeny targets, we fix
the geometric realization
of \cite[Section 2.2]{HuCauchy}. The shtuka construction goes back to
\cite[Section 0.3]{Thakur93}, with the higher-degree normalizations
corrected in \cite[item 13]{ThakurUpdates}.
We use the Cauchy and residue normalizations of \cite[Section 6]{HuCauchy}.
Choose the standard geometric motive data:
a divisor $V_\varphi$ on $X_{\Cinf}$, a point
$q_\varphi\in X(\Cinf)$ above $\infty$, and a nonzero rational
function $f_\varphi\in\Cinf(X)$, satisfying
\begin{gather*}
V_\varphi\ge0,\qquad \deg V_\varphi=g,\qquad h^0(V_\varphi)=1,\qquad
\Xi,q_\varphi^{(-1)}\notin\operatorname{supp}V_\varphi,\\
\divisor(f_\varphi)=V_\varphi^{(1)}-V_\varphi+\Xi-q_\varphi,
\qquad q_\varphi\mid\infty.
\end{gather*}
Set $F_\varphi^{[0]}=1$ and
$F_\varphi^{[k]}=\prod_{j=0}^{k-1}f_\varphi^{(j)}$ for $k\ge1$.
Let $s_\varphi$ denote the identity polynomial in the
polynomial realization of the motive. In the geometric realization,
$s_\varphi$ corresponds to the section $1$ of $\OO(V_\varphi)$,
and $\tau^ks_\varphi$ corresponds to $F_\varphi^{[k]}$. These functions form a $\Cinf$-basis of
$\Gamma(U_{\Cinf},\OO(V_\varphi))$ by \cite[Lemma 2.2]{HuCauchy}.
For Hayes modules this is the realization of
\cite[Sections 2.3 and 3.1]{ANDTR17}. For a scalar conjugate
$\varphi=z\varphi_0z^{-1}$, use $V_\varphi=V_{\varphi_0}$,
$q_\varphi=q_{\varphi_0}$, and $f_\varphi=z^{q-1}f_{\varphi_0}$; then
$F_\varphi^{[k]}=z^{q^k-1}F_{\varphi_0}^{[k]}$, so the same coordinate
identities hold \cite[Remark 2.1]{HuCauchy}.

For two such modules, we use the isogeny datum of
\cite[Notation 2.3]{HuCauchy}:
\[
J=[\eta_J;h_J;\mathfrak i_J,\mu_J]:\varphi\longrightarrow\varphi',
\qquad \eta_J^!(s_{\varphi'})=h_Js_\varphi.
\]
Here $\eta_J$ is a separable isogeny, and $\eta_J^!$ maps
the motive of $\varphi'$ to that of $\varphi$.
The ideal $\mathfrak i_J$ is the annihilator of $\ker\eta_J$, and
$\eta_J=\mu_J\varphi_{\mathfrak i_J}$ with $\mu_J\in\Cinf^\times$
and $\varphi_{\mathfrak i_J}$ monic. The rational multiplier
$h_J\in\Cinf(X)$ satisfies
\begin{equation}\label{eq:intertwiner}
h_J^{(1)}f_\varphi=h_Jf_{\varphi'}.
\end{equation}
For the projection $\pi:U_{\Cinf}\to U$, the affine divisors in these
rational frames satisfy
\[
\pi^*\operatorname{div}_U(\mathfrak i_J)
 =\bigl(\divisor(h_J)-V_{\varphi'}+V_\varphi\bigr)|_{U_{\Cinf}}.
\]

The module of special functions is
\[
\sfm(\varphi)=\{\omega\in\mathbb T_A:
             \varphi_a(\omega)=\chi(a)\omega\text{ for every }a\in A\}.
\]
For each module, fix $0\ne\omega_\varphi\in\sfm(\varphi)$. It satisfies
$\omega_\varphi^{(1)}=f_\varphi\omega_\varphi$;
see \cite[Lemma 3.6 and Remark 3.10]{ANDTR17},
\cite[Proposition 3.18]{GM21}, and
\cite[Section 2.2]{HuCauchy} for the motive interpretation.

\begin{lemma}\label{lem:motive-multiplier}
If $\eta_J=\sum_{i=0}^d\alpha_i\tau^i$, then for any nonzero $\omega_\varphi\in\sfm(\varphi)$,
\begin{equation}\label{eq:ideal-coordinate}
h_J=\sum_{i=0}^d\alpha_iF_\varphi^{[i]}
     =\frac{\eta_J(\omega_\varphi)}{\omega_\varphi}.
\end{equation}
It is regular at $\Xi$, with
$h_J(\Xi)=\partial\eta_J=\alpha_0$.
\end{lemma}
\begin{proof}
 The motive realization sends $F_\varphi^{[i]}s_\varphi$ to
$\tau^i$, so $\eta_J^!(s_{\varphi'})$ has coordinate
$\sum_i\alpha_iF_\varphi^{[i]}$. This proves the first equality
and shows that $h_J$ is independent of $\omega_\varphi$.
The second equality follows from
$\omega_\varphi^{(i)}=F_\varphi^{[i]}\omega_\varphi$.
Telescoping the shtuka divisor gives
\begin{equation}\label{eq:product-divisor}
\divisor(F_\varphi^{[k]})
 =V_\varphi^{(k)}-V_\varphi
   +\sum_{j=0}^{k-1}(\Xi^{(j)}-q_\varphi^{(j)}).
\end{equation}
Since $V_\varphi$ avoids $\Xi$ and the characteristic is generic,
every $F_\varphi^{[k]}$ with $k\ge1$ is regular and vanishes at
$\Xi$. Evaluating \eqref{eq:ideal-coordinate} at $\Xi$ therefore
gives $h_J(\Xi)=\alpha_0$.
\end{proof}

The exponentials and period lattices of the source and target
are related by
\begin{equation}\label{eq:isogeny-uniformization}
\eta_J(\exp_\varphi w)=\exp_{\varphi'}(h_J(\Xi)w),\qquad
\Lambda_{\varphi'}=h_J(\Xi)\,\iota(\mathfrak i_J^{-1})\Lambda_\varphi;
\end{equation}
see \cite[Theorem 8.14]{Hayes92} for the ideal action on lattices and
\cite[Proposition 2.4]{HuCauchy} for these isogeny conventions.
For a nonzero integral ideal $\Ia$, write
$\varphi^\Ia=\Ia*\varphi$ for the target of $\varphi_\Ia$ and put
\[
d_\Ia^\varphi=\partial\varphi_\Ia,\qquad
u_\Ia^\varphi=\frac{\varphi_\Ia(\omega_\varphi)}{\omega_\varphi}.
\]
Thus $u_\Ia^\varphi=h_J$ for the monic datum
$J=[\varphi_\Ia;h_J;\Ia,1]$. This is the coordinate denoted $u_I$
in \cite[Section 3.1]{ANDTR17}. Lemma~\ref{lem:motive-multiplier} gives
$u_\Ia^\varphi(\Xi)=d_\Ia^\varphi\ne0$.

For a principal ideal $(a)$, let $\LT_\varphi(a)$ be the leading
$\tau$-coefficient of $\varphi_a$. Monic normalization
\cite[Lemma 4.4]{Hayes92} gives
\begin{equation}\label{eq:principal-normalization}
\varphi_{(a)}=\LT_\varphi(a)^{-1}\varphi_a,\qquad
u_{(a)}^\varphi=\frac{\chi(a)}{\LT_\varphi(a)},\qquad
d_{(a)}^\varphi=\frac{\iota(a)}{\LT_\varphi(a)}.
\end{equation}
Consequently $u_{(a)}^\varphi/d_{(a)}^\varphi=\chi(a)/\iota(a)$.

\subsection{The first Pellarin series}
Fix a sign function
$\operatorname{sgn}:K_\infty^\times\to\mathbb F_\infty^\times$
whose restriction to $\mathbb F_\infty^\times$ is the identity,
as in \cite[Section 2.1]{ANDTR17}, and put
$A_+=\{a\in A\setminus\{0\}:\operatorname{sgn}(a)=1\}$.
The two series below extend the element sum and the ideal sum of
\cite[Section 7, equations (90) and (95)]{GP18}.

\begin{definition}
For an integer $n\ge1$, the element-type Pellarin series is
\begin{equation}\label{eq:type-one}
L_+(A;n)=\sum_{a\in A_+}\frac{\chi(a)}{\iota(a)^n}.
\end{equation}
The number $ n $ is called the \emph{weight} of the series.
\end{definition}
For a fractional ideal $\mathfrak j\subset K$ and
$\varepsilon\in\mathbb F_\infty^\times$, we also use the sums
\begin{equation}\label{eq:sign-series}
L_\varepsilon(\mathfrak j;n)
 =\sum_{\substack{0\ne c\in\mathfrak j\\
                   \operatorname{sgn}(c)=\varepsilon}}
       \frac{\chi(c)}{\iota(c)^n},\qquad
L_+(\mathfrak j;n)=L_1(\mathfrak j;n).
\end{equation}
Thus the usual element-type series has no Drinfeld-module parameter.
For $\Ninf>1$, $A_+$ need not contain a generator of every
principal ideal: two generators differ only by an element of
$A^\times=\Fq^\times$, whereas their signs lie in
$\mathbb F_\infty^\times$.

Following \cite[Section 8.1]{HuCauchy}, write
\begin{equation}\label{eq:all-element-series}
\Lcal_{\mathfrak j}(n)
 =\sum_{0\ne c\in\mathfrak j}\frac{\chi(c)}{\iota(c)^n}.
\end{equation}
 These series converge unconditionally after multiplication
by a fixed geometric denominator. Choose $b\in A\setminus\{0\}$
with $b\mathfrak j\subset A$. Then
$\chi(b)\Lcal_{\mathfrak j}(n)=\iota(b)^n\Lcal_{b\mathfrak j}(n)$.
The numerators in $\Lcal_{b\mathfrak j}(n)$ have Gauss norm one,
and only finitely many denominators have bounded absolute value.
The summands therefore tend to zero. The same argument applies
to each sign subseries; compare \cite[Proposition 8.1]{HuCauchy}.
These sums belong to
$\chi(b)^{-1}\mathbb T_A\subset\operatorname{Frac}(\mathbb T_A)$.
For the ideal sums, we also use the completion of
$\operatorname{Frac}(\mathbb T_A)$ for the extended Gauss norm.

\begin{lemma}\label{lem:vanishing}
Let $\mathfrak j\subset K$ be a fractional ideal and
$n\ge1$.
\begin{enumerate}
\item[(i)] $\Lcal_{\mathfrak j}(n)=0$ unless $n\equiv1\pmod{q-1}$.
\item[(ii)] If $n\equiv1\pmod{q-1}$, then
\begin{equation}\label{eq:all-element-series-signed}
 \Lcal_{\mathfrak j}(n)
 =-\sum_{c\in\mathfrak j/\Fq^\times}\frac{\chi(c)}{\iota(c)^n},
\end{equation}
where the sum is over a set of representatives of the
$\Fq^\times$-orbits in $\mathfrak j\setminus\{0\}$.
\end{enumerate}
\end{lemma}
\begin{proof}
\emph{(i).} For $\zeta\in\mathbb F_q^\times$, the substitution $c\mapsto\zeta c$
preserves $\mathfrak j$ and gives
\[
\frac{\chi(\zeta c)}{\iota(\zeta c)^n}
=\zeta^{1-n}\frac{\chi(c)}{\iota(c)^n}.
\]
Since the sum converges unconditionally, it is invariant under this
reindexing, so
\[
\Lcal_{\mathfrak j}(n)=\zeta^{1-n}\Lcal_{\mathfrak j}(n)
\qquad(\zeta\in\mathbb F_q^\times).
\]
If $n\not\equiv1\pmod{q-1}$, choose $\zeta\in\mathbb F_q^\times$
with $\zeta^{1-n}\ne1$. Then $\Lcal_{\mathfrak j}(n)=0$.

\emph{(ii).} The quotient $\chi(c)/\iota(c)^n$ is invariant under
$c\mapsto\zeta c$ for $\zeta\in\Fq^\times$, because
$n\equiv1\pmod{q-1}$.  
 Each orbit has $q-1$ elements, so 
\[
\sum_{0\ne c\in\mathfrak j}\frac{\chi(c)}{\iota(c)^n} = (q-1) \sum_{c\in \mathfrak j/\Fq^\times} \frac{\chi(c)}{\iota(c)^n} = - \sum_{c\in \mathfrak j/\Fq^\times} \frac{\chi(c)}{\iota(c)^n}.
\]
\end{proof}

\begin{remark}\label{rem:Lplus-vs-L}
The all-element series $\Lcal_{\mathfrak j}(n)$ vanishes
when $n\not\equiv1\pmod{q-1}$, whereas the sign-one series need
not vanish. The character twists in Section~\ref{sec:twisted-traces}
apply within the range $n\equiv1\pmod{q-1}$: they remove the
additional obstruction to descent when
$n\not\equiv1\pmod{q^{\Ninf}-1}$.
\end{remark}

At rational infinity, each $\Fq^\times$-orbit in $\mathfrak j\setminus\{0\}$ has a unique element of sign one, so
\begin{equation}\label{eq:monic-all-element}
\Lcal_{\mathfrak j}(n)=
\begin{cases}
 -L_+(\mathfrak j;n),&n\equiv1\pmod{q-1},\\
 0,&n\not\equiv1\pmod{q-1}.
\end{cases}
\qquad(\Ninf=1).
\end{equation}
For integral $\mathfrak j$ at rational infinity,
$\Lcal_{\mathfrak j}(1)$ is the series $\zeta_{\mathfrak j}$ of
\cite[Section 5]{Ferraro24}.

\subsection{The second Pellarin series}
\begin{definition}
For an integer $n\ge1$ and a fixed rank-one module $\varphi$,
the ideal-type Pellarin series is
\begin{equation}\label{eq:ideal-series}
\LL_\varphi(n)=\sum_{0\ne\Ia\subset A}
                \frac{u_\Ia^\varphi}{(d_\Ia^\varphi)^n},
\end{equation}
whenever the sum converges. The sum runs over nonzero integral ideals.
\end{definition}
The numerator $u_\Ia^\varphi$ is the isogeny multiplier of
\cite[Section 3.1]{ANDTR17}. For a principal ideal,
\eqref{eq:principal-normalization} gives
\begin{equation}\label{eq:principal-weight}
\frac{u_{(a)}^\varphi}{(d_{(a)}^\varphi)^n}
 =\LT_\varphi(a)^{n-1}\frac{\chi(a)}{\iota(a)^n}.
\end{equation}
This expression is unchanged when $a$ is multiplied by an element
of $\Fq^\times$. If $\Ninf=1$ and $\varphi$ is sign-normalized,
then $\LT_\varphi(a)=1$ for $a\in A_+$, and the principal-ideal
subseries is exactly $L_+(A;n)$ for every $n\ge1$.

For the value at $n=1$, put
\begin{equation}\label{eq:ideal-definition}
\LL_\varphi(1)=\sum_{0\ne\Ia\subset A}r_\Ia^\varphi,\qquad
r_\Ia^\varphi=\frac{u_\Ia^\varphi}{d_\Ia^\varphi}.
\end{equation}

For a sign-normalized module $\psi$ with $\deg\infty=1$
and free period lattice, $\LL_\psi(1)$ is Ferraro's Anderson zeta function
\cite[Definition 8.7]{Ferraro24}.
That definition uses a lattice-dependent range of fractional ideals
for other lattice classes; here the range in \eqref{eq:ideal-series}
is always the integral ideals of $A$.

The ideal Pellarin $\LL$-series can be decomposed by ideal classes.
\begin{notation}
Choose integral representatives $\Ia_1,\ldots,\Ia_h$ for
$\Cl(A)$, where $h=h(A)=|\Cl(A)|$, and set $\mathfrak j_i=\Ia_i^{-1}$. For each ordinary ideal class, write
\[
\LL_{\varphi,[\Ia_i]}(n)
 :=\sum_{\substack{0\ne\Ia\subset A\\{[\Ia]=[\Ia_i]}}}
       \frac{u_\Ia^\varphi}{(d_\Ia^\varphi)^n},
\]
whenever the subseries converges. 
\end{notation}

Summing over the finitely many ideal classes
gives 
\begin{equation}\label{eq:class-decomposition-sum}
      \LL_\varphi(n) = \sum_{i=1}^h \LL_{\varphi,[\Ia_i]}(n).
\end{equation}

\subsection{Comparison at weight one}
For a nonzero integral ideal $\Ia$, put
$\deg\Ia=\dim_{\Fq}(A/\Ia)$.
\begin{lemma}\label{lem:rescaling}
If $c\in K^\times$ and both $\Ia,c\Ia$ are integral, then
\begin{equation}\label{eq:rescaling}
r_{c\Ia}^\varphi=\frac{\chi(c)}{\iota(c)}r_\Ia^\varphi.
\end{equation}
For every $z\in\Cinf^\times$, one also has
$r_\Ia^{z\varphi z^{-1}}=r_\Ia^\varphi$.
\end{lemma}
\begin{proof}
For $a\in A\setminus\{0\}$, the monic ideal polynomial for $(a)$
is a scalar multiple of the action polynomial at $a$, for both
$\varphi$ and $\varphi^\Ia$. Ideal-isogeny composition and
$\varphi_\Ia\varphi_a=\varphi_a^\Ia\varphi_\Ia$ therefore give
$\varphi_{a\Ia}(\omega_\varphi)=\gamma\chi(a)\varphi_\Ia(\omega_\varphi)$
and $d_{a\Ia}^\varphi=\gamma \iota(a) d_\Ia^\varphi$ with the same scalar
$\gamma$. Taking the ratio proves \eqref{eq:rescaling} for
$c=a\in A\setminus\{0\}$. Write $c=a/b$ and compare
$b(c\Ia)=a\Ia$ to prove the general case. If
$\varphi^z=z\varphi z^{-1}$ and $d=\deg\Ia$, monic normalization gives
\[
 (\varphi^z)_\Ia=z^{q^d}\varphi_\Ia z^{-1},\qquad
 (\varphi^z)^\Ia=z^{q^d}\varphi^\Ia z^{-q^d}.
\]
The second equality follows by substituting the first in the
ideal-intertwining identity. Since $z\omega_\varphi$ is special for $\varphi^z$,
\[
u_\Ia^{\varphi^z}=z^{q^d-1}u_\Ia^\varphi,\qquad
d_\Ia^{\varphi^z}=z^{q^d-1}d_\Ia^\varphi.
\]
The common factor cancels in $r_\Ia^{\varphi^z}$.
\end{proof}
\begin{proposition}\label{prop:class-decomposition}
At weight one, each class subseries converges and satisfies
\begin{equation}\label{eq:class-subseries}
\LL_{\varphi,[\Ia_i]}(1)=-r_{\Ia_i}^\varphi\Lcal_{\mathfrak j_i}(1).
\end{equation}
\end{proposition}
\begin{proof}
The integral ideals in the class of $\Ia_i$ are exactly
$c\Ia_i$ with $0\ne c\in\Ia_i^{-1}$, modulo multiplication by
$A^\times=\Fq^\times$. The quantity $\chi(c)/\iota(c)$ is constant
on each such orbit. Each orbit has $q-1$ elements, and $q-1=-1$ in $H$.
Lemma~\ref{lem:rescaling} therefore gives
\eqref{eq:class-subseries}. The convergence just established for
$\Lcal_{\mathfrak j_i}(1)$ proves convergence of each class subseries and,
since there are finitely many classes, of the full ideal series.
\end{proof}

The right side of \eqref{eq:class-subseries} is therefore independent
of the representative $\Ia_i$.
Summing over the classes gives
\begin{equation}\label{eq:class-decomposition}
\LL_\varphi(1)=-\sum_{i=1}^h r_{\Ia_i}^\varphi \Lcal_{\mathfrak j_i}(1).
\end{equation}

\subsection{Comparison at arbitrary positive integers}
Let $\varphi$ be sign-normalized. Following
\cite[Section 2.2]{ANDTR17}, choose $\xi_\varphi\in\{0,\ldots,\Ninf-1\}$
such that
\[ 
\LT_\varphi(a)=\iota(\operatorname{sgn}(a))^{q^{\xi_\varphi}}.\]

Put $\xi_i=\xi_\varphi+\deg\Ia_i$, and choose representatives
$S\subset\mathbb F_\infty^\times$ of
$\mathbb F_\infty^\times/\Fq^\times$.

\begin{proposition}\label{prop:all-weight-comparison}
For every integer $n\ge1$, the ideal-type series of $\varphi$ converges
unconditionally, and each ideal-class subseries satisfies
\begin{equation}\label{eq:all-weight-comparison}
\LL_{\varphi,[\Ia_i]}(n)
 =\frac{u_{\Ia_i}^\varphi}{(d_{\Ia_i}^\varphi)^n}
   \sum_{\varepsilon\in S}
      \iota(\varepsilon)^{(n-1)q^{\xi_i}}
      L_\varepsilon(\Ia_i^{-1};n).
\end{equation}
\end{proposition}
\begin{proof}
For $a\in A\setminus\{0\}$, the monic isogeny attached to
the ideal $a\Ia_i$ is
\[
\varphi_{a\Ia_i}
 =\iota(\operatorname{sgn}(a))^{-q^{\xi_i}}\varphi_{\Ia_i}\varphi_a.
\]
Indeed, $\varphi_{\Ia_i}\varphi_a$ has kernel
$\varphi[a\Ia_i]$ and leading coefficient
$\LT_\varphi(a)^{q^{\deg\Ia_i}}$.
Taking the $\tau^0$ coefficients gives
$d_{a\Ia_i}^\varphi
 =\iota(a)\iota(\operatorname{sgn}(a))^{-q^{\xi_i}}d_{\Ia_i}^\varphi$.
For $c=a/b\in K^\times$ with $c\Ia_i$ integral, compare
$b(c\Ia_i)=a\Ia_i$. The degrees of $c\Ia_i$ and $\Ia_i$
differ by $-\Ninf v_\infty(c)$, so their Frobenius exponents
agree on $\mathbb F_\infty$. Hence
\begin{equation}\label{eq:hayes-principal-rescaling}
d_{c\Ia_i}^\varphi
 =\iota(c)\iota(\operatorname{sgn}(c))^{-q^{\xi_i}}d_{\Ia_i}^\varphi.
\end{equation}
Combining this with Lemma~\ref{lem:rescaling} yields
\begin{equation}\label{eq:weight-rescaling}
\frac{u_{c\Ia_i}^\varphi}{(d_{c\Ia_i}^\varphi)^n}
 =\frac{u_{\Ia_i}^\varphi}{(d_{\Ia_i}^\varphi)^n}
  \iota(\operatorname{sgn}(c))^{(n-1)q^{\xi_i}}
  \frac{\chi(c)}{\iota(c)^n}.
\end{equation}

The map $c\mapsto c\Ia_i$ identifies
$(\Ia_i^{-1}\setminus\{0\})/\Fq^\times$ with the integral ideals
in $[\Ia_i]$. Each orbit has exactly one representative whose sign
lies in $S$. Summing \eqref{eq:weight-rescaling} over those
representatives proves \eqref{eq:all-weight-comparison}.
The sign factors have absolute value one. Convergence of
the finitely many element sums \eqref{eq:sign-series} therefore
gives unconditional convergence of each class subseries and,
since there are finitely many ideal classes, of $\LL_\varphi(n)$.
\end{proof}

Both sides of \eqref{eq:all-weight-comparison} are independent of
the chosen integral representative of $[\Ia_i]$ and of $S$.
To check independence of $S$, replace $\varepsilon$ by
$\zeta\varepsilon$, with $\zeta\in\Fq^\times$. This multiplies
$L_\varepsilon$ by $\zeta^{1-n}$ and the sign factor by
$\zeta^{n-1}$.

Grouping ideals into narrow classes gives a formula in sign-one sums.
Let $I(A)$ be the group of nonzero fractional ideals and put
\[
P^+(A)=\{cA:c\in K^\times,\ \operatorname{sgn}(c)=1\},
\qquad \Cl^+(A)=I(A)/P^+(A).
\]
This narrow ideal class group has order
$h^+=h(q^{\Ninf}-1)/(q-1)$; see
\cite[Section 2.1]{ANDTR17}. Choose integral representatives
$\Ib_1,\ldots,\Ib_{h^+}$.

\begin{corollary}\label{cor:narrow-weight-comparison}
Let $\varphi$ be sign-normalized. For every $n\ge1$,
\begin{equation}\label{eq:narrow-weight-comparison}
\LL_\varphi(n)=\sum_{j=1}^{h^+}
       \frac{u_{\Ib_j}^\varphi}{(d_{\Ib_j}^\varphi)^n}
       L_+(\Ib_j^{-1};n).
\end{equation}
Each term is the subseries for the corresponding narrow ideal class.
\end{corollary}
\begin{proof}
Every integral ideal in the narrow class of $\Ib_j$ has a
unique expression $c\Ib_j$ with $c\in\Ib_j^{-1}$ and
$\operatorname{sgn}(c)=1$.
Uniqueness follows because the only unit of $A$ of sign one is $1$.
Apply \eqref{eq:weight-rescaling}, with $\Ia_i$ replaced by
$\Ib_j$: its sign factor is one. Summing gives the claimed subseries,
and then \eqref{eq:narrow-weight-comparison}.
\end{proof}

For every $\Ninf$, $L_+(A;n)$ is therefore the ideal-type subseries
over ideals generated by an element of sign one.
An ordinary ideal class combines $(q^{\Ninf}-1)/(q-1)$ narrow
classes; \eqref{eq:all-weight-comparison} expresses that combination
through the individual signs in a single inverse ideal.

\begin{corollary}\label{cor:type-one-type-two}
Let $\varphi$ be sign-normalized and assume $\Ninf=1$.
Then, for every $n\ge1$,
\begin{equation}\label{eq:type-one-type-two}
\LL_\varphi(n)=\sum_{i=1}^h
    \frac{u_{\Ia_i}^\varphi}{(d_{\Ia_i}^\varphi)^n}
       L_+(\Ia_i^{-1};n),\qquad
\LL_{\varphi,[A]}(n)=L_+(A;n).
\end{equation}
\end{corollary}
\begin{proof}
Take $S=\{1\}$ in \eqref{eq:all-weight-comparison}, and use
$\Ia_i=A$ for the principal class.
\end{proof}

Equation~\eqref{eq:type-one-type-two} extends the weight-one class
decomposition of
\cite[Section 7, equations (97)--(98)]{GP18}.
For $\Ninf>1$, the sign components in
\eqref{eq:all-weight-comparison} are necessary even for the principal class.
The equivalent sum over all nonzero elements is
\begin{equation}\label{eq:all-weight-signed}
\LL_{\varphi,[\Ia_i]}(n)
 =-\frac{u_{\Ia_i}^\varphi}{(d_{\Ia_i}^\varphi)^n}
  \sum_{0\ne c\in\Ia_i^{-1}}
     \iota(\operatorname{sgn}(c))^{(n-1)q^{\xi_i}}
     \frac{\chi(c)}{\iota(c)^n}.
\end{equation}
The minus sign is $1/(q-1)$ in the coefficient field. If
$n\equiv1\pmod{q^{\Ninf}-1}$, the sign factors disappear, and
\begin{equation}\label{eq:all-weight-unweighted}
\LL_{\varphi,[\Ia_i]}(n)
 =-\frac{u_{\Ia_i}^\varphi}{(d_{\Ia_i}^\varphi)^n}
       \Lcal_{\Ia_i^{-1}}(n).
\end{equation}
At $n=1$ this recovers \eqref{eq:class-subseries}.

\begin{remark}\label{rem:weight-scalar-model}
Let $\varphi_0$ be sign-normalized, take $z\in\Cinf^\times$,
and put $\varphi=z\varphi_0z^{-1}$. The computation in
Lemma~\ref{lem:rescaling} gives
\begin{equation}\label{eq:weight-scalar-model}
\frac{u_\Ia^\varphi}{(d_\Ia^\varphi)^n}
 =z^{(1-n)(q^{\deg\Ia}-1)}
       \frac{u_\Ia^{\varphi_0}}{(d_\Ia^{\varphi_0})^n}.
\end{equation}
For $n>1$, the ideal series converges unconditionally if and only if
$|z|\ge1$. For $|z|\ge1$, the extra factors have absolute value at
most one, so the preceding convergence proof applies. For $|z|<1$,
choose $a\in A\setminus\Fq$ and take $\Ia=(a^m)$. With
$D_m=m\deg(aA)$, the Gauss norm of this summand is
\[
 |z|^{(1-n)(q^{D_m}-1)}q^{-nD_m},
\]
which tends to infinity. In contrast, at $n=1$ the factor in
\eqref{eq:weight-scalar-model} is one for every $z$.
\end{remark}

\section{Projective lattice pairings and ideal values}
\label{sec:projective-pairing}

The class decomposition \eqref{eq:class-subseries} reduces
the weight-one ideal value to element sums on inverse ideals.
Pairing these sums with period differentials yields the rational
Cauchy differentials of the corresponding ideal-isogeny targets.
Coefficient descent then expresses the finite class sum as an
ordinary Hilbert-class trace.

\subsection{The generating differential and period maps}
For a rank-one module $\varphi$ with lattice $\Lambda_\varphi$,
the residue at $\Xi$, taken after meromorphic continuation, identifies
$\sfm(\varphi)\otimes_A\Omega_A$ with $\Lambda_\varphi$ \cite[Theorem 5.8]{GM24}; see also
\cite[Corollary 5.21 and Proposition 5.23]{HuCauchy}.
\begin{definition}
On the period lattice, the generating differential of
$\varphi$ is the inverse of this residue isomorphism:
\begin{equation}\label{eq:period-differential}
\begin{aligned}
\Lambda_\varphi & \longrightarrow\sfm(\varphi)\otimes_A\Omega_A
\subset\mathbb M_A,\\
\lambda & \longmapsto \Hcal_{\varphi}(\lambda).
\end{aligned}
\end{equation}
Thus $\Res_\Xi\Hcal_\varphi(\lambda)=\lambda$ for every
$\lambda\in\Lambda_\varphi$.
\end{definition}

Continuous residue duality \cite[Sections 5.1 and 5.6]{HuCauchy}
gives the topological isomorphism
\begin{equation}\label{eq:continuous-residue-duality}
\mathbb M_A
\simeq\operatorname{Hom}^{\mathrm{cont}}_{\Fq}(K_\infty/A,\Cinf).
\end{equation}
For $c\in K_\infty$ and $\nu\in\Omega_A$, the pairing is
$\langle c,\nu\rangle _\infty
=\Tr_{\mathbb F_\infty/\Fq}\res_\infty(c\nu)$;
it extends $\Cinf$-linearly and continuously in the
second argument to $\mathbb M_A$.
Under this identification, $\Hcal_\varphi(\lambda)$ corresponds to
$c+A\mapsto\exp_\varphi(\iota(c)\lambda)$.
An isogeny acts on $\mathbb M_A$ through its polynomial in coefficient
Frobenius. For a separable isogeny $\eta:\varphi\to\varphi'$, the period map satisfies
\begin{equation}\label{eq:period-isogeny}
\eta(\Hcal_{\varphi}(\lambda))
       =\Hcal_{\varphi'}(\partial\eta\,\lambda).
\end{equation}
Indeed, evaluation at $c+A\in K_\infty/A$ gives
\[
\bigl\langle c,\eta(\Hcal_\varphi(\lambda))\bigr\rangle _\infty
 =\eta\bigl(\exp_\varphi(\iota(c)\lambda)\bigr)
 =\exp_{\varphi'}(\iota(c)\partial\eta\,\lambda).
\]
The last equality is \eqref{eq:isogeny-uniformization}, and residue
duality proves \eqref{eq:period-isogeny}. In particular,
\begin{equation}\label{eq:period-scalar}
\Hcal_{z\varphi z^{-1}}(z\lambda)=z\Hcal_\varphi(\lambda).
\end{equation}
If $c=a/b\in K$ and $\iota(c)\lambda\in\Lambda_\varphi$, then
$A$-linearity gives
\[
\chi(b)\Hcal_\varphi(\iota(c)\lambda)
 =\Hcal_\varphi(\iota(a)\lambda)
 =\chi(a)\Hcal_\varphi(\lambda).
\]
Division by $\chi(b)$ gives
$\Hcal_\varphi(\iota(c)\lambda)=\chi(c)\Hcal_\varphi(\lambda)$
after extending coefficients to $\operatorname{Frac}(\mathbb T_A)$.
For $\lambda\ne0$, the meromorphic continuation of
$\Hcal_\varphi(\lambda)$ has a simple pole at $\Xi$ with residue $\lambda$
\cite[Proposition 4.23]{HuCauchy}.

The line bundle $\OO(V_\varphi-q_\varphi^{(-1)})$ is acyclic: a nonzero
section would be a nonzero constant and hence could not vanish at
$q_\varphi^{(-1)}$, and Riemann--Roch gives $h^1=0$.
The normalized shtuka differential of \cite[Definition 6.1]{HuCauchy}
is the rational differential characterized by
\[
\widetilde\lambda_\varphi\in
H^0\left(X_{\Cinf},\Omega_X^1
   (-V_\varphi+q_\varphi^{(-2)}+q_\varphi^{(-1)})\right),\qquad
\Res_\Xi\frac{\widetilde\lambda_\varphi^{(1)}}{f_\varphi}=1.
\]
The normalized Cauchy coordinate of \cite[Theorem 6.2]{HuCauchy}
has characteristic restriction
\begin{equation}\label{eq:cauchy-differential}
C_\varphi:=K_\varphi(\Xi,\cdot)
 =\frac{\widetilde\lambda_\varphi^{(1)}}{f_\varphi},\qquad
\Res_\Xi C_\varphi=1.
\end{equation}
The differential $C_\varphi$ is the unique element with residue
one at $\Xi$ in the one-dimensional space
\begin{equation}\label{eq:cauchy-line}
H^0\left(X_{\Cinf},\Omega_X^1(-V_\varphi+\Xi+q_\varphi^{(-1)})\right).
\end{equation}
Indeed, Serre duality gives
$H^0(X_{\Cinf},\Omega_X^1(-V_\varphi+q_\varphi^{(-1)}))=0$.
Allowing a simple pole at $\Xi$ adds one dimension.
The residue map at $\Xi$ therefore identifies the space in
\eqref{eq:cauchy-line} with $\Cinf$.
For an arbitrary parameter $w\in\Cinf$, we use the
Szeg\H{o}-kernel construction of \cite[Definition 4.19]{HuCauchy}:
\begin{equation}\label{eq:entire-generating-differential}
\Hcal_\varphi(w)=\mathcal E^\Xi_{\varphi,w}(K_\varphi s_\varphi).
\end{equation}
Here $\mathcal E^\Xi_{\varphi,w}$ is the orbital evaluation of the
first-variable motive, with $s_\varphi$ the identity frame.
This is an entire $\Fq$-linear differential-valued function of $w$.
Its restriction to $\Lambda_\varphi$ is
\eqref{eq:period-differential} by
\cite[Theorem 4.20 and Corollary 5.21]{HuCauchy}.
Thus the values away from the lattice are fixed by the normalized
Cauchy kernel.

 For a differential-valued power series, define the parameter
Frobenius by
\begin{equation}\label{eq:parameter-frobenius}
\tau_w\left(\sum_{j\ge0}B_jw^j\right)
 :=\sum_{j\ge0}B_j^{(1)}w^{qj}.
\end{equation}
Thus $\tau_w$ acts on both the arithmetic coefficients and $w$,
while fixing the geometric variable. For a fixed $w\in\Cinf$,
$(\tau_w B)(w)=B(w)^{(1)}$. In particular,
$\tau_w(\exp_\varphi)=\exp_\varphi^q$.

\begin{lemma}\label{lem:generating-frobenius}
For every rank-one module $\varphi$,
\begin{equation}\label{eq:generating-frobenius}
\tau_w (\Hcal_\varphi) -f_\varphi\Hcal_\varphi
       =-f_\varphi C_\varphi\exp_\varphi.
\end{equation}
\end{lemma}
\begin{proof}
Apply the identity-kernel argument of \cite[Theorem 6.5]{HuCauchy}.
Write $\one$ and $\two$ for the first and second
geometric variables of the kernel.
The first-variable divisor bounds for
$f_\varphi(\one)K_\varphi^{(1)}$ and
$f_\varphi(\two)K_\varphi$ are, respectively,
\[
 -V_\varphi+[\Xi]-[\two],
\qquad
 -V_\varphi+[q_\varphi^{(-1)}]-[\two].
\]
Their diagonal residues agree, so their difference has poles
bounded by $V_\varphi$ in the first variable.
Since $h^0(V_\varphi)=1$, that difference is independent of
the first variable. Specializing at $\one=\Xi$ gives
$-f_\varphi C_\varphi$: the first term vanishes there.
Orbital evaluation, as in \cite[Theorem 6.14]{HuCauchy},
now yields \eqref{eq:generating-frobenius}, since evaluation
of $s_\varphi$ is $\exp_\varphi(w)$.
\end{proof}

\subsection{Residue-dual denominators}
A differential with poles bounded by $m\infty$ is determined
by finitely many residue pairings. The following bound places a
complete set of test classes in $b^{-1}A/A$. It follows from the
residue theorem and the degree of a canonical divisor; see
\cite[Chapter 6]{Rosen02}.
\begin{lemma}\label{lem:dual-denominator}
Let $D=m\infty$ with $m\ge1$. If $b\in A\setminus\{0\}$ satisfies
$\deg(bA)>2g-2+m\Ninf$, then the residue map
\[
A/bA\longrightarrow H^0(X,\Omega_X^1(D))^*,\qquad
c\longmapsto\bigl(\nu\mapsto\langle c/b,\nu\rangle _\infty\bigr)
\]
is surjective.
\end{lemma}
\begin{proof}
Changing $c$ by $bA$ does not change the pairing, since a differential
regular on $U$ pairs trivially with $A$. Suppose $\nu$ is annihilated
by all these classes. The global residue theorem and local residue
duality at the zero divisor $Z_b$ of $b$ imply that $\nu/b$ is
regular on $U$. Indeed, the Chinese remainder theorem allows the
local residue test at each zero, with its full multiplicity, independently.
Thus $\nu$ vanishes on $Z_b$. A nonzero differential with poles bounded
by $m\infty$ can have at most $2g-2+m\Ninf$ finite zeros, contradicting
$\deg Z_b=\deg(bA)$ and the strict inequality. The transpose map is
therefore injective, which proves surjectivity. Riemann--Roch supplies
$b$ of sufficiently large pole order.
\end{proof}

\begin{remark}\label{rem:local-pairing-criterion}
Choose a uniformizer $t_\infty$ at $\infty$. For
$\nu\in\Omega_A$ with local expansion
$\sum_j a_jt_\infty^jdt_\infty$ and any integer $m$, vanishing of
the pairings of $\nu$ with $t_\infty^\ell\mathbb F_\infty$ for
every $\ell\ge m$ is equivalent to $a_j=0$ for $j<-m$. This uses the
nondegenerate residue-field trace and remains valid when
$p\mid\Ninf$.
\end{remark}

For $m$ and $b$ satisfying the lemma, choose an $\Fq$-basis
$\nu_1,\ldots,\nu_d$ of $H^0(X,\Omega_X^1(m\infty))$.
Surjectivity supplies $c_1,\ldots,c_d\in A$
with $\langle c_i/b,\nu_j\rangle _\infty=\delta_{ij}$.
These numerators are obtained by solving a finite linear system over $\Fq$.

\subsection{The lattice pairing}
Theorem~\ref{thm:projective-pairing} extends \cite[Theorem 8.11]{HuCauchy} to
projective period lattices. At rational infinity, the corresponding
relation between ideal-lattice series and special functions is given
in \cite[Theorem 7.27]{Ferraro24}. Here the Cauchy normalization and
finite torsion formula apply to infinite points of arbitrary degree.
No principality assumption is imposed on the period lattice
or on $\Omega_A$.

Choose $A$-generators $\gamma_1,\ldots,\gamma_r$ of $\Lambda_\varphi$.
For $\lambda=\sum_j\iota(a_j)\gamma_j$, the $A$-linearity of
\eqref{eq:period-differential} gives
\[
\Hcal_\varphi(\lambda)^{(k)}
 =\sum_j\chi(a_j)\Hcal_\varphi(\gamma_j)^{(k)},\]
with the norm estimate
 \[
 \|\Hcal_\varphi(\lambda)^{(k)}\|
\le M_k:=\max_j\|\Hcal_\varphi(\gamma_j)^{(k)}\|.
\]
Thus, for every $n\ge1$ and $k\ge0$,
\begin{equation}\label{eq:period-sum-bound}
\left\|\frac{\Hcal_\varphi(\lambda)^{(k)}}{\lambda^n}\right\|
\le\frac{M_k}{|\lambda|^n}\longrightarrow0
\qquad(|\lambda|\to\infty).
\end{equation}
The lattice has finitely many elements of bounded norm, so these
terms are unconditionally summable in $\mathbb M_A$.

\begin{notation}
Let $\Lambda_\varphi=\tilde{\pi}\mathfrak j$, where $\mathfrak j\subset K$ is a fractional
ideal containing $A$ and $0\ne\tilde{\pi}\in\Lambda_\varphi$.
For $k\ge0$, set
\[
\Rcal_\varphi^{[k]}
 =-\sum_{0\ne\lambda\in\Lambda_\varphi}
       \frac{\Hcal_\varphi(\lambda)^{(k)}}{\lambda}
\in\mathbb M_A.
\]
As in \cite[Section 8.3]{HuCauchy}, set
$\Rcal_\varphi=\Rcal_\varphi^{[0]}$ and write
$\langle u,\Rcal_\varphi^{[k]}\rangle _\infty$ for the
residue-dual evaluation on $u+A\in K_\infty/A$.
\end{notation}

After extending coefficients, $A$-linearity of the period map and
the substitution $\lambda=\tilde\pi\iota(c)$ give, in
$\operatorname{Frac}(\mathbb T_A)\otimes_A\Omega_A$,
\begin{equation}\label{eq:lattice-product}
\Rcal_\varphi^{[k]}
 =-\tilde\pi^{-1}\Lcal_{\mathfrak j}(1)
       \Hcal_\varphi(\tilde\pi)^{(k)}.
\end{equation}
Indeed, multiplication by a common denominator $\chi(b)$ with
$b\mathfrak j\subset A$ reduces the identity to equality of convergent
sums in $\mathbb M_A$. In particular, the product in
\eqref{eq:lattice-product} belongs to the domain of residue duality.

For $b\in A\setminus\Fq$, set
$\mathscr T_b^\varphi=H[Z]/(\varphi_b(Z)/Z)$.
Since $\varphi_b'(Z)=\iota(b)\ne0$, this algebra is
finite \'etale, and $Z$ is invertible in it because
$(\varphi_b(Z)/Z)|_{Z=0}=\iota(b)$.

\begin{proposition}\label{prop:torsion-trace}
For every $k\ge0$, $b\in A\setminus\Fq$, and $c\in A$,
\begin{equation}\label{eq:torsion-trace}
\langle c/b,\Rcal_\varphi^{[k]}\rangle _\infty
 =-\iota(b)^{-1}\Tr_{\mathscr T_b^\varphi/H}
                 \left(\frac{\varphi_c(Z)^{q^k}}{Z}\right).
\end{equation}
\end{proposition}

\begin{proof}
By \eqref{eq:period-sum-bound}, continuity of residue-dual evaluation
allows us to evaluate the defining sum term by term:
\begin{equation}\label{eq:lattice-sum}
\langle u,\Rcal_\varphi^{[k]}\rangle _\infty
  =-\sum_{0\ne\lambda\in\Lambda_\varphi}
            \frac{\exp_\varphi(\iota(u)\lambda)^{q^k}}{\lambda},
\qquad u\in K_\infty.
\end{equation}
Unconditional convergence permits grouping into lattice cosets.
Take $u=c/b$ and group the sum modulo $\iota(b)\Lambda_\varphi$.
The zero coset contributes zero. For a nonzero coset represented by
$\lambda_0$, put $x=\exp_\varphi(\lambda_0/\iota(b))$.
The lattice exponential has the product expansion of
\cite[Section 8]{Hayes92} and derivative one; its logarithmic derivative,
evaluated at $\lambda_0/\iota(b)$, gives
\[
\sum_{\mu\in\iota(b)\Lambda_\varphi}\frac1{\lambda_0+\mu}
 =\iota(b)^{-1}
   \frac{\exp_\varphi'(\lambda_0/\iota(b))}
        {\exp_\varphi(\lambda_0/\iota(b))}
 =\frac1{\iota(b)x}.
\]
The numerator on this coset is $\varphi_c(x)^{q^k}$.
The map $\lambda\mapsto\exp_\varphi(\lambda/\iota(b))$
identifies $\Lambda_\varphi/\iota(b)\Lambda_\varphi$ with $\varphi[b]$.
Summing proves \eqref{eq:torsion-trace} after passing to a splitting
field. Thus the sum is the trace of the indicated element of
$\mathscr T_b^\varphi$ over $H$.
\end{proof}

\begin{theorem}\label{thm:projective-pairing}
Let $\Lambda_\varphi=\tilde{\pi}\mathfrak j$, where $\mathfrak j\supset A$
is a fractional ideal and $0\ne\tilde{\pi}\in\Lambda_\varphi$. Then
\begin{equation}\label{eq:projective-pairing}
  C_\varphi=-\tilde{\pi}^{-1}\Lcal_{\mathfrak j}(1) \Hcal_{\varphi}(\tilde{\pi}). 
\end{equation}
\end{theorem}

\begin{proof}
Write $\Rcal^{[k]}=\Rcal_\varphi^{[k]}$ and
$\Hcal=\Hcal_\varphi(\tilde{\pi})$.
The equation $\Hcal^{(1)}=f_\varphi\Hcal$ gives
\begin{equation}\label{eq:pairing-product}
\Rcal^{[k]}=F_\varphi^{[k]}\Rcal^{[0]}.
\end{equation}
We identify $\Rcal^{[0]}$ with $C_\varphi$ by its divisor and residue.

\emph{Rationality and initial pole bounds.}
For $k>0$, the negative-index identity of
\cite[Proposition 5.20]{Ferraro25}, in the form
\cite[Lemma 8.2]{HuCauchy}, applies to this rank-one lattice:
\[
\sum_{0\ne\lambda\in\Lambda_\varphi}
\frac{\exp_\varphi(\iota(u)\lambda)}{\lambda^{q^{-k}}}=0
\qquad\text{if }|\iota(u)|\le q^{-(k+1)}.
\]
Raising to $q^k$ and using \eqref{eq:lattice-sum} shows that
the evaluation of $\Rcal^{[k]}$ vanishes
on $\mathfrak m_\infty^{m_k}$, where
$m_k=\lceil(k+1)/\Ninf\rceil$.
The functional
$u\mapsto\langle u,\Rcal^{[k]}\rangle _\infty$ therefore factors through
\[
Q_k=K_\infty/(A+\mathfrak m_\infty^{m_k}),
\]
which is finite because $K_\infty/A$ is compact and the image of
$\mathfrak m_\infty^{m_k}$ is open. Residue duality identifies the
$\Fq$-dual $Q_k^*$ with $H^0(X,\Omega_X^1(m_k\infty))$, as in the proof of
\cite[Theorem 8.8(i)]{HuCauchy}; the passage from the
vanishing of pairings to the pole order is
Remark~\ref{rem:local-pairing-criterion}.

The finite-dimensionality of $Q_k$ gives
$\operatorname{Hom}_{\Fq}(Q_k,\Cinf)=\Cinf\otimes_{\Fq}Q_k^*$,
so
\begin{equation}\label{eq:positive-bound}
\Rcal^{[k]}\in\Cinf\otimes_{\Fq}
H^0(X,\Omega_X^1(m_k\infty)).
\end{equation}
Choose an $\Fq$-basis of $H^0(X,\Omega_X^1(m_k\infty))$ and
residue-dual classes $c_j/b$ as in Lemma~\ref{lem:dual-denominator}.
Equation~\eqref{eq:torsion-trace} expresses each coefficient of
$\Rcal^{[k]}$ in this basis as a finite trace over $H$, so
$\Rcal^{[k]}$ has coefficients in $H$.

To continue $\Rcal^{[0]}$ rationally, take $a\in A\setminus\Fq$,
put $n=-v_\infty(a)$, and write
$\varphi_a=\iota(a)+\sum_{i=1}^{n\Ninf}\alpha_i\tau^i$.
The motive expansion and \eqref{eq:pairing-product} give
\begin{equation}\label{eq:clearing}
B_a:=(\chi(a)-\iota(a))\Rcal^{[0]}
     =\sum_{i=1}^{n\Ninf}\alpha_i\Rcal^{[i]}
     \in H\otimes_{\Fq}H^0(X,\Omega_X^1((n+1)\infty)).
\end{equation}
Since $|\iota(a)|>1$, the factor $\chi(a)-\iota(a)$ is invertible
on the Tate domain. Thus $B_a/(\chi(a)-\iota(a))$ defines a rational
continuation of $\Rcal^{[0]}$, independent of $a$ by
\eqref{eq:clearing}. The factors $\chi(a)-\iota(a)$ generate the ideal of $\Xi$.
At each affine point other than $\Xi$, at least one of the
functions $\chi(a)-\iota(a)$ is a local unit. At $\Xi$, one of these
functions generates the maximal ideal, since $X_H$ is smooth. Hence $\Rcal^{[0]}$ has no
affine pole except a possible simple pole at $\Xi$.
At every point above $\infty$, $\chi(a)-\iota(a)$ has
pole order $n$, while $B_a$ has pole order at most $n+1$.
Thus $\Rcal^{[0]}$ has at most simple poles at infinity.

\emph{The precise divisor.}
The Frobenius basis of the motive and the shtuka divisor imply
\[
\sum_{k\ge1}\Cinf F_\varphi^{[k]}
 =f_\varphi\Gamma(U_{\Cinf},\OO(V_\varphi^{(1)}))
 =\Gamma(U_{\Cinf},\OO(V_\varphi-\Xi)).
\]
By \eqref{eq:pairing-product} and \eqref{eq:positive-bound},
multiplication by $\Rcal^{[0]}$ sends
$\Gamma(U_{\Cinf},\OO(V_\varphi-\Xi))$ to regular differentials
on $U_{\Cinf}$.
Localization at an affine point therefore yields
$\divisor(\Rcal^{[0]})|_{U_{\Cinf}}
\ge V_\varphi|_{U_{\Cinf}}-\Xi$.

For $\Ninf>1$, let $p_j=q_\varphi^{(j)}$.
At $p_j$ with $0\le j\le\Ninf-2$, choose $k=j+1$.
In \eqref{eq:product-divisor}, the contribution from
$V_\varphi^{(k)}$ is zero because $p_j^{(-k)}=q_\varphi^{(-1)}$
avoids $V_\varphi$. The infinity terms are distinct, so
\[
\operatorname{ord}_{p_j}F_\varphi^{[k]}
 =-\operatorname{mult}_{p_j}V_\varphi-1.
\]
Here $1\le k\le\Ninf-1$, hence $m_k=1$ in
\eqref{eq:positive-bound}. It follows that
$\operatorname{ord}_{p_j}\Rcal^{[0]}\ge\operatorname{mult}_{p_j}V_\varphi$.
At $q_\varphi^{(-1)}$ the initial simple-pole bound suffices,
since this point avoids $V_\varphi$. It is the only infinity point
when $\Ninf=1$. We have proved that $\Rcal^{[0]}$ lies in
\eqref{eq:cauchy-line}.

\emph{Characteristic residue.}
Putting $c=1$, $b=a$, and $k=0$ in \eqref{eq:torsion-trace} gives
$\langle1/a,\Rcal^{[0]}\rangle _\infty=1/\iota(a)$:
the trace of one is $q^{\deg(aA)}-1=-1$ in $H$.
Also $\langle1,\Rcal^{[0]}\rangle _\infty=0$.
Multiplication by $\chi(a)$ is dual to multiplication by $a$,
so \eqref{eq:clearing} yields $\langle1/a,B_a\rangle _\infty=-1$.
Since $B_a$ is a rational differential regular on $U$, the residue
pairing gives
\[
 -1=\sum_{p\mid\infty}\res_p(B_a/\chi(a))
    =\sum_{p\mid\infty}\res_p(\Rcal^{[0]}).
\]
For the second equality, use
$B_a/\chi(a)=\Rcal^{[0]}-\iota(a)\Rcal^{[0]}/\chi(a)$;
the last term is regular at infinity because $n\ge1$.
The residue-field trace becomes the sum over the geometric points
above $\infty$. The global residue theorem now gives
$\Res_\Xi\Rcal^{[0]}=1$.
Uniqueness in \eqref{eq:cauchy-line} implies
$\Rcal^{[0]}=C_\varphi$, proving \eqref{eq:projective-pairing}.
\end{proof}

In particular, \eqref{eq:pairing-product} gives
$\Rcal_\varphi^{[1]}=f_\varphi C_\varphi
=\widetilde\lambda_\varphi^{(1)}$.

\subsection{An effective finite-trace algorithm}
Proposition~\ref{prop:torsion-trace} and the denominator bound in
Lemma~\ref{lem:dual-denominator} reconstruct $C_\varphi$ from its
torsion pairings; compare \cite[Theorem 8.8]{HuCauchy}.
The required basis of differentials can be obtained from a
Riemann--Roch computation, for example by the methods of \cite{Hess02}.

Choose $a\in A\setminus\Fq$, put $n=-v_\infty(a)$, and choose
an $\Fq$-basis $\nu_1,\ldots,\nu_d$ of
$H^0(X,\Omega_X^1((n+1)\infty))$.
Choose $b\in A\setminus\{0\}$ with
$\deg(bA)>2g-2+(n+1)\Ninf$.
Lemma~\ref{lem:dual-denominator} supplies $c_j\in A$ such that
$\langle c_i/b,\nu_j\rangle _\infty=\delta_{ij}$.

\begin{corollary}\label{cor:finite-algorithm}
With these choices,
\begin{equation}\label{eq:finite-algorithm}
 {C_\varphi=
 -\frac{\displaystyle\sum_{j=1}^d
\Tr_{\mathscr T_b^\varphi/H}
   \left(\frac{\varphi_{ac_j}(Z)-\iota(a)\varphi_{c_j}(Z)}Z\right)\nu_j}
 {\iota(b)(\chi(a)-\iota(a))}.}
\end{equation}
The resulting differential is invariant under scalar isomorphisms of $\varphi$.
\end{corollary}
\begin{proof}
Write $B_a=\sum_j\beta_j\nu_j$. By the choice of the dual classes,
\begin{align*}
\beta_j
 &=\langle c_j/b,B_a\rangle _\infty\\
 &=\langle ac_j/b,\Rcal_\varphi^{[0]}\rangle _\infty
   -\iota(a)\langle c_j/b,\Rcal_\varphi^{[0]}\rangle _\infty\\
 &=-\iota(b)^{-1}\Tr_{\mathscr T_b^\varphi/H}
       \left(\frac{\varphi_{ac_j}(Z)-\iota(a)\varphi_{c_j}(Z)}{Z}\right).
\end{align*}
Divide by $\chi(a)-\iota(a)$ in \eqref{eq:clearing} to obtain
\eqref{eq:finite-algorithm}. For $\varphi^z=z\varphi z^{-1}$,
extend the coefficient field to contain $z$. The substitution $Z=zW$
identifies the torsion algebra of $\varphi^z$ with that of $\varphi$ and gives
\[
\frac{\varphi^z_{ac_j}(zW)-\iota(a)\varphi^z_{c_j}(zW)}{zW}
 =\frac{\varphi_{ac_j}(W)-\iota(a)\varphi_{c_j}(W)}{W}.
\]
Finite traces commute with scalar extension; hence
$C_{\varphi^z}=C_\varphi$.
\end{proof}

The expression in \eqref{eq:finite-algorithm} computes $C_\varphi$
independently of the choices of $a$, $b$, the basis, and the dual classes.
The computation requires neither a torsion splitting field nor a period generator.

 Finite traces commute with extension of the coefficient field,
so the construction also determines the coefficient Galois action
on $C_\varphi$.

\begin{lemma}\label{lem:rational-galois-covariance}
For an $\iota(K)$-embedding $\sigma:H\hookrightarrow\Cinf$,
\begin{equation}\label{eq:rational-galois-covariance}
\sigma(C_\varphi)=C_{\varphi^\sigma},\qquad
\sigma(f_\varphi)=f_{\varphi^\sigma},\qquad
\sigma(r_\Ia^\varphi)=r_\Ia^{\varphi^\sigma}.
\end{equation}
\end{lemma}
\begin{proof}
For $P\in\mathscr T_b^\varphi$, compatibility of finite
traces with extension of scalars gives
\[
\sigma\!\left(\Tr_{\mathscr T_b^\varphi/H}(P)\right)
 =\Tr_{\mathscr T_b^{\varphi^\sigma}/\sigma(H)}(P^\sigma).
\]
In \eqref{eq:finite-algorithm}, this identity gives covariance of
$C_\varphi$; in \eqref{eq:torsion-trace}, it gives covariance of
$\Rcal_\varphi^{[1]}$ after choosing a residue-dual basis in
\eqref{eq:positive-bound}. Taking the quotient
$f_\varphi=\Rcal_\varphi^{[1]}/C_\varphi$ shows that
$f_\varphi\in H(X)$ and $f_{\varphi^\sigma}=\sigma(f_\varphi)$.
Finally, $(\varphi_\Ia)^\sigma=(\varphi^\sigma)_\Ia$ by uniqueness
of the monic ideal polynomial. Applying
\eqref{eq:ideal-coordinate} to these polynomials gives covariance
of $u_\Ia^\varphi$. Their $\tau^0$ coefficients give covariance
of $d_\Ia^\varphi$, hence of
$r_\Ia^\varphi=u_\Ia^\varphi/d_\Ia^\varphi$.
\end{proof}
For sign-normalized modules, the middle identity of
\eqref{eq:rational-galois-covariance} is also
\cite[Corollary 3.2]{ANDTR17}.

Finite traces can be evaluated by multiplication matrices in
$\mathscr T_b^\varphi$. The following recursion computes the required
power traces directly from the coefficients of $\varphi_b$.

\begin{remark}\label{rem:trace-recursion}
The traces entering \eqref{eq:torsion-trace} and the
finite algorithm \eqref{eq:finite-algorithm} are linear combinations
of the scalars $s_i=\Tr(Z^{q^i-1})$, where
$\Tr=\Tr_{\mathscr T_b^\varphi/H}$: for $\varphi_c=\sum_i
c_i\tau^i$ one has
$\varphi_c(Z)^{q^k}/Z=\sum_i c_i^{q^k}Z^{q^{i+k}-1}$. We compute the
$s_i$ from the coefficients of $\varphi_b(Z)=\sum_{i=0}^D B_iZ^{q^i}$
alone; neither a splitting field of the torsion polynomial nor a
multiplication matrix of $\mathscr T_b^\varphi$ is needed. The
initial values are
\begin{equation}\label{eq:trace-recursion}
s_0=-1,\qquad s_i=0\quad(0<i<D),
\end{equation}
and, for $i\ge D$, the recurrence is
\begin{equation}\label{eq:trace-recursion2}
s_i=-\sum_{j=0}^{D-1}(B_j/B_D)^{q^{i-D}}s_{i-D+j}.
\end{equation}
Indeed, the recurrence follows by raising
$\varphi_b(x)=0$ to $q^{i-D}$, dividing by $x$, and summing:
the sum over the nonzero roots $x\in\varphi[b]$ is the trace in the
finite \'etale algebra $\mathscr T_b^\varphi$, and the isolated
leading term is $B_D^{q^{i-D}}s_i$.
For the initial values, put
\[
P_b(T)=B_D^{-1}T^{q^D}\varphi_b(T^{-1})
       =\prod_{0\ne x\in\varphi[b]}(1-xT).
\]
Then $-TP_b'(T)/P_b(T)=\sum_{m\ge1}\Tr(Z^m)T^m$, Newton's
identity for the nonzero torsion points. All exponents of $P_b$ except
$q^D-1$ are divisible by $p$,
so the first possible nonzero term of $-TP_b'(T)/P_b(T)$
has degree $q^D-1$.
This proves $s_i=0$ for $0<i<D$. Also, $s_0=\Tr(1)$
is the image in $H$ of $\dim_H\mathscr T_b^\varphi=q^D-1$,
and hence equals $-1$.
Thus the coefficients of $\varphi_b$ and of the action
polynomials in the trace input suffice to evaluate
\eqref{eq:torsion-trace} and \eqref{eq:finite-algorithm}.
\end{remark}

\subsection{Galois trace formula}\label{sec:galois-trace}
The inverse ideals in \eqref{eq:class-subseries} are,
up to scaling, the period lattices of the ideal-isogeny targets
of a module with free period lattice.
Let $\psi$ be a rank-one module with free period lattice
$\Lambda_\psi=A\tilde{\pi}$. For the ideal representatives
of Proposition~\ref{prop:class-decomposition}, put
$\varphi_i=\Ia_i*\psi$ and $d_i=\partial\psi_{\Ia_i}$.
Equations~\eqref{eq:isogeny-uniformization} and
\eqref{eq:period-isogeny} give
\[
\Lambda_{\varphi_i}=d_i\tilde{\pi}\Ia_i^{-1},\qquad
\Hcal_{\varphi_i}(d_i\tilde{\pi})
 =\psi_{\Ia_i}(\Hcal_\psi(\tilde{\pi}))
 =d_i r_{\Ia_i}^\psi\Hcal_\psi(\tilde{\pi}).
\]

\begin{theorem}\label{thm:principal-class}
For every $i$,
\begin{equation}\label{eq:individual-class}
\LL_{\psi,[\Ia_i]}(1)\Hcal_\psi(\tilde{\pi})
 =\tilde{\pi}C_{\varphi_i}.
\end{equation}
\end{theorem}
\begin{proof}
Apply Theorem~\ref{thm:projective-pairing} to $\varphi_i$ with
$\mathfrak j=\Ia_i^{-1}$ and period $d_i\tilde{\pi}$. Substitution gives
\[
C_{\varphi_i}
 =-(d_i\tilde{\pi})^{-1}\Lcal_{\mathfrak j_i}(1)
       \Hcal_{\varphi_i}(d_i\tilde{\pi})
 =-\tilde{\pi}^{-1}r_{\Ia_i}^\psi\Lcal_{\mathfrak j_i}(1)
       \Hcal_\psi(\tilde{\pi}).
\]
Equation~\eqref{eq:class-subseries} now gives \eqref{eq:individual-class}.
\end{proof}

Summing \eqref{eq:individual-class} over the ideal classes gives
\begin{equation}\label{eq:principal-sum}
 {\LL_\psi(1)\Hcal_\psi(\tilde{\pi})
 =\tilde{\pi}\sum_{i=1}^h C_{\varphi_i}
 =\tilde{\pi}\sum_{i=1}^h
                  \frac{\widetilde\lambda_{\varphi_i}^{(1)}}{f_{\varphi_i}}.}
\end{equation}
Each differential in the sum is given by \eqref{eq:finite-algorithm}
for the target action $\varphi_i$.

Let $H_A^+$ be the sign class field for a fixed sign function and
$H_A$ the ordinary Hilbert class field, where $\infty$ splits completely.
The Hayes correspondence \cite[Theorems 14.7 and 15.6]{Hayes92} identifies
\[
\operatorname{Gal}(H_A^+/K)\simeq\Cl^+(A),\qquad
\operatorname{Gal}(H_A/K)\simeq\Cl(A).
\]
The kernel of restriction is
\[
G_0:=\operatorname{Gal}(H_A^+/H_A)
\simeq\mathbb F_\infty^\times/\Fq^\times.
\]
Thus $|G_0|=(q^{\Ninf}-1)/(q-1)$.
The group $G_0$ acts on sign-normalized rank-one modules
by scalar isomorphisms
\cite[Corollary 14.9 and Proposition 15.4]{Hayes92}; see also
\cite[Section 2.1 and Lemma 2.1]{ANDTR17}.
Following \cite[Sections 2--3]{ANDTR17}, we use the Artin convention
$\varphi^\Ia=\varphi^{\sigma_\Ia}$ for sign-normalized Hayes modules.
The principal lattice class is represented
by a standard module in the terminology of \cite[Lemma 2.2]{ANDTR17}.

For the coefficient action fixed in Section~\ref{sec:ideal-series},
\[
 (H_A^+(X))^{G_0}=H_A(X),\qquad
\bigl(\Omega^1_{H_A^+(X)/H_A^+}\bigr)^{G_0}
     =\Omega^1_{H_A(X)/H_A}.
\]
The trace of a rational differential is the sum under this action;
$\Tr_{H_A/K}$ has exactly $h$ summands.

\begin{corollary}\label{cor:hayes-trace}
Let $\psi$ be a rank-one Drinfeld $A$-module with free period
lattice $\Lambda_\psi=A\tilde{\pi}$. Then
\begin{enumerate}
\item[(i)] the coefficient fields satisfy
\begin{equation}\label{eq:hayes-fields}
C_\psi\in\Omega^1_{H_A(X)/H_A},\qquad
r_\Ia^\psi\in H_A(X)\quad(0\ne\Ia\subset A);
\end{equation}
\item[(ii)] the weight-one value is the ordinary trace
\begin{equation}\label{eq:hayes-trace}
\LL_\psi(1)\frac{\Hcal_\psi(\tilde{\pi})}{\tilde{\pi}}
       =\Tr_{H_A/K}(C_\psi).
\end{equation}
\end{enumerate} 
\end{corollary}
\begin{proof}
Choose a sign-normalized scalar model $\psi_0$ and
$v\in\Cinf^\times$ with $\psi=v\psi_0v^{-1}$, and put
$\tilde\pi_0=v^{-1}\tilde\pi$. Scalar invariance
(Lemma~\ref{lem:rescaling} and Corollary~\ref{cor:finite-algorithm})
and \eqref{eq:period-scalar} give
\[
r_\Ia^\psi=r_\Ia^{\psi_0},\qquad C_\psi=C_{\psi_0},\qquad
\frac{\Hcal_\psi(\tilde\pi)}{\tilde\pi}
 =\frac{\Hcal_{\psi_0}(\tilde\pi_0)}{\tilde\pi_0}.
\]
Thus it suffices to prove the assertions for a sign-normalized $\psi$,
which we assume for the remainder of the proof.

\emph{(i).} \eqref{eq:ideal-coordinate} and \eqref{eq:finite-algorithm}
give $r_\Ia^\psi\in H_A^+(X)$ and
$C_\psi\in\Omega^1_{H_A^+(X)/H_A^+}$.
Their Galois covariance is Lemma~\ref{lem:rational-galois-covariance}.
Take $\sigma\in G_0$ and represent it by a principal ideal $(a)$,
with $a\in A\setminus\{0\}$. Since
$\psi_{(a)}=\LT_\psi(a)^{-1}\psi_a$, the ideal-intertwining identity gives
\begin{equation}\label{eq:sign-kernel-conjugation}
\psi^\sigma=z_a\psi z_a^{-1},\qquad
z_a=\LT_\psi(a)^{-1}
    =\iota(\operatorname{sgn}(a))^{-q^{\xi_\psi}}
       \in\iota(\mathbb F_\infty^\times).
\end{equation}
Corollary~\ref{cor:finite-algorithm} and Lemma~\ref{lem:rescaling} now imply
\begin{equation}\label{eq:sign-kernel-invariance}
\sigma(C_\psi)=C_{z_a\psi z_a^{-1}}=C_\psi,\qquad
\sigma(r_\Ia^\psi)=r_\Ia^{z_a\psi z_a^{-1}}=r_\Ia^\psi.
\end{equation}
The fixed-field identities above prove \eqref{eq:hayes-fields}.

\emph{(ii).} $C_{\varphi_i}=\sigma_{\Ia_i}(C_\psi)$.
By \eqref{eq:sign-kernel-invariance}, this conjugate differential
depends only on the restriction of $\sigma_{\Ia_i}$ to $H_A$. The ordinary ideal classes
enumerate $\operatorname{Gal}(H_A/K)$, so \eqref{eq:principal-sum}
gives \eqref{eq:hayes-trace}.
\end{proof}

For the sign-normalized model $\psi_0$, the shtuka function satisfies
\[
\sigma_{(a)}(f_{\psi_0})
 =\LT_{\psi_0}(a)^{1-q}f_{\psi_0}.
\]
At higher-degree infinity this character is nontrivial; compare
\cite[proof of Lemma 3.3]{ANDTR17}.
Thus $f_{\psi_0}\notin H_A(X)$ when $\Ninf>1$.
This field assertion concerns the sign-normalized model;
$C_\psi$ and $r_\Ia^\psi$ descend to $H_A$ for every scalar model.

Combining $C_\psi=-\Lcal_A(1)\Hcal_\psi(\tilde{\pi})/\tilde{\pi}$
from Theorem~\ref{thm:projective-pairing} with Corollary~\ref{cor:hayes-trace}
gives the following direct relation.
\begin{corollary}\label{cor:trace-formula}
Let $\psi$ be a rank-one Drinfeld $A$-module with free period lattice $\Lambda_\psi=A\tilde{\pi}$. Then
\begin{equation}\label{eq:type-trace-one}
\LL_\psi(1)\frac{\Hcal_\psi(\tilde{\pi})}{\tilde{\pi}}
 = - \Tr_{H_A/K}\left(
       \Lcal_A(1)\frac{\Hcal_\psi(\tilde{\pi})}{\tilde{\pi}}\right).
\end{equation}
\end{corollary}
 For sign-normalized $\psi$ at rational infinity,
Corollary~\ref{cor:type-trace-higher} extends this identity to weights
$n\ge1$ with $n\equiv1\pmod{q-1}$. 

At rational infinity, the scalar form of \eqref{eq:hayes-trace}
is the counterpart of \cite[Theorem 8.9]{Ferraro24} for a free
period lattice. The differential formulation specifies each summand
by its Cauchy divisor and residue and applies to infinite points of
arbitrary degree.

\begin{remark}
 For the rational differentials in \eqref{eq:hayes-trace},
the ordinary trace agrees with the trace from the sign class
field. Indeed, for $D\in\Omega^1_{H_A(X)/H_A}$, 
\begin{equation}\label{eq:sign-trace-equality}
\Tr_{H_A^+/K}(D)
 =\frac{q^{\Ninf}-1}{q-1}\Tr_{H_A/K}(D)
 =\Tr_{H_A/K}(D),
\end{equation}
because $1+q+\cdots+q^{\Ninf-1}=1$ in the coefficient field.
Each summand of the ordinary trace occurs
$(q^{\Ninf}-1)/(q-1)$ times in the sign trace.
A further finite extension of the trace
field multiplies the trace by its relative degree.
\end{remark}

\subsection{Isogenies to a free period lattice}\label{sec:general-modules}

An isogeny to a free target transports the preceding trace
formula to an arbitrary rank-one source. Let $\varphi$ be a
rank-one module and choose a separable isogeny datum with target
$\psi$ whose period lattice is free:
\begin{equation}\label{eq:source-isogeny}
J=[\eta_J;h_J;\Ib,\mu_J]:\varphi\longrightarrow\psi,
\qquad \Lambda_\psi=A\tilde\pi.
\end{equation}
Every $\varphi$ admits such a datum. Indeed, write
$\Lambda_\varphi=\iota(\Ib)z$ with $z\in\Cinf^\times$
and $\Ib\subset A$ a nonzero integral ideal; the
inclusion $\iota(\Ib)z\subset Az$ gives an isogeny to the module
with lattice $Az$. Put $ \partial\eta_J=h_J(\Xi)$
and $z_J=\tilde\pi/\partial\eta_J$. The normalized motive multiplier is
\begin{equation}\label{eq:isogeny-normalization}
g_J=\frac{h_J}{\partial\eta_J}=r_\Ib^\varphi.
\end{equation}
The lattice relation \eqref{eq:isogeny-uniformization} becomes
\begin{equation}\label{eq:adapted-parameter-lattice}
\Lambda_\varphi=\iota(\Ib)z_J,\qquad
\iota(\Ib^{-1})\Lambda_\varphi=Az_J.
\end{equation}
Thus $z_J$ generates the enlarged lattice $\iota(\Ib^{-1})\Lambda_\varphi$.
The source periods are exactly $\iota(a)z_J$ with $a\in\Ib$.

The isogeny-adapted generating differential of
\cite[Notation 6.6 and Proposition 6.13]{HuCauchy} is defined on
the enlarged lattice $Az_J$ by
\begin{equation}\label{eq:adapted-generating-differential}
\Gcal_J(\iota(a)z_J)
 =\frac{\Hcal_\psi(\iota(a)\tilde\pi)}{h_J}
 =\chi(a)\frac{\Hcal_\psi(\tilde\pi)}{h_J}
\qquad(a\in A).
\end{equation}
Each $\Gcal_J(\iota(a)z_J)$ lies in
$\operatorname{Frac}(\mathbb T_A)\otimes_A\Omega_A$ and has residue
$\iota(a)z_J$ at $\Xi$. For $a\in\Ib$, period functoriality gives
\begin{equation}\label{eq:adapted-source-period}
\Hcal_\varphi(\iota(a)z_J)
 =\chi(a)\Gcal_J(z_J).
\end{equation}
Indeed, applying $\eta_J$ to the left side gives
$\Hcal_\psi(\iota(a)\tilde\pi)$ by \eqref{eq:period-isogeny},
and on $\sfm(\varphi)\otimes_A\Omega_A$ the action of
$\eta_J$ is multiplication by $h_J$.
In particular, $\Gcal_J(z_J)$ has residue $z_J$ even when $z_J$
is not a period of $\varphi$.

\begin{proposition}\label{prop:divisor-shift}
Put
$D_J=V_\psi-q_\psi^{(-1)}-\divisor(h_J)$ and
$\mathcal L_J=\OO(D_J)$. This line bundle is acyclic, and its
normalized Cauchy coordinate satisfies
\begin{equation}\label{eq:divisor-shift}
K_J(P,Q)=\frac{h_J(P)}{h_J(Q)}K_\psi(P,Q),\qquad
K_J(\Xi,\cdot)=\frac{C_\psi}{g_J},\qquad
\Res_\Xi K_J(\Xi,\cdot)=1.
\end{equation}
\end{proposition}
\begin{proof}
Multiplication by $h_J^{-1}$ identifies $\mathcal L_J$ with
$\OO(V_\psi-q_\psi^{(-1)})$, which is acyclic.
The factor $h_J(P)/h_J(Q)$ changes the Cauchy divisor condition
from $V_\psi-q_\psi^{(-1)}$ to $D_J$ and is one on the diagonal.
Uniqueness of the residue-normalized coordinate gives the first identity;
specializing $P=\Xi$ gives the other two. This is the transport
formula of \cite[equation (128)]{HuCauchy}.
\end{proof}

\begin{theorem}\label{thm:general-value}
Let $J:\varphi\to\psi$ be as in \eqref{eq:source-isogeny}, with
the quantities in \eqref{eq:isogeny-normalization}. Then
\begin{enumerate}
\item[(i)] $g_J\in H_A(X)$ and
$K_J(\Xi,\cdot)\in\Omega^1_{H_A(X)/H_A}$;
\item[(ii)] the ideal value is the ordinary trace
\begin{equation}\label{eq:general-trace}
 \LL_\varphi(1)\Gcal_J(z_J)
       =z_J\Tr_{H_A/K}\bigl(K_J(\Xi,\cdot)\bigr); 
\end{equation}
\item[(iii)] equivalently, in terms of the generating differential at
$\tilde\pi\in\Lambda_\psi$,
\begin{equation}\label{eq:general-value}
\LL_\varphi(1)\Hcal_\psi(\tilde\pi)
 =\tilde\pi\,g_J\Tr_{H_A/K}\left(\frac{C_\psi}{g_J}\right).
\end{equation}
\end{enumerate}
\end{theorem}
\begin{proof}
\emph{(i).} A reverse isogeny relates $g_J$ to an ideal
ratio of the free target.
Choose $0\ne a\in\Ib$ and put $\mathfrak d=(a)\Ib^{-1}\subset A$.
Since $\eta_J$ kills $\varphi[\Ib]$, the action polynomial $\varphi_a$
factors as
\[
\varphi_a=\eta_a^\vee\eta_J,\qquad
\eta_a^\vee:\psi\longrightarrow\varphi.
\]
The lattice relation \eqref{eq:adapted-parameter-lattice} shows that
the annihilator ideal of $\ker\eta_a^\vee$ is $\mathfrak d$.
The isogeny $\eta_a^\vee$ has derivative $\iota(a)/\partial\eta_J$
and motive multiplier $\chi(a)/h_J$. Hence
\begin{equation}\label{eq:reverse-normalized-multiplier}
r_\mathfrak d^\psi
 =\frac{\chi(a)}{\iota(a)g_J}.
\end{equation}
Corollary~\ref{cor:hayes-trace} gives
$r_\mathfrak d^\psi\in H_A(X)$ and
$C_\psi\in\Omega^1_{H_A(X)/H_A}$.
Equation~\eqref{eq:reverse-normalized-multiplier} therefore gives
$g_J\in H_A(X)$ and $C_\psi/g_J\in\Omega^1_{H_A(X)/H_A}$.

\emph{(ii).} The target $\varphi$ of $\eta_a^\vee$ is
scalar-isomorphic to $\mathfrak d*\psi$.
Choose a sign-normalized scalar model
$\psi_0$ of $\psi$. Scalar invariance transfers the normalized
ideal cocycle of \cite[Lemma 3.1]{ANDTR17} from $\psi_0$ to $\psi$:
\begin{equation}\label{eq:ideal-cocycle}
r_{\Ia\mathfrak d}^\psi
 =r_\Ia^\psi(r_\mathfrak d^\psi)^{\sigma_\Ia}
 =r_\mathfrak d^\psi(r_\Ia^\psi)^{\sigma_\mathfrak d}.
\end{equation}
Scalar invariance and coefficient covariance
(Lemma~\ref{lem:rational-galois-covariance}) therefore give
\begin{equation}\label{eq:induced-multiplier}
\frac{r_\Ia^\varphi}{r_\Ia^\psi}
 =\frac{(r_\mathfrak d^\psi)^{\sigma_\Ia}}{r_\mathfrak d^\psi}
 =\frac{g_J}{g_J^{\sigma_\Ia}}.
\end{equation}
Here $\chi(a)/\iota(a)$ is fixed by the coefficient action, and all
Galois actions in these identities are on rational functions.

Let $\varphi_i=\Ia_i*\psi$. The class decomposition
\eqref{eq:class-subseries} and Theorem~\ref{thm:principal-class} give
\[
\LL_{\varphi,[\Ia_i]}(1)\Hcal_\psi(\tilde\pi)
 =\tilde\pi\frac{r_{\Ia_i}^\varphi}{r_{\Ia_i}^\psi}C_{\varphi_i}
 =\tilde\pi\,g_J
       \left(\frac{C_\psi}{g_J}\right)^{\sigma_{\Ia_i}}.
\]
The last equality uses \eqref{eq:induced-multiplier} and
$C_{\varphi_i}=C_\psi^{\sigma_{\Ia_i}}$. To obtain the latter
identity, apply the Hayes correspondence and
Lemma~\ref{lem:rational-galois-covariance} to $\psi_0$.
The ideal-isogeny targets of $\psi$ and $\psi_0$ are
scalar-isomorphic by Lemma~\ref{lem:rescaling}; hence the scalar
invariance in Corollary~\ref{cor:finite-algorithm} transfers the
identity to $\psi$.
Dividing by $h_J=\partial\eta_J g_J$ proves the individual class identity
\begin{equation}\label{eq:general-class}
\LL_{\varphi,[\Ia_i]}(1)\Gcal_J(z_J)
 =z_J K_J(\Xi,\cdot)^{\sigma_{\Ia_i}}.
\end{equation}
The ordinary classes enumerate $\operatorname{Gal}(H_A/K)$;
summing proves \eqref{eq:general-trace}.

\emph{(iii).}
Multiplication by $h_J$ gives \eqref{eq:general-value}.
\end{proof}

For $J=\mathrm{id}_\psi$, one has $\partial\eta_J=g_J=1$, $z_J=\tilde\pi$,
$\Gcal_J(z_J)=\Hcal_\psi(\tilde\pi)$ and $K_J(\Xi,\cdot)=C_\psi$;
the theorem recovers \eqref{eq:hayes-trace}.
For a general datum, multiplication of \eqref{eq:general-trace}
by $\chi(a)$ gives the formulation at any nonzero source period:
\begin{equation}\label{eq:general-source-period}
\LL_\varphi(1)\Hcal_\varphi(\iota(a)z_J)
 =z_J\chi(a)\Tr_{H_A/K}\bigl(K_J(\Xi,\cdot)\bigr)
\qquad(0\ne a\in\Ib).
\end{equation}
After division by $\Hcal_\varphi(\iota(a)z_J)$,
\eqref{eq:general-source-period} gives the same function $\LL_\varphi(1)$
for every $0\ne a\in\Ib$.
For $v\in\Cinf^\times$, compose $\eta_J$ with the scalar
isomorphism $\psi\to v\psi v^{-1}$ and choose $v\tilde\pi$ as the
new target period generator. Then $h_J,\partial\eta_J,\tilde\pi$, and
$\Hcal_\psi(\tilde\pi)$ are multiplied by $v$, so
$g_J,z_J,K_J$, and $\Gcal_J(z_J)$ are unchanged.
 Thus the transported formula is invariant under scalar
isomorphisms of the target.

At rational infinity, \cite[Remark 8.10]{Ferraro24} gives the
characteristic value of the Anderson zeta function. Here specialization
is also taken after meromorphic continuation: for $a\in A\setminus\Fq$,
$\chi(a)-\iota(a)$ is a Tate unit but vanishes at $\Xi$.
\begin{corollary}\label{cor:class-number}
For any rank-one module $\varphi$ and every ordinary ideal
class $[\Ia_i]$, equation~\eqref{eq:general-class} gives a meromorphic
continuation of $\LL_{\varphi,[\Ia_i]}(1)$ that is regular at $\Xi$
with value one. Consequently,
\begin{equation}\label{eq:class-number}
\LL_\varphi(1)|_\Xi=h(A)\cdot1_{\Cinf}.
\end{equation}
\end{corollary}
\begin{proof}
For a local parameter $t$ at $\Xi$, the residue normalizations give
\[
K_J(\Xi,\cdot)^{\sigma_{\Ia_i}}=(t^{-1}+O(1))\,dt,
\qquad \Gcal_J(z_J)=(z_Jt^{-1}+O(1))\,dt.
\]
Dividing $z_JK_J(\Xi,\cdot)^{\sigma_{\Ia_i}}$ by
$\Gcal_J(z_J)$ in \eqref{eq:general-class} therefore gives
$\LL_{\varphi,[\Ia_i]}(1)=1+O(t)$.
Summing over the finitely many classes gives \eqref{eq:class-number},
including when the characteristic divides $h(A)$.
\end{proof}

\section{Generating functions and higher-weight traces}\label{sec:general-value}

The Cauchy differential is the constant term of a meromorphic
lattice sum. The generating identity below identifies this sum
with $\Hcal_\varphi/\exp_\varphi$.
For a sign-normalized module with free period lattice, the
Galois characters of its Taylor coefficients determine the
higher-weight Hilbert-class traces and their character twists.

\subsection{Lattice generating series}\label{sec:higher-weights}
Let $\varphi$ be an arbitrary rank-one module.
\begin{notation}\label{not:higher-lattice-differential}
Choose $0\ne\tilde\pi\in\Lambda_\varphi$ and put
$\mathfrak j=\tilde\pi^{-1}\Lambda_\varphi\supset A$. For $n\ge1$, set
\begin{equation}\label{eq:higher-lattice-differential}
C_{\varphi,n}
 =-\tilde{\pi}^{-n}\Lcal_{\mathfrak j}(n)
           \Hcal_\varphi(\tilde{\pi})\in\mathbb M_A.
\end{equation}
In particular, $C_{\varphi,1}=C_\varphi$ by Theorem~\ref{thm:projective-pairing}.
For $n\not\equiv1\pmod{q-1}$, one has $C_{\varphi,n}=0$
because $\Lcal_{\mathfrak j}(n)=0$.
\end{notation}
\begin{lemma}\label{lem:higher-lattice-differential}
   The differential $C_{\varphi,n}$ has the period-sum expression 
    \[
       C_{\varphi,n}
 =-\sum_{0\ne\lambda\in\Lambda_\varphi}
           \frac{\Hcal_\varphi(\lambda)}{\lambda^n},
    \]
   which is independent of the choice of $\tilde{\pi}$.
\end{lemma}
\begin{proof}
Convergence in $\mathbb M_A$ follows from the uniform
bound \eqref{eq:period-sum-bound} at $k=0$.
Write $\Lambda_\varphi=\tilde\pi\iota(\mathfrak j)$. Then
\[
\mathfrak j\setminus\{0\}\longrightarrow \Lambda_\varphi\setminus\{0\},
\qquad c\longmapsto \tilde\pi\iota(c)
\]
is bijective. Clearing a fixed denominator of
$\mathfrak j$ in the $A$-linearity of $\Hcal_\varphi$ gives
\[
\mathcal H_\varphi(\tilde\pi\iota(c))
=\chi(c)\mathcal H_\varphi(\tilde\pi).
\]
Hence
\[
\begin{aligned}
-\sum_{0\ne\lambda\in\Lambda_\varphi}
\frac{\mathcal H_\varphi(\lambda)}{\lambda^n}
&=-\sum_{0\ne c\in\mathfrak j}
\frac{\chi(c)\mathcal H_\varphi(\tilde\pi)}
      {(\tilde\pi\iota(c))^n}\\
&=-\tilde\pi^{-n}\Lcal_{\mathfrak j}(n)
\mathcal H_\varphi(\tilde\pi)
=C_{\varphi,n}.
\end{aligned}
\]
\end{proof}

\begin{notation}
 Define the lattice generating series 
    \begin{equation}\label{eq:generating-coefficients}
    \mathscr{C}_\varphi (z)= \sum_{n\ge1}C_{\varphi,n}z^{n-1}.
    \end{equation}
    In particular, $\mathscr{C}_\varphi (0)= C_{\varphi}$. Theorem~\ref{thm:lattice-generating} identifies this series
with $\Hcal_\varphi(z)/\exp_\varphi(z)$.
\end{notation}
 
\begin{corollary}\label{cor:sumH}
Put $r_\varphi=\min_{0\ne\lambda\in\Lambda_\varphi}|\lambda|$.
For $|z|<r_\varphi$,
\begin{equation}\label{eq:lattice-taylor-sum}
\begin{aligned}
\sum_{0\ne\lambda\in\Lambda_\varphi} \frac{\Hcal_\varphi(\lambda)}{z-\lambda}
 &=\mathscr{C}_\varphi(z)\\
 &=-\Hcal_\varphi(\tilde\pi)
       \sum_{n\ge1}\tilde\pi^{-n}\Lcal_{\mathfrak j}(n)z^{n-1}.
\end{aligned}
\end{equation}
\end{corollary}
\begin{proof}
The bound \eqref{eq:period-sum-bound} gives normal convergence
on every closed disk of radius less than $r_\varphi$.
Expand $(z-\lambda)^{-1}=-\sum_{n\ge1}z^{n-1}/\lambda^n$
and apply Lemma~\ref{lem:higher-lattice-differential}.
The second equality is \eqref{eq:higher-lattice-differential}.
\end{proof}

Write the normalized exponential of $\varphi$ as in
\cite[Section 2.6]{HuCauchy}:
\[
\exp_\varphi(z)=\sum_{j\ge0}e_j(\varphi)z^{q^j},
\qquad e_0(\varphi)=1.
\]

Finite $\Fq$-linear interpolation gives relations among the
$C_{\varphi,n}$. The same method underlies the Carlitz generating
series of \cite[Sections 2.4--2.6]{Perkins14}.

\begin{lemma}\label{lem:weighted-interpolation}
Put $I_\varphi(w)=\exp_\varphi(w)\mathscr C_\varphi(w)$. By
\eqref{eq:lattice-taylor-sum},
\[
I_\varphi(w)
 =\exp_\varphi(w)\sum_{0\ne\lambda\in\Lambda_\varphi}
       \frac{\Hcal_\varphi(\lambda)}{w-\lambda}.
\]
\begin{enumerate}
\item[(i)] The power series $I_\varphi$ has nonzero coefficients only at
powers $w^{q^j}$.
\item[(ii)] $I_\varphi$ extends to an entire function satisfying
\[
\tau_w I_\varphi-f_\varphi I_\varphi=-f_\varphi C_\varphi\exp_\varphi.
\]

\end{enumerate}
\end{lemma}
\begin{proof}

 Fix $0\ne\nu_0\in\Omega_A^{\mathrm{rat}}$ and write
$h(\lambda)=\Hcal_\varphi(\lambda)/\nu_0$.
This $\Fq$-linear map takes values in
$\operatorname{Frac}(\mathbb T_A)$; we take limits in the
Gauss completion of that field.
Division by $\nu_0$ defines a continuous map
$\mathbb M_A\to\operatorname{Frac}(\mathbb T_A)$, so
\eqref{eq:period-sum-bound} gives a uniform bound
$\|h(\lambda)\|\le M$ for $\lambda\in\Lambda_\varphi$.

 \emph{(i)}
For a nonzero finite-dimensional $\Fq$-subspace
$V\subset\Lambda_\varphi$, put
\[
E_V(w)=w\prod_{0\ne\lambda\in V}\Bigl(1-\frac{w}{\lambda}\Bigr),\qquad
I_V(w)=E_V(w)\sum_{0\ne\lambda\in V}\frac{h(\lambda)}{w-\lambda}.
\]
The polynomial $E_V$ is $\Fq$-linear of degree $|V|$ and
has derivative one. For $0\ne\lambda\in V$, the quotient
$E_V(w)/(w-\lambda)$ is a polynomial of degree $|V|-1$.
At $\mu\in V$, its value is one if $\mu=\lambda$ and zero
otherwise. Thus $I_V$ interpolates $h$ on $V$, including at zero.
The Moore determinant of an $\Fq$-basis of $V$ also gives an
$\Fq$-linear polynomial $P_V$ of degree at most $|V|/q$
interpolating $h$ on $V$. Uniqueness of interpolation in degree
less than $|V|$ gives $I_V=P_V$, so $I_V$ is $\Fq$-linear.

Now take $V=\{\lambda\in\Lambda_\varphi:|\lambda|\le R\}$ and let
$R\to\infty$. The lattice product expansion gives
$E_V\to\exp_\varphi$ uniformly on bounded disks
\cite[Section 8]{Hayes92}. Choose
$0<r<\min_{0\ne\lambda\in\Lambda_\varphi}|\lambda|$ and let
$\|\cdot\|_r$ be the Gauss norm on series converging on $|w|\le r$.
For every nonzero period,
\[
\left\|\frac{h(\lambda)}{w-\lambda}\right\|_r
\le\frac{M}{|\lambda|},\qquad
\left\|\sum_{|\lambda|>R}\frac{h(\lambda)}{w-\lambda}\right\|_r
\le\frac{M}{R}.
\]
Hence the sum converges in $\|\cdot\|_r$, and Corollary~\ref{cor:sumH}
gives
\begin{equation}\label{eq:sumLambda}
\sum_{0\ne\lambda\in\Lambda_\varphi}\frac{h(\lambda)}{w-\lambda}
 =\frac1{\nu_0}\sum_{n\ge1}C_{\varphi,n}w^{n-1}
 =\frac{\mathscr C_\varphi(w)}{\nu_0}.
\end{equation}
It follows that
$I_V\to\exp_\varphi\mathscr C_\varphi/\nu_0=I_\varphi/\nu_0$ in
$\|\cdot\|_r$, hence coefficientwise. Since every $I_V$ contains only
$q$-powers of $w$, so does $I_\varphi$.

\emph{(ii)}
On any bounded disk, choose $R$ larger than its radius.
By \eqref{eq:period-sum-bound}, the tail
$\sum_{|\lambda|>R}\Hcal_\varphi(\lambda)/(w-\lambda)$
has norm $O(R^{-1})$ there. Multiplication by $\exp_\varphi(w)$
cancels the finitely many remaining simple poles, since the
exponential has derivative one and simple zeros at the periods.
Thus $I_\varphi$ is entire. By (i), it is $\Fq$-linear, and
$I_\varphi'(0)=C_\varphi$ because $\mathscr C_\varphi(0)=C_\varphi$.

Let $D_w^{[m]}$
denote the Hasse derivative of order $m$ in $w$. Since $(-1)^{q-1}=1$
in the coefficient field,
$D_w^{[q-1]}(w-\lambda)^{-1}=(w-\lambda)^{-q}$. Twisting
\eqref{eq:lattice-taylor-sum} term by term and using the period relation
$\Hcal_\varphi(\lambda)^{(1)}=f_\varphi\Hcal_\varphi(\lambda)$ gives
\[
 \tau_w(\mathscr C_\varphi)
 =f_\varphi\sum_{0\ne\lambda\in\Lambda_\varphi}
       \frac{\Hcal_\varphi(\lambda)}{(w-\lambda)^q}
 =f_\varphi D_w^{[q-1]}\mathscr C_\varphi.
\]
 The additivity of $I_\varphi$ and $\exp_\varphi$,
together with their linear coefficients, gives
\[
I_\varphi(w+v)=I_\varphi(w)+C_\varphi v+O(v^q),\]
and 
\[ 
\exp_\varphi(w+v)=\exp_\varphi(w)+v+O(v^q).
\]
So taking the coefficient of $v^{q-1}$ in
$I_\varphi(w+v)/\exp_\varphi(w+v)$ yields
\[
D_w^{[q-1]}\mathscr C_\varphi
 =D_w^{[q-1]}\frac{I_\varphi}{\exp_\varphi}
 =\frac{I_\varphi-C_\varphi\exp_\varphi}{\exp_\varphi^q}.
\]
Substituting this into
$\tau_w (I_\varphi)=\exp_\varphi^q\,\tau_w (\mathscr C_\varphi) $ gives
\[
\tau_w (I_\varphi)
 =f_\varphi\bigl(I_\varphi-C_\varphi\exp_\varphi\bigr),
\]
which is the equation of (ii).
\end{proof}

 The interpolation numerator $I_\varphi$ and the generating
differential $\Hcal_\varphi$ satisfy the same inhomogeneous
Frobenius equation. Uniqueness yields the generating identity.

\begin{theorem}\label{thm:lattice-generating}
For every rank-one module $\varphi$, the following is an
identity of meromorphic differential-valued functions of $w$:
\begin{equation}\label{eq:lattice-generating} 
\frac{\Hcal_\varphi(w)}{\exp_\varphi(w)}
=\sum_{0\ne\lambda\in\Lambda_\varphi}
       \frac{\Hcal_\varphi(\lambda)}{w-\lambda}.
\end{equation}
The lattice sum converges normally away from the nonzero
periods. Its value at zero is $C_\varphi$, and its Taylor expansion
on $|w|<r_\varphi$ is $\mathscr C_\varphi(w)$.
Thus the quotient on the left has a removable singularity at zero
and gives the meromorphic continuation of $\mathscr C_\varphi$.
\end{theorem}
\begin{proof}
Both $I_\varphi$ and $\Hcal_\varphi$ are entire
$\Fq$-linear functions vanishing at zero, and they satisfy the same
Frobenius equation by Lemma~\ref{lem:weighted-interpolation}(ii) and
Lemma~\ref{lem:generating-frobenius}. Their difference
$D(w)=\sum_{r\ge0}D_rw^{q^r}$ satisfies
$\tau_w D=f_\varphi D$.
The coefficient of $w$ gives $D_0=0$, and then
$D_r=f_\varphi^{-1}D_{r-1}^{(1)}=0$ successively.
Hence $I_\varphi=\Hcal_\varphi$, proving \eqref{eq:lattice-generating}.
\end{proof}

\begin{corollary}\label{cor:element-generating}
Let $0\ne\lambda_0\in\Lambda_\varphi$ and write
$\Lambda_\varphi=\lambda_0\iota(\mathfrak j)$, with
$\mathfrak j\supset A$. Then
\begin{equation}\label{eq:fractional-generating}
\sum_{0\ne a\in\mathfrak j}\frac{\chi(a)}{z-\iota(a)}
 =\frac{\lambda_0}{\exp_\varphi(\lambda_0z)}
       \frac{\Hcal_\varphi(\lambda_0z)}{\Hcal_\varphi(\lambda_0)}.
\end{equation}
In particular, if $\Lambda_\psi=A\tilde\pi$, then
\begin{equation}\label{eq:free-generating}
\sum_{0\ne a\in A}\frac{\chi(a)}{z-\iota(a)}
 =\frac{\tilde\pi}{\exp_\psi(\tilde\pi z)}
       \frac{\Hcal_\psi(\tilde\pi z)}{\Hcal_\psi(\tilde\pi)}.
\end{equation}
The latter identity holds in $\mathbb T_A$ for $|z|<1$,
and meromorphically in $z$, with the value at zero obtained
by continuation. Neither sign normalization nor freeness
of $\Omega_A$ is required.
\end{corollary}
\begin{proof}
Clearing a denominator in the $A$-linearity of the period map gives
$\Hcal_\varphi(\lambda_0\iota(a))
=\chi(a)\Hcal_\varphi(\lambda_0)$ for $a\in\mathfrak j$.
Substitute $w=\lambda_0z$ in \eqref{eq:lattice-generating}
and divide by the nonzero period differential.
The quotient of differentials is taken in
$\operatorname{Frac}(\mathbb T_A)$.
For $\mathfrak j=A$, the left side belongs to $\mathbb T_A$
on $|z|<1$, since every nonzero $a\in A$ has $|\iota(a)|\ge1$.
For general $\mathfrak j$, the Taylor disk is
$|z|<r_\varphi/|\lambda_0|$, and a fixed denominator of
$\mathfrak j$ clears the geometric denominators.
\end{proof}

\subsection{Recurrences for the generating coefficients}
\begin{proposition}\label{prop:recurrence}
Let $\Hcal_\varphi(w)=\sum_{r\ge0}B_{\varphi,r}w^{q^r}$.
The coefficients $B_{\varphi,r}$ satisfy
\[
B_{\varphi,0}=C_\varphi,
\]
and, for every $r\ge 1$,
\begin{equation}\label{eq:recurrence}
B_{\varphi,r}
=e_r(\varphi)C_\varphi+f_\varphi^{-1}B_{\varphi,r-1}^{(1)}.
\end{equation}
Equivalently, iterating \eqref{eq:recurrence} gives
\[
B_{\varphi,r}
=\sum_{j=0}^{r}
\Bigl(\prod_{k=0}^{r-j-1}(f_\varphi^{(k)})^{-1}\Bigr)
e_j(\varphi)^{(r-j)}C_\varphi^{(r-j)},
\]
where the empty product is $1$ when $j=r$.
\end{proposition}

\begin{proof}
Substitute $\Hcal_\varphi(w)=\sum_{r\ge0}B_{\varphi,r}w^{q^r}$ into \eqref{eq:generating-frobenius}. Comparing the coefficients of $w^{q^r}$ gives
\[
B_{\varphi,r-1}^{(1)}-f_\varphi B_{\varphi,r}
=-f_\varphi C_\varphi e_r(\varphi)
\]
for $r\ge 1$, and
\[
-f_\varphi B_{\varphi,0}=-f_\varphi C_\varphi e_0(\varphi)=-f_\varphi C_\varphi
\]
for $r=0$. Hence $B_{\varphi,0}=C_\varphi$, and, for $r\ge 1$,
\[
f_\varphi B_{\varphi,r}
=B_{\varphi,r-1}^{(1)}+f_\varphi C_\varphi e_r(\varphi).
\]
Dividing by $f_\varphi$ yields \eqref{eq:recurrence}. The iterated formula follows by applying \eqref{eq:recurrence} successively and using $B_{\varphi,0}=C_\varphi$.
\end{proof}

 The expansion of $\Hcal_\varphi(w)$ and the identity
$\Hcal_\varphi(w)=\exp_\varphi(w)\mathscr C_\varphi(w)$ give
a recurrence for the quotient coefficients $C_{\varphi,n}$. 

\begin{proposition}\label{prop:higher-weight-recurrence}
Let $\varphi:A\to H\{\tau\} $ be a Drinfeld module over $ H $.
\begin{enumerate}
\item[(i)] $C_{\varphi,n}\in\Omega^1_{H(X)/H}$ for every $n\ge1$.
\item[(ii)] These differentials are determined by
\begin{align}
C_{\varphi,q^r}&=\frac{C_\varphi^{(r)}}{F_\varphi^{[r]}}
       &&(r\ge0),\label{eq:higher-weight-initial}\\
C_{\varphi,n}&=-\sum_{\substack{j\ge1\\q^j\le n}}
       e_j(\varphi)C_{\varphi,n-q^j+1}
       &&(n\notin\{q^r:r\ge0\}).\label{eq:higher-weight-recurrence}
\end{align}
\end{enumerate}
\end{proposition}
\begin{proof}
\emph{(ii).} Coefficient Frobenius is continuous on $\mathbb M_A$, so
\eqref{eq:higher-lattice-differential} at $n=1$ gives
\[
C_\varphi^{(r)}
 =-\sum_{0\ne\lambda\in\Lambda_\varphi}
       \frac{\Hcal_\varphi(\lambda)^{(r)}}{\lambda^{q^r}}
 =F_\varphi^{[r]}C_{\varphi,q^r}.
\]
Here the last equality uses
$\Hcal_\varphi(\lambda)^{(r)}=F_\varphi^{[r]}\Hcal_\varphi(\lambda)$
in $\operatorname{Frac}(\mathbb T_A)\otimes_A\Omega_A$.
Division by the nonzero function $F_\varphi^{[r]}$ proves
\eqref{eq:higher-weight-initial}.
For $n$ not a power of $q$, the coefficient of $w^n$ in
$\exp_\varphi(w)\mathscr C_\varphi(w)$ vanishes by
Lemma~\ref{lem:weighted-interpolation}(i). Since $e_0(\varphi)=1$,
this is precisely \eqref{eq:higher-weight-recurrence}.
Every index on its right is smaller than $n$.

\emph{(i).} Choose $a\in A\setminus\Fq$ and write
$\varphi_a=\sum_{i=0}^d\alpha_i\tau^i$, with $\alpha_0=\iota(a)$.
Comparing the coefficient of $z^{q^j}$ in
$\exp_\varphi(\iota(a) z)=\varphi_a(\exp_\varphi(z))$ gives
\begin{equation}\label{eq:exponential-coefficients}
 (\iota(a)^{q^j}-\iota(a))e_j(\varphi)
 =\sum_{i=1}^{\min(j,d)}\alpha_i e_{j-i}(\varphi)^{q^i}
\qquad(j\ge1).
\end{equation}
The factor $\iota(a)^{q^j}-\iota(a)$ is nonzero because
$\iota(a)$ is transcendental over $\Fq$. Induction gives $e_j(\varphi)\in H$.
Corollary~\ref{cor:finite-algorithm} and the calculation following it
give $C_\varphi\in\Omega^1_{H(X)/H}$ and $f_\varphi\in H(X)$.
Induction on $n$ in \eqref{eq:higher-weight-recurrence}, starting
from \eqref{eq:higher-weight-initial}, proves
$C_{\varphi,n}\in\Omega^1_{H(X)/H}$.
\end{proof}

In particular, 
\[
      C_{\varphi,q} = \frac{C_\varphi^{(1)}}{f_\varphi}
\]
and for $n = 2 q-1$, 
\begin{equation}\label{eq:2qn1}
      C_{\varphi,2 q -1 } = - e_1(\varphi) C_{\varphi, q} = - e_1(\varphi) \frac{C_\varphi^{(1)}}{f_\varphi}. 
\end{equation}
Substituting the definition \eqref{eq:higher-lattice-differential} of
$C_{\varphi,n}$, we obtain the explicit formula:
\[
        \tilde{\pi}^{-(2q-1)} \Lcal_{\mathfrak j}(2q-1)
        \Hcal_\varphi(\tilde{\pi})
        =e_1(\varphi)\frac{C_\varphi^{(1)}}{f_\varphi}.
\]
\subsection{Higher-weight Hilbert-class trace}
The coefficients $C_{\varphi,n}$ are rational over the field
of definition of $\varphi$. For a sign-normalized model, descent
from $H_A^+$ to $H_A$ is governed by their behavior under scalar
isomorphisms.
For $\varphi^z=z\varphi z^{-1}$,
the lattice is $z\Lambda_\varphi$, and \eqref{eq:period-scalar} gives
\begin{equation}\label{eq:higher-weight-covariance}
C_{\varphi^z,n}
 =-\sum_{0\ne\lambda\in\Lambda_\varphi}
       \frac{z\Hcal_\varphi(\lambda)}{(z\lambda)^n}
 =z^{1-n}C_{\varphi,n}.
\end{equation}
For an $\iota(K)$-embedding $\sigma:H\hookrightarrow\Cinf$,
\eqref{eq:exponential-coefficients} also gives
$\sigma(e_j(\varphi))=e_j(\varphi^\sigma)$ by induction.
Since $\sigma(x^{q^r})=\sigma(x)^{q^r}$, coefficient action commutes
with Frobenius twisting. Lemma~\ref{lem:rational-galois-covariance}
therefore proves covariance of the initial values
\eqref{eq:higher-weight-initial}. Applying $\sigma$ to the recurrence
and inducting on $n$ yields
\begin{equation}\label{eq:higher-galois-covariance}
\sigma(C_{\varphi,n})=C_{\varphi^\sigma,n}.
\end{equation}
\begin{theorem}\label{thm:higher-weight-trace}
Let $\psi$ be a sign-normalized Hayes module with
$\Lambda_\psi=A\tilde{\pi}$. If $n\ge1$ and
$n\equiv1\pmod{q^{\Ninf}-1}$, then
\begin{enumerate}
\item[(i)] $C_{\psi,n}\in\Omega^1_{H_A(X)/H_A}$;
\item[(ii)] the trace formula holds:
\begin{equation}\label{eq:higher-weight-trace}
\LL_\psi(n)\frac{\Hcal_\psi(\tilde{\pi})}{\tilde{\pi}^{n}}
       =\Tr_{H_A/K}(C_{\psi,n}).
\end{equation}
\end{enumerate}
\end{theorem}
\begin{proof}
Put $\varphi_i=\Ia_i*\psi$, $d_i=\partial\psi_{\Ia_i}$, and
$u_i=u_{\Ia_i}^\psi$. The lattice of $\varphi_i$ is
$d_i\tilde{\pi}\Ia_i^{-1}$, and
$\Hcal_{\varphi_i}(d_i\tilde{\pi})=u_i\Hcal_\psi(\tilde{\pi})$.
Equations~\eqref{eq:higher-lattice-differential} and
\eqref{eq:all-weight-unweighted} therefore give
\begin{equation}\label{eq:higher-weight-class}
C_{\varphi_i,n}
 =-\frac{u_i}{(d_i\tilde{\pi})^n}
       \Lcal_{\Ia_i^{-1}}(n)\Hcal_\psi(\tilde{\pi})
 =\LL_{\psi,[\Ia_i]}(n)
       \frac{\Hcal_\psi(\tilde{\pi})}{\tilde{\pi}^n}.
\end{equation}
\emph{(i).} For $\sigma\in G_0$, equations~\eqref{eq:sign-kernel-conjugation},
\eqref{eq:higher-weight-covariance}, and~\eqref{eq:higher-galois-covariance} give
\[
\sigma(C_{\psi,n})=C_{\psi^\sigma,n}
                  =z_a^{1-n}C_{\psi,n}=C_{\psi,n},
\]
where the last equality uses $n\equiv1\pmod{q^{\Ninf}-1}$.
Thus $C_{\psi,n}\in\Omega^1_{H_A(X)/H_A}$.
\emph{(ii).} Galois covariance identifies $C_{\varphi_i,n}$ with
$C_{\psi,n}^{\sigma_{\Ia_i}}$. Summing
\eqref{eq:higher-weight-class} proves the formula.
\end{proof}

 Substitution of \eqref{eq:higher-lattice-differential} gives
the corresponding trace relation with the element series. 
\begin{corollary}\label{cor:type-trace-higher}
Let $\psi$ be as in Theorem~\ref{thm:higher-weight-trace}.
For $n\ge1$ satisfying $n\equiv1\pmod{q^{\Ninf}-1}$,
\begin{equation}\label{eq:type-trace-higher1}
\LL_\psi(n)\frac{\Hcal_\psi(\tilde{\pi})}{\tilde{\pi}^n}
 = - \Tr_{H_A/K}\left(
       \Lcal_A(n)\frac{\Hcal_\psi(\tilde{\pi})}{\tilde{\pi}^n}\right).
\end{equation}
\end{corollary}

At rational infinity (i.e., $\Ninf=1$) and $n\equiv1\pmod{q-1}$, the equality
\eqref{eq:monic-all-element} gives
\[ 
C_{\psi,n}=L_+(A;n)\Hcal_\psi(\tilde{\pi})/\tilde{\pi}^n.\]
Thus \eqref{eq:type-trace-higher1} becomes
\begin{equation}\label{eq:type-trace-higher}
\LL_\psi(n)\frac{\Hcal_\psi(\tilde{\pi})}{\tilde{\pi}^n}
 =\Tr_{H_A/K}\left(
       L_+(A;n)\frac{\Hcal_\psi(\tilde{\pi})}{\tilde{\pi}^n}\right).
\end{equation}
At the weight $2q-1$, Theorem~\ref{thm:higher-weight-trace}
and \eqref{eq:2qn1} give
\begin{equation}\label{eq:first-nonfrobenius-trace}
\LL_\psi(2q-1)\frac{\Hcal_\psi(\tilde{\pi})}{\tilde{\pi}^{2q-1}}
 = \Tr_{H_A/K}(C_{\psi,2q-1})
 = -\Tr_{H_A/K}\left(e_1(\psi)\frac{C_\psi^{(1)}}{f_\psi}\right)
\qquad(\Ninf=1).
\end{equation}

\subsection{A generating trace for ideal values}
\label{sec:ideal-generating}
The free-lattice identity \eqref{eq:free-generating} describes
the element sum. To generate ideal-type values, sum the corresponding
lattice quotients over the ideal classes and retain the exponents
that descend to $H_A$.
Put $Q=q^{\Ninf}$ and define, for a meromorphic
differential-valued function $B(w)$,
\begin{equation}\label{eq:weight-projector}
 (\mathcal P_QB)(w)=\frac1{Q-1}
       \sum_{\varepsilon\in\iota(\mathbb F_\infty^\times)}
                     B(\varepsilon w).
\end{equation}
On Taylor series this retains exactly the powers divisible by
$Q-1$. The scalar $Q-1$ is invertible in the coefficient field.

\begin{theorem}\label{thm:ideal-generating-trace}
Let $\psi$ be a sign-normalized Hayes module with
$\Lambda_\psi=A\tilde\pi$.
Choose integral representatives $\Ia_1,\ldots,\Ia_h$ of
$\Cl(A)$ and put $\varphi_i=\Ia_i*\psi$. Then
\begin{equation}\label{eq:ideal-generating-trace}
\begin{aligned}
\sum_{k\ge0}\LL_\psi(1+k(Q-1))z^{k(Q-1)}
 &=\frac{\tilde\pi}{\Hcal_\psi(\tilde\pi)}
       \sum_{i=1}^h(\mathcal P_Q\mathscr C_{\varphi_i})(\tilde\pi z)\\
 &=\frac{\tilde\pi}{\Hcal_\psi(\tilde\pi)}
       \left.
       \Tr_{H_A/K}\bigl((\mathcal P_Q\mathscr C_\psi)(w)\bigr)
       \right|_{w=\tilde\pi z}.
\end{aligned}
\end{equation}
The series converges and the identities hold for
$|z|<|\tilde\pi|^{-1}\min_i r_{\varphi_i}$.
The trace acts coefficientwise on the Taylor series, fixing $w$,
before the substitution $w=\tilde\pi z$.
The finite class sum on the right gives the meromorphic
continuation of the generating series on the left.
\end{theorem}

\begin{proof}
By \eqref{eq:generating-coefficients} and \eqref{eq:weight-projector},
\[
 (\mathcal P_Q\mathscr C_{\varphi_i})(w)
 =\sum_{k\ge0}C_{\varphi_i,1+k(Q-1)}w^{k(Q-1)}.
\]
Theorem~\ref{thm:higher-weight-trace} puts the coefficients for
$\psi$ in $\Omega^1_{H_A(X)/H_A}$, and
\eqref{eq:higher-galois-covariance} identifies their conjugates
with the displayed coefficients for $\varphi_i$.
For $n=1+k(Q-1)$, equation~\eqref{eq:higher-weight-class} gives
\[
\frac{\tilde\pi}{\Hcal_\psi(\tilde\pi)}
       C_{\varphi_i,n}(\tilde\pi z)^{n-1}
 =\LL_{\psi,[\Ia_i]}(n)z^{n-1}.
\]
Sum over $i$ and $k$.
Normal convergence on every smaller closed disk follows from
Theorem~\ref{thm:lattice-generating}; the same theorem gives
the asserted meromorphic continuation.
\end{proof}

At $z=0$, \eqref{eq:ideal-generating-trace} is the weight-one
trace \eqref{eq:hayes-trace}.
When $\Ninf=1$, $\Fq$-linearity of $\Hcal_\varphi$ and
$\exp_\varphi$ gives
$\mathscr C_\varphi(\zeta w)=\mathscr C_\varphi(w)$ for
$\zeta\in\Fq^\times$. Thus $\mathcal P_q\mathscr C_\varphi
=\mathscr C_\varphi$, and no projection is needed.
For higher-degree infinity, Section~\ref{sec:twisted-weights} identifies
the obstruction to dropping $\mathcal P_Q$.

\subsection{Fields of definition}\label{sec:twisted-weights}
The scalar covariance of $C_{\psi,n}$ determines its coefficient
field and the character needed to obtain an ordinary Hilbert-class trace.
Let $\psi$ be a sign-normalized Hayes module with
$\Lambda_\psi=A\tilde{\pi}$.
\begin{definition}
For $n\equiv1\pmod{q-1}$, define
the character
\begin{equation}\label{eq:weight-character}
\kappa_{\psi,n}:G_0\longrightarrow\iota(\mathbb F_\infty^\times),
\qquad
\kappa_{\psi,n}(\sigma_{(a)})
 =\iota(\operatorname{sgn}(a))^{(n-1)q^{\xi_\psi}}.
\end{equation}
The congruence makes it independent of the representative in
$\mathbb F_\infty^\times/\Fq^\times$. Put
$H_n=(H_A^+)^{\ker\kappa_{\psi,n}}$.
\end{definition}

\begin{proposition}\label{prop:weight-field}
For $n\ge1$ with $n\equiv1\pmod{q-1}$:
\begin{enumerate}
\item[(i)] the differential $C_{\psi,n}$ is nonzero;
\item[(ii)] the smallest field $H_A\subset E\subset H_A^+$ for which
$C_{\psi,n}\in\Omega^1_{E(X)/E}$ is $H_n$;
\item[(iii)] the degree of $H_n/H_A$ is
\begin{equation}\label{eq:weight-field-degree}
 [H_n:H_A]=\frac{q^{\Ninf}-1}{\gcd(n-1,q^{\Ninf}-1)}.
\end{equation}
\end{enumerate}
\end{proposition}
\begin{proof}
\emph{(i).} The constant elements in $\Lcal_A(n)$ contribute $-1$. Every
nonconstant $a\in A$ has $|\iota(a)|\ge q^{\Ninf}$, so
\[
 \|\Lcal_A(n)+1\|\le q^{-\Ninf n}<1.
\]
Thus $\Lcal_A(n)\ne0$. Since
$\Res_\Xi\Hcal_\psi(\tilde{\pi})=\tilde{\pi}\ne0$,
equation~\eqref{eq:higher-lattice-differential} gives
$C_{\psi,n}\ne0$.
\emph{(ii).} Equations~\eqref{eq:sign-kernel-conjugation},
\eqref{eq:higher-weight-covariance}, and~\eqref{eq:higher-galois-covariance}
give
\begin{equation}\label{eq:weight-character-action}
\sigma(C_{\psi,n})=\kappa_{\psi,n}(\sigma)C_{\psi,n}
\qquad(\sigma\in G_0).
\end{equation}
Since $C_{\psi,n}\ne0$, its stabilizer in $G_0$ is exactly
$\ker\kappa_{\psi,n}$. The fixed-space identity for rational
differentials therefore identifies the smallest coefficient field
of $C_{\psi,n}$ over $H_A$ with $H_n$. \emph{(iii).} The character in
\eqref{eq:weight-character} has order
$(q^{\Ninf}-1)/\gcd(n-1,q^{\Ninf}-1)$, since
$q^{\xi_\psi}$ is prime to $q^{\Ninf}-1$.
This proves \eqref{eq:weight-field-degree}.
\end{proof}

For $n\ge1$ with $n\equiv1\pmod{q-1}$, the nonzero differential
$C_{\psi,n}$ thus has coefficients in $H_A$ if and only if
$n\equiv1\pmod{q^{\Ninf}-1}$, the congruence in
Theorem~\ref{thm:higher-weight-trace}. 
If $n\not\equiv1\pmod{q-1}$, then $C_{\psi,n}=0$ by the
$\Fq^\times$-action; this vanishing does not apply to the sign-one
series $L_+(A;n)$.

\subsection{Twisted Pellarin series}\label{sec:twisted-traces}
Character twists remove the descent obstruction at weights
$n\equiv1\pmod{q-1}$ identified in Section~\ref{sec:twisted-weights}.
Throughout this subsection, let $\psi$ be a sign-normalized Hayes
module with free period lattice $\Lambda_\psi=A\tilde\pi$.
\begin{definition}
For any $n\ge1$ and $\mu\in (H_A^+)^{\times}$, define
\begin{equation}\label{eq:twisted-ideal-series}
\LL_\psi^\mu(n)=\sum_{0\ne\Ia\subset A}
       \sigma_\Ia(\mu)\frac{u_\Ia^\psi}{(d_\Ia^\psi)^n}.
\end{equation}
Such sums with Artin conjugates of $\mu$ occur in
\cite[Section 5]{ANDTR17}. 
\end{definition}
The coefficients $\sigma_\Ia(\mu)$ range over a finite set, so
Proposition~\ref{prop:all-weight-comparison} implies unconditional
convergence of $\LL_\psi^\mu(n)$. On each narrow ideal class, $\sigma_\Ia(\mu)$ is constant.
Consequently Corollary~\ref{cor:narrow-weight-comparison} gives,
for all $n\ge1$,
\begin{equation}\label{eq:twisted-narrow-comparison}
\LL_\psi^\mu(n)=\sum_{j=1}^{h^+}
       \sigma_{\Ib_j}(\mu)\frac{u_{\Ib_j}^\psi}{(d_{\Ib_j}^\psi)^n}
       L_+(\Ib_j^{-1};n).
\end{equation}

\begin{theorem}\label{thm:twisted-weight-trace}
Let $\psi$ be a sign-normalized Hayes module with
$\Lambda_\psi=A\tilde\pi$. 
Let $n\ge1$ satisfy $n\equiv1\pmod{q-1}$, and let $\mu\in(H_A^+)^\times$ satisfy
\begin{equation}\label{eq:twist-character}
\sigma(\mu)=\kappa_{\psi,n}(\sigma)^{-1}\mu
\quad(\sigma\in G_0).
\end{equation}
Then
\begin{enumerate}
\item[(i)] $\mu C_{\psi,n}\in\Omega^1_{H_A(X)/H_A}$;
\item[(ii)] the twisted trace formula holds:
\begin{equation}\label{eq:twisted-weight-trace}
\LL_\psi^\mu(n)\frac{\Hcal_\psi(\tilde{\pi})}{\tilde{\pi}^n}
 =\Tr_{H_A/K}(\mu C_{\psi,n}).
\end{equation}
\end{enumerate}
\end{theorem}

\begin{proof}
\emph{(i).} Equations~\eqref{eq:weight-character-action} and
\eqref{eq:twist-character} show that $\mu C_{\psi,n}$ is fixed by $G_0$,
which proves the asserted coefficient field.

\emph{(ii).} Fix an ordinary ideal-class representative $\Ia_i$. Every integral
ideal $\Ia$ in the ordinary class $[\Ia_i]$ can be written uniquely as
$\Ia=c\Ia_i$ with $c\in \Ia_i^{-1}\setminus\{0\}$, modulo
multiplication by units in $\Fq^\times$. Hence the subseries of
\eqref{eq:twisted-ideal-series} over $[\Ia_i]$ is
\[
\LL_{\psi,[\Ia_i]}^\mu(n)
 =\sum_{c\in \Ia_i^{-1}/\Fq^\times}
   \sigma_{c\Ia_i}(\mu)
   \frac{u_{c\Ia_i}^\psi}{(d_{c\Ia_i}^\psi)^n}.
\]
The ratio
$\sigma_{c\Ia_i}(\mu)/\sigma_{\Ia_i}(\mu)$ is the inverse of
the sign factor in \eqref{eq:weight-rescaling}.
To verify this relation, use the Artin action on the residue
field constants:
\[
\sigma_{\Ia_i}(\iota(\varepsilon))
 =\iota(\varepsilon)^{q^{\deg\Ia_i}}
\quad(\varepsilon\in\mathbb F_\infty).
\]
For $0\ne c\in\Ia_i^{-1}$, it follows that
\[
\sigma_{c\Ia_i}(\mu)
 =\iota(\operatorname{sgn}(c))^{(1-n)q^{\xi_\psi+\deg\Ia_i}}
   \sigma_{\Ia_i}(\mu).
\]
Multiplication by \eqref{eq:weight-rescaling} cancels the sign factor:
\begin{equation}\label{eq:twisted-principal-rescaling}
\sigma_{c\Ia_i}(\mu)
   \frac{u_{c\Ia_i}^\psi}{(d_{c\Ia_i}^\psi)^n}
 =\sigma_{\Ia_i}(\mu)
   \frac{u_{\Ia_i}^\psi}{(d_{\Ia_i}^\psi)^n}
   \frac{\chi(c)}{\iota(c)^n}.
\end{equation}

Summing
\eqref{eq:twisted-principal-rescaling}
over $c\in \Ia_i^{-1}/\Fq^\times$, and using \eqref{eq:all-element-series-signed} for $ \mathfrak j = \Ia_i^{-1} $, we obtain
\[
\LL_{\psi,[\Ia_i]}^\mu(n)
 =-\sigma_{\Ia_i}(\mu)
   \frac{u_{\Ia_i}^\psi}{(d_{\Ia_i}^\psi)^n}
   \Lcal_{\Ia_i^{-1}}(n).
\]

For $\varphi_i=\Ia_i*\psi$, the period-map isogeny identity and
\eqref{eq:higher-lattice-differential} now give
\[
\LL_{\psi,[\Ia_i]}^\mu(n)
       \frac{\Hcal_\psi(\tilde{\pi})}{\tilde{\pi}^n}
 =\sigma_{\Ia_i}(\mu)C_{\varphi_i,n}
 =\sigma_{\Ia_i}(\mu C_{\psi,n}).
\]
The last equality is \eqref{eq:higher-galois-covariance}.
Since $\mu C_{\psi,n}$ has coefficients in $H_A$, summing over the
ordinary ideal classes proves \eqref{eq:twisted-weight-trace}.
\end{proof}

\begin{example}
When $n\equiv1\pmod{q^{\Ninf}-1}$, one may take $\mu=1$,
and \eqref{eq:twisted-weight-trace} recovers
Theorem~\ref{thm:higher-weight-trace}. 
\end{example}

\begin{remark}
Let $n\ge1$ satisfy $n\equiv1\pmod{q-1}$.
The character $\kappa_{\psi,n}$ takes values in
$\iota(\mathbb F_\infty^\times)\subset H_A^\times$.
For a normal-basis generator $x$ of $H_A^+/H_A$, the element
\[
\mu=\sum_{\sigma\in G_0}\kappa_{\psi,n}(\sigma)\sigma(x)
\]
is nonzero because the conjugates $\sigma(x)$ are linearly
independent over $H_A$. Reindexing the sum gives
\eqref{eq:twist-character}. Thus a twisting scalar exists at every
such weight, and it belongs to $H_n$ because it is fixed by
$\ker\kappa_{\psi,n}$.
\end{remark}
\begin{example}\label{ex:gamma-twist}
Choose $\gamma\in(H_A^+)^\times$ satisfying
\eqref{eq:twist-character} at $n=q$.
Since $\sigma(f_\psi)=\kappa_{\psi,q}(\sigma)^{-1}f_\psi$
for $\sigma\in G_0$, we have
\[
      \gamma C_{\psi,q} = \gamma \frac{C_\psi^{(1)}}{f_\psi}, \qquad \frac{f_\psi}{\gamma}\in H_A(X).
\]
Such scalar normalizations of the shtuka function occur in
\cite[Lemma 3.3 and Remark 5.4]{ANDTR17}.
For any $n\equiv1\pmod{q-1}$, set
\begin{equation}\label{eq:twist-power-choice}
\mu=\gamma^{(n-1)/(q-1)}.
\end{equation}
The element $\mu$ satisfies \eqref{eq:twist-character} because
$\kappa_{\psi,n}=\kappa_{\psi,q}^{(n-1)/(q-1)}$.

For $q=2$, every $n\ge1$
admits the choice $\mu=\gamma^{n-1}$ at every degree of infinity.
The weights $q$ and $2q-1$ in Theorem~\ref{thm:twisted-weight-trace} give
\begin{align*}
\LL_\psi^\gamma(q)\frac{\Hcal_\psi(\tilde{\pi})}{\tilde{\pi}^q}
 &=\Tr_{H_A/K}\left(\gamma\frac{C_\psi^{(1)}}{f_\psi}\right),\\
\LL_\psi^{\gamma^2}(2q-1)
       \frac{\Hcal_\psi(\tilde{\pi})}{\tilde{\pi}^{2q-1}}
 &=-\Tr_{H_A/K}\left(\gamma^2e_1(\psi)
       \frac{C_\psi^{(1)}}{f_\psi}\right).
\end{align*}
These identities follow from \eqref{eq:higher-weight-initial}
and \eqref{eq:higher-weight-recurrence}, and hold for every
$\Ninf\ge1$.

\end{example}

The powers of $\gamma$ in \eqref{eq:twist-power-choice}
can be incorporated into the generating parameter.
This gives one ordinary Hilbert-class trace for the whole
family of twisted values.

\begin{corollary}\label{cor:twisted-generating-trace}
Let $\psi$ be a sign-normalized Hayes module with
$\Lambda_\psi=A\tilde\pi$, and let $\gamma\in(H_A^+)^\times$
satisfy \eqref{eq:twist-character} at $n=q$. Choose
$\beta\in\Cinf^\times$ with $\beta^{q-1}=\gamma$, and put
\begin{equation}\label{eq:twisted-generating-input}
\mathscr C_{\psi,\gamma}(w):=\mathscr C_\psi(\beta w)
 =\sum_{k\ge0}\gamma^kC_{\psi,1+k(q-1)}w^{k(q-1)}.
\end{equation}
\begin{enumerate}
\item[(i)] $\mathscr C_{\psi,\gamma}$ is independent of the choice of
$\beta$ and has Taylor coefficients in $\Omega^1_{H_A(X)/H_A}$;
\item[(ii)] near $z=0$,
\begin{equation}\label{eq:twisted-generating-trace}
\sum_{k\ge0}\LL_\psi^{\gamma^k}(1+k(q-1))z^{k(q-1)}
 =\frac{\tilde\pi}{\Hcal_\psi(\tilde\pi)}
       \left.\Tr_{H_A/K}\bigl(\mathscr C_{\psi,\gamma}(w)\bigr)
       \right|_{w=\tilde\pi z},
\end{equation}
where, as in \eqref{eq:ideal-generating-trace}, the trace is taken
coefficientwise before substitution;
\item[(iii)] put $\varphi_i=\Ia_i*\psi$ and choose
$\beta_i$ with $\beta_i^{q-1}=\sigma_{\Ia_i}(\gamma)$.
The generating series on the left of
\eqref{eq:twisted-generating-trace} has meromorphic continuation
\begin{equation}\label{eq:twisted-generating-class-sum}
\frac{\tilde\pi}{\Hcal_\psi(\tilde\pi)}
       \sum_{i=1}^h
       \frac{\Hcal_{\varphi_i}(\beta_i\tilde\pi z)}
            {\exp_{\varphi_i}(\beta_i\tilde\pi z)}.
\end{equation}
\end{enumerate}
\end{corollary}
\begin{proof}
\emph{(i).} The quotient $\mathscr C_\psi$ is invariant under
$w\mapsto\zeta w$, $\zeta\in\Fq^\times$.
Thus $\mathscr C_\psi(\beta w)$ depends on $\beta$ only
through $\beta^{q-1}=\gamma$.
Equation~\eqref{eq:generating-coefficients} gives
\eqref{eq:twisted-generating-input}, and
Theorem~\ref{thm:twisted-weight-trace}(i) gives the asserted
coefficient field for each Taylor coefficient.

\emph{(ii).} For $n=1+k(q-1)$, the character of $\gamma^k$ is
$\kappa_{\psi,n}^{-1}$. Theorem~\ref{thm:twisted-weight-trace}(ii)
therefore gives \eqref{eq:twisted-generating-trace} by comparison
of coefficients.

\emph{(iii).} Galois covariance (\eqref{eq:higher-galois-covariance}) gives
\[
\sigma_{\Ia_i}\bigl(\gamma^kC_{\psi,1+k(q-1)}\bigr)
 =\beta_i^{k(q-1)}C_{\varphi_i,1+k(q-1)}.
\]
Consequently the coefficientwise trace is
$\sum_i\mathscr C_{\varphi_i}(\beta_iw)$.
This proves \eqref{eq:twisted-generating-class-sum}, with
convergence on
$|z|<\min_i r_{\varphi_i}/|\beta_i\tilde\pi|$.
\end{proof}

At rational infinity one may take $\gamma=\beta=1$,
recovering \eqref{eq:ideal-generating-trace}.
For $q=2$, equation~\eqref{eq:twisted-generating-trace}
generates every positive weight for any $\Ninf$.
For $q>2$, the range remains $n\equiv1\pmod{q-1}$,
in agreement with the vanishing of the other lattice coefficients.

\section{The projective line}\label{sec:scalar}
On $\mathbb P^1$, the Cauchy differential $C_\psi$ is
determined by its two simple poles and their residues.
The Hu--Huang--Yau modules give closed ideal-value formulas for
both free and nonfree period lattices, together with explicit
character twists at higher weights. At rational infinity, the
generating identity specializes to the Carlitz identity of Perkins.
\subsection{The geometric setting of the HHY ring}
Let $\rho(t)\in\Fq[t]$ be monic irreducible of degree
$\Ninf\ge2$, and set
\[
A=\Gamma(\mathbb P^1\setminus P_\rho,\OO),\qquad
\theta=\iota(t),\qquad
I_\infty=(1/\rho,t/\rho,\ldots,t^{\Ninf-1}/\rho).
\]
Here $I_\infty$ is the ideal of the usual rational point
$\infty_0=(t=\infty)$, whereas the removed infinite point is $P_\rho$.
The divisor class sequence gives
$\Cl(A)=\langle[I_\infty]\rangle\simeq\mathbb Z/\Ninf\mathbb Z$
and $I_\infty^{\Ninf}=(1/\rho)$.
The ordinary Hilbert class field is
$H_A=\mathbb F_{q^{\Ninf}}(\theta)$: this constant extension is
unramified, splits $P_\rho$ completely, and has degree $|\Cl(A)|$
\cite[Chapters 8 and 14]{Rosen02} and
\cite[Theorem 15.6]{Hayes92}.
The HHY modules over $H_A$ and their sign-normalized
scalar conjugates are constructed in \cite{HHY26}.

Let $\eta\in\mathbb F_{q^{\Ninf}}$ be the root of $\rho$
chosen in \cite{HHY26}, and put $\eta_i=\eta^{q^i}$.
Write $\varphi$ for the Hu--Huang--Yau source module, denoted
$\Psi$ in \cite{HHY26}, whose shtuka pole is
$q_\varphi=(t=\eta)$. Let $\sigma$ be the coefficient
automorphism of $H_A$ fixing $\theta$ and sending $\eta$ to
$\eta^q$; its action on $H_A(t)$ fixes the geometric variable $t$.
Set $\psi=\varphi^{\sigma^2}$, the module denoted by $\Psi^{(2)}$
in \cite{HHY26}. Thus $q_\psi=(t=\eta_2)$ and
$q_\psi^{(-1)}=(t=\eta_1)$.

\subsection{The value for a free period lattice}\label{sec:p1-free}
The module $\psi=\varphi^{\sigma^2}$ has shtuka pole
$q_\psi=(t=\eta_2)$ and shtuka function
\[
      f_\psi = \frac{1}{t - \eta_2 } - \frac{1}{\theta - \eta_2}.
\]
Choose the embedding of $\mathbb F_{q^{\Ninf}}$ into
$\Cinf$ so that $\theta$ reduces to $\eta_0$; the fixed absolute-value
normalization gives $|\theta-\eta_0|=q^{-\Ninf}$.
Use the compatible root and period choices of
\cite[Theorems 3.32 and 6.8]{HHY26}. Put 
$\Theta_i=(\theta-\eta_i)^{-1}$, and $\Theta=\Theta_0$.
A generator of $\Lambda_\psi$ is
\begin{equation}\label{eq:p1-special-period}
\tilde\pi
 =\sqrt[q-1]{-\Theta_2}\,\Theta_1\prod_{k\ge1}
   \frac{\Theta^{q^k}}{\Theta^{q^k}-\Theta_k},
\end{equation}
see the product formula of \cite[Theorem 3.32]{HHY26}, whose proof also gives the freeness
$\Lambda_\psi=A\tilde\pi$.

The identity datum $J=\mathrm{id}_\psi$ gives
\begin{equation}\label{eq:p1-free-cauchy}
K_J(\Xi,\cdot)=C_\psi
 =\left(\frac1{t-\theta}-\frac1{t-\eta_1}\right)dt.
\end{equation}
This is the unique differential with
only simple poles at $t=\theta$ and $t=\eta_1$ and residues $1$
and $-1$, respectively.
Taking the coefficient trace gives the ideal value by
Corollary~\ref{cor:hayes-trace}.
\begin{proposition}\label{prop:p1-free-ideal-value}
With the notation above, the integral-ideal Pellarin value is
\begin{equation}\label{eq:rational-closed-infinity}
{\LL_\psi(1)\frac{\Hcal_{\psi}(\tilde{\pi})}{\tilde{\pi}}
 = \Tr_{H_A / K} (C_\psi)= \left(\frac{\Ninf}{t-\theta}-\frac{\rho'(t)}{\rho(t)}\right)dt.}
\end{equation}
\end{proposition}
\begin{proof}
       The logarithmic derivative identity
$\rho'/\rho=\sum_{\rho(\alpha)=0}(t-\alpha)^{-1}$ proves the
formula, which also holds when $p\mid\Ninf$.
\end{proof}

\subsection{The HHY ideal value}\label{sec:hhy-ideal-value}

The source shtuka function and special function are
\begin{equation}\label{eq:hhy-special-function}
f_\varphi(t)=\frac1{t-\eta}-\Theta,\qquad
\omega_\varphi(t)=\sqrt[q-1]{-\Theta}
    \prod_{k=0}^{\infty}
       \frac{\Theta^{q^k}}
            {\Theta^{q^k}-(t-\eta_k)^{-1}}.
\end{equation}
 Put 
\begin{equation}\label{eq:hhy-period-constant}
C_{0,2}=\sqrt[q-1]{\Theta_2/\Theta}\,\frac{\Theta_1}{\Theta},
\qquad \kappa=\frac{\tilde\pi}{C_{0,2}\Theta^2}.
\end{equation}

\begin{lemma}[The HHY special function]\label{lem:hhy-special-function}
For the module $\varphi$ and choices above:
\begin{enumerate}
\item[(i)] the product $\omega_\varphi$ converges to a unit of
$\mathbb T_A$ and satisfies
$\omega_\varphi^{(1)}=f_\varphi\omega_\varphi$, and
$\sfm(\varphi)=A\omega_\varphi$;
\item[(ii)] the residues at the characteristic point are
\begin{equation}\label{eq:hhy-residues}
\Res_{t=\theta}\left(\omega_\varphi\,d\frac1{t-\eta}\right)
 =-\frac{\tilde\pi}{C_{0,2}},\qquad
\Res_{t=\theta}(\omega_\varphi dt)=\kappa,
\end{equation}
and $\Lambda_\varphi=\kappa\,\iota(I_\infty^2)$ with
\begin{equation}\label{eq:hhy-period-differential}
\Hcal_\varphi(\iota(a)\kappa)=\chi(a)\omega_\varphi dt
\qquad(a\in I_\infty^2).
\end{equation}
\end{enumerate}
\end{lemma}
\begin{proof}
Each $(t-\eta_i)^{-1}$ lies in
$A\otimes_{\Fq}\mathbb F_{q^{\Ninf}}$ and has Gauss norm one.
Since $|\Theta|>1$, the product defining $\omega_\varphi$ in
\eqref{eq:hhy-special-function} converges in $\mathbb T_A$ to a unit,
and $\omega_\varphi^{(1)}=f_\varphi\omega_\varphi$; the quotient of
any special function by $\omega_\varphi$ is Frobenius-fixed in
$\mathbb T_A$, hence belongs to $A$. This proves (i).
The residue formula \cite[Theorem 6.8]{HHY26} gives
\eqref{eq:hhy-residues}. Since $\Omega_A=I_\infty^2dt$, the residue
isomorphism yields the lattice description and
\eqref{eq:hhy-period-differential}.
\end{proof}

\begin{proposition}[The HHY ideal value]\label{prop:hhy-ideal-value}
 For the module $\varphi$ and choices above, the integral-ideal Pellarin value is
\begin{equation}\label{eq:hhy-ideal-value}
 \LL_\varphi(1)\omega_\varphi(t)
 =\kappa\left(\frac{\Ninf}{t-\theta}
                     +\frac{\rho'(\theta)}{\rho(\theta)}\right),
\end{equation}
where $\Ninf$ is read in $\Fq$ and $\rho'(\theta)$ is evaluated
in the arithmetic variable.
\end{proposition}
\begin{proof}
The isogeny of \cite[Proposition 5.1 and Corollary 5.3]{HHY26}
from the source $\varphi$ to the free target $\psi$ is
\begin{equation}\label{eq:hhy-isogeny-datum}
\eta_J=C_{0,2}(\tau+\Theta)\circ(\tau+\Theta),\qquad
\mathfrak i_J=I_\infty^2,\qquad \mu_J=C_{0,2}.
\end{equation}
In particular $\partial\eta_J=C_{0,2}\Theta^2$ and $z_J=\kappa$.
The motive multiplier of $\eta_J$ is computed in the source frame:
\begin{align*}
h_J
 &=C_{0,2}\bigl(F_\varphi^{[2]}
                 +(\Theta^q+\Theta)f_\varphi+\Theta^2\bigr)\\
 &=\frac{C_{0,2}\Theta}
          {\Theta_1(t-\eta)(t-\eta_1)}.
\end{align*}
The second equality follows by inserting
$f_\varphi^{(1)}=(t-\eta_1)^{-1}-\Theta^q$ and
$\Theta-\Theta_1=(\eta-\eta_1)\Theta\Theta_1$.
Thus
\begin{equation}\label{eq:hhy-normalized-datum}
g_J=\frac{(\theta-\eta)(\theta-\eta_1)}
           {(t-\eta)(t-\eta_1)},\qquad
\Gcal_J(\kappa)=\omega_\varphi dt.
\end{equation}
For the last identity, combine \eqref{eq:hhy-period-differential}
with \eqref{eq:adapted-source-period} for any nonzero
$a\in I_\infty^2$ and cancel $\chi(a)$.

Since $q_\psi^{(-1)}=(t=\eta_1)$, $C_\psi$ is the
differential of \eqref{eq:p1-free-cauchy}. Consequently
\begin{equation}\label{eq:hhy-trace-input}
K_J(\Xi,\cdot)=\frac{C_\psi}{g_J}
 =\frac{t-\eta}
        {(\theta-\eta)(t-\theta)}\,dt.
\end{equation}
The coefficient trace fixes $t,dt$ and runs through
the roots $\eta_0,\ldots,\eta_{\Ninf-1}$ of $\rho$. Since
\[
\sum_{i=0}^{\Ninf-1}\frac{t-\eta_i}{\theta-\eta_i}
 =\Ninf+(t-\theta)\frac{\rho'(\theta)}{\rho(\theta)},
\]
we obtain
\begin{equation}\label{eq:hhy-trace-evaluated}
\Tr_{H_A/K}\bigl(K_J(\Xi,\cdot)\bigr)
 =\left(\frac{\Ninf}{t-\theta}
                +\frac{\rho'(\theta)}{\rho(\theta)}\right)dt.
\end{equation}
Substitution of $z_J=\kappa$ and
$\Gcal_J(\kappa)=\omega_\varphi dt$ in
\eqref{eq:general-trace} proves the value formula.

\end{proof}

The lattice description explains the normalization in
\eqref{eq:hhy-ideal-value}: $\kappa$ is the residue of
$\omega_\varphi dt$ and generates
$\iota(I_\infty^{-2})\Lambda_\varphi$. The source periods form
$\kappa\,\iota(I_\infty^2)$ even though $\sfm(\varphi)$ is free.

The generating identity also applies directly to the
source lattice, including the nonprincipal cases $\Ninf>2$.
Combining \eqref{eq:lattice-generating} with
\eqref{eq:hhy-period-differential} gives
\begin{equation}\label{eq:hhy-generating}
\sum_{0\ne a\in I_\infty^2}\frac{\chi(a)}{z-\iota(a)}
 =\frac{\kappa}{\exp_\varphi(\kappa z)}
       \frac{\Hcal_\varphi(\kappa z)}{\omega_\varphi\,dt}.
\end{equation}
Here the denominator is the adapted differential
$\Gcal_J(\kappa)=\omega_\varphi dt$.
Since $\kappa$ need not belong to $\Lambda_\varphi$, this
denominator need not be a value of $\Hcal_\varphi$ at a source
period.

\subsection{Higher weights}
For higher weights, use the sign function determined by
$\operatorname{sgn}(\rho)=1$, as in
\cite[Example 2.8]{HHY26}, and choose a sign-normalized scalar
conjugate $\psi^+$ of $\psi$. The shtuka divisor gives
\[
f_{\psi^+}=\gamma\frac{t-\theta}{ t-\eta^{q^2}},
\qquad \gamma\in(H_A^+)^\times.
\]
Write $\psi^+=z\psi z^{-1}$ with $z\in\Cinf^\times$.
Then $\Lambda_{\psi^+}=Az\tilde\pi$, and scalar covariance gives
$f_{\psi^+}=z^{q-1}f_\psi$. Since
$f_\psi=-\Theta_2(t-\theta)/(t-\eta_2)$, we obtain
$\gamma=-z^{q-1}\Theta_2$.

\begin{lemma}\label{lem:normalizing-scalar}
        The following hold for the normalizing scalar $\gamma$ of the sign-normalized module $\psi^+$.
\begin{enumerate}
\item[(i)] Put $m=(q^{\Ninf}-1)/(q-1)$. The scalar $\gamma$ satisfies
\begin{equation}\label{eq:gamma-kummer}
 \gamma^m=\prod_{j=2}^{\Ninf+1}(-\Theta_j)^{q^{\Ninf+1-j}}\in H_A,
\end{equation}
and $H_A^+=H_A(\gamma)$.
\item[(ii)] Every $\sigma\in G_0=\operatorname{Gal}(H_A^+/H_A)$ acts
by $\gamma\mapsto\varepsilon\gamma$ with
$\varepsilon\in\iota(\mathbb F_\infty^\times)$,
$\varepsilon^m=1$.
\item[(iii)] One has
$\sigma(\gamma)=\kappa_{\psi^+,q}(\sigma)^{-1}\gamma$ for
$\sigma\in G_0$; that is, $\gamma$ satisfies
\eqref{eq:twist-character} at $n=q$.
\end{enumerate}
\end{lemma}
\begin{proof}
\emph{(i).} Applied to the $\sigma^2$-conjugate,
\cite[Proposition 4.1]{HHY26} gives
$\psi^+=z\psi z^{-1}$ with $z^{q-1}=-\gamma/\Theta_2$.
In the notation of that reference, it also gives
$\gamma^m=(-\Theta_2)^{\gamma_\Ninf(\sigma)}$.
Here $\gamma_\Ninf(\sigma)=\sum_{i=0}^{\Ninf-1}q^i\sigma^{\Ninf-1-i}$
acts multiplicatively on $-\Theta_2$, with
$\sigma(\eta)=\eta^q$ and $\sigma(\theta)=\theta$.
Expanding this action yields \eqref{eq:gamma-kummer}.
The right side lies in $H_A=\mathbb F_{q^{\Ninf}}(\theta)$
and has valuation $-1$ at $\theta=\eta_{\Ninf+1}$, since the
$\Ninf$ roots in the product are distinct. Consequently
$[H_A(\gamma):H_A]=m$. The shtuka function of the sign-normalized
module has coefficients in $H_A^+$, so $\gamma\in H_A^+$.
As $[H_A^+:H_A]=m$, this proves $H_A(\gamma)=H_A^+$,
in agreement with \cite[Section 4.1]{HHY26}.

\emph{(ii).} By (i), $H_A^+/H_A$ is Kummer of degree $m$, so every
$\sigma\in G_0$ sends $\gamma$ to $\varepsilon\gamma$ with
$\varepsilon^m=1$; being a constant, $\varepsilon$ lies in
$\iota(\mathbb F_\infty^\times)$.

\emph{(iii).} For $\sigma=\sigma_{(a)}\in G_0$,
\eqref{eq:sign-kernel-conjugation} gives $(\psi^+)^\sigma=z_a\psi^+
z_a^{-1}$ with $z_a=\LT_{\psi^+}(a)^{-1}$, and scalar conjugation
multiplies the shtuka function by $z_a^{q-1}$
(Section~\ref{sec:ideal-series}); hence
\[
\sigma(f_{\psi^+})=z_a^{q-1}f_{\psi^+}=\LT_{\psi^+}(a)^{1-q}f_{\psi^+}
=\kappa_{\psi^+,q}(\sigma)^{-1}f_{\psi^+},
\]
using $\LT_{\psi^+}(a)=\iota(\operatorname{sgn}(a))^{q^{\xi}}$ with $\xi = \xi_{\psi^+}$ and
\eqref{eq:weight-character} at $n=q$. Since the rational factor
$(t-\theta)/(t-\eta^{q^2})$ lies in $H_A(X)$, it is $G_0$-invariant,
and therefore
$\sigma(\gamma)=\kappa_{\psi^+,q}(\sigma)^{-1}\gamma$.
\end{proof}
For $r\ge1$, put $s_r=(q^r-1)/(q-1)$.
Lemma~\ref{lem:normalizing-scalar}(iii) gives the twist
$\gamma^{s_r}$. Using $f_{\psi^+}=\gamma(t-\theta)/(t-\eta_2)$
and $C_{\psi^+}=C_\psi$, equation~\eqref{eq:higher-weight-initial} gives
\begin{equation}\label{eq:projective-line-twisted-differential}
\gamma^{s_r}C_{\psi^+,q^r}
 =\frac{(\theta^{q^r}-\eta^{q^{r+1}})
             \displaystyle\prod_{j=2}^{r}(t-\eta^{q^j})}
           {\displaystyle\prod_{j=0}^{r}(t-\theta^{q^j})}\,dt.
\end{equation}
The coefficient trace of $\gamma^{s_r}C_{\psi^+,q^r}$ over
$\mathbb F_{q^{\Ninf}}(\theta)/\mathbb F_q(\theta)$ determines
$\LL_{\psi^+}^{\gamma^{s_r}}(q^r)$
by Theorem~\ref{thm:twisted-weight-trace}, for every $r\ge1$.

\begin{example}\label{ex:twisted-degree-three}
Take $q=2$ and $\rho(t)=t^3+t+1$. The roots of $\rho$ satisfy
$\Tr_{\mathbb F_8/\mathbb F_2}(\eta)=0$ and
$\Tr_{\mathbb F_8/\mathbb F_2}(\eta^6)=1$.
Equation~\eqref{eq:projective-line-twisted-differential} therefore gives
\begin{align*}
\LL_{\psi^+}^\gamma(2)\frac{\Hcal_{\psi^+}(z\tilde{\pi})}{(z\tilde{\pi})^{2}}
 &=\frac{\theta^2}{(t-\theta)(t-\theta^2)}\,dt,\\
\LL_{\psi^+}^{\gamma^3}(4)\frac{\Hcal_{\psi^+}(z\tilde{\pi})}{(z\tilde{\pi})^{4}}
 &=\frac{\theta^4t+1}
        {(t-\theta)(t-\theta^2)(t-\theta^4)}\,dt.
\end{align*}
Both weights lie outside $n\equiv1\pmod7$.
In each case, $[H_n:H_A]=7$ by
Proposition~\ref{prop:weight-field}, so $C_{\psi^+,n}$ itself has
smallest coefficient field $H_A^+$ over $H_A$.
\end{example}

\subsection{Carlitz specializations}
Take $A=\Fq[t]$ with its usual rational infinite point,
and put $\theta=\iota(t)$. Write $C$ for the Carlitz module,
$A\tilde\pi_C$ for its period lattice, and $\omega_C$ for the
Anderson--Thakur special function.

\begin{example}[Perkins's generating identity]\label{ex:carlitz-generating}
Write
\[
s_C(w;t)=\sum_{j\ge0}\exp_C(w/\theta^{j+1})t^j
\]
for the Anderson generating function, so that
$s_C(\tilde\pi_C;t)=\omega_C(t)$.
The shtuka function is $f_C=t-\theta$, and the Cauchy
normalization gives
$\Hcal_C(w)=-s_C(w;t)\,dt$.
Thus \eqref{eq:free-generating} becomes
\begin{equation}\label{eq:carlitz-perkins-generating}
\sum_{0\ne a\in A}\frac{\chi(a)}{z-\iota(a)}
 =\frac{\tilde\pi_C}{\exp_C(\tilde\pi_C z)}
       \frac{s_C(\tilde\pi_C z;t)}{\omega_C(t)}.
\end{equation}
This is the $s=1$ specialization of \cite[Theorem 1.1]{Perkins14};
the $a=0$ term may be inserted as zero, and
\cite{PP16} proves functional identities for this family of lattice
generating functions.
It holds in $\Cinf\langle t\rangle$ for $|z|<1$, with
the quotient continued at zero.
Since $\Cl(A)$ is trivial and $\Ninf=1$,
\eqref{eq:ideal-generating-trace} also identifies the right side with
\begin{equation}\label{eq:carlitz-ideal-generating}
\sum_{k\ge0}\LL_C(1+k(q-1))z^{k(q-1)}.
\end{equation}
 The constant coefficient is Pellarin's value at one;
the higher coefficients are the ideal-type values at weights
$1+k(q-1)$, $k\ge1$. 
\end{example}

\begin{example}[Pellarin's identity and Frobenius values]
At the usual rational infinity, the Carlitz normalization is
$C_C=dt/(t-\theta)$ and $\Hcal_C(\tilde{\pi}_C)=-\omega_Cdt$.
Here the sign agrees with the residue formula for $\omega_C$ in
\cite[Section 2.5]{AT90} and \cite[equations (1.1)--(1.2)]{EP14}.
We recover Pellarin's formula
$\LL_C(1)=-\tilde{\pi}_C/((t-\theta)\omega_C)$
\cite[Theorem 1]{Pellarin12}.
Since $f_C=t-\theta$, Theorem~\ref{thm:higher-weight-trace}
and \eqref{eq:higher-weight-initial} give, for every $r\ge0$,
\begin{equation}\label{eq:carlitz-frobenius-weight}
\LL_C(q^r)
 =-\frac{\tilde{\pi}_C^{q^r}}
        {\omega_C\displaystyle\prod_{j=0}^{r}(t-\theta^{q^j})}.
\end{equation}
This agrees with twisting the element-type identity, since all ideals are
principal and $\omega_C^{(r)}=F_C^{[r]}\omega_C$.
\end{example}
\begin{example}\label{ex:weight-congruence}
Let $A=\Fq[t]$, $q>2$, and let $n\ge1$ satisfy
$n\not\equiv1\pmod{q-1}$.
Although $\LL_C(n)=L_+(A;n)$, the normalized element-type differential
\[
B_n=L_+(A;n)\frac{\Hcal_C(\tilde{\pi}_C)}{\tilde{\pi}_C^n}
     =-\frac{L_+(A;n)\omega_C}{\tilde{\pi}_C^n}\,dt
\]
is not rational over $H_A=K=\Fq(\theta)$.
Indeed, choose $\mu^{q-1}=-\theta$. The product formulas
for the Anderson--Thakur special function and the Carlitz period are 
\[
\omega_C=\mu\prod_{j\ge0}(1-t/\theta^{q^j})^{-1},\]
 \[
\tilde{\pi}_C=\mu\theta
        \prod_{j\ge1}(1-\theta^{1-q^j})^{-1}.
\]
Thus $\omega_C/\tilde\pi_C^n$ belongs to
$K_\infty(\mu)\langle t\rangle$. Since
$L_+(A;n)\in K_\infty\langle t\rangle$, it follows that
$B_n\in K_\infty(\mu)\langle t\rangle\,dt$.

The totally ramified extension $K_\infty(\mu)/K_\infty$ has
automorphisms $\sigma_\zeta(\mu)=\zeta\mu$, $\zeta\in\Fq^\times$.
 These automorphisms act continuously on
$K_\infty(\mu)\langle t\rangle$ by acting on coefficients and
fixing $t$. They fix $L_+(A;n)$, so
$\sigma_\zeta(B_n)=\zeta^{1-n}B_n$. 
Moreover $B_n\ne0$, since
$\|L_+(A;n)-1\|\le q^{-n}<1$.
Choose $\zeta\in\Fq^\times$ with $\zeta^{1-n}\ne1$.
Then $B_n\notin K_\infty(t)\,dt$.
Thus $q=3$, $n=2$ gives a counterexample to extending
\eqref{eq:type-trace-higher} to every weight as an ordinary rational
Hilbert-class trace. The series comparison
\eqref{eq:type-one-type-two} remains valid.
\end{example}

\section{Elliptic curves}\label{sec:elliptic}
On an elliptic curve, the poles and residues determine the
Cauchy differential up to a holomorphic differential; the shtuka
divisor determines the remaining term. At rational infinity, the
resulting trace recovers the Green--Papanikolas formula.
A degree-three infinite point gives an explicit ideal value with
a nonzero holomorphic contribution and a residue-dual basis for
the finite torsion-trace formula.

\subsection{The Green--Papanikolas value at rational infinity}
Let $E/\Fq$ be the elliptic curve in Weierstrass form
\[
E:\quad y^2+a_1ty+a_3y=t^3+a_2t^2+a_4t+a_6,
\]
with $a_i\in\Fq$ and origin $O$. Put
$A=\Gamma(E\setminus\{O\},\OO_E)$ and $K=\Fq(E)$, so the removed
point is $\infty=O$. Write $\theta=\iota(t)$ and $\eta=\iota(y)$.
Normalize $|a|_\infty=q^{-v_O(a)}$ for $a\in K^\times$; then
$\Xi=(\theta,\eta)$ satisfies $|\theta|=q^2$ and $|\eta|=q^3$.
For the Green--Papanikolas module, denoted here by $\psi$, the chosen
shtuka divisor is $V_\psi=[(\alpha,\beta)]$ with $|\alpha|=q^2$.
The shtuka function of $\psi$ and the invariant differential on $E$ are
\[
f_\psi=\frac{y-\eta-m(t-\theta)}{t-\alpha},\qquad
m=\frac{\eta-\beta^q}{\theta-\alpha^q},\qquad
\nu=\frac{dt}{2y+a_1t+a_3}.
\]
These are the conventions of \cite[equation (18)]{GP18}.
Put $\kappa=(\eta-m\theta)/\alpha$ and choose
$\beta_*^{q-1}=\kappa$ as in the Green--Papanikolas normalization of the special function.
The estimate of \cite[Example 7.4]{HuCauchy} is
\[
 |\kappa|=q^q,\qquad
 \|f_\psi/\kappa-1\|_{\mathrm G}\le q^{-q-1}<1.
\]
Thus \cite[Proposition 7.3 and Theorem 7.6]{HuCauchy} give
\[
\omega_\psi=\beta_*\prod_{j\ge0}
             \frac{\kappa^{q^j}}{f_\psi^{(j)}}\in\mathbb T_A^\times,
\qquad \sfm(\psi)=A\omega_\psi.
\]
Since $\Omega_A=A\nu$, the residue isomorphism proves that
$\Lambda_\psi$ is free.  With the period convention of
\cite[equation (66)]{GP18}, 
\[
\tilde\pi=-\Res_\Xi(\omega_\psi\nu),\qquad
\Lambda_\psi=A\tilde\pi,\qquad
\Hcal_\psi(\tilde\pi)=-\omega_\psi\nu.
\]
The module $\psi$ and period $\tilde\pi$ are denoted by
$\rho$ and $\pi_\rho$ in \cite{GP18}.
With $\delta=t-\alpha$, the Green--Papanikolas
element-type formula
\cite[Theorem 7.1]{GP18} identifies
$C_\psi=(\delta^{(1)}/f_\psi)\nu$.
Corollary~\ref{cor:hayes-trace} therefore gives 
\[
\LL_\psi(1)
 =-\frac{\tilde\pi}{\omega_\psi}
   \sum_{\sigma\in\operatorname{Gal}(H_A/K)}
           \left(\frac{\delta^{(1)}}{f_\psi}\right)^\sigma .
\]
This identity recovers \cite[Theorem 7.3]{GP18}.
Corollary~\ref{cor:class-number} recovers the characteristic value
$\LL_\psi(1)|_\Xi=h(A)\cdot1$ of \cite[Corollary 7.4]{GP18}. 

\subsection{The shtuka datum and the free lattice}

Consider
\begin{equation}\label{eq:elliptic-curve}
E/\mathbb F_2:\qquad y^2+y=x^3+x+1,
\end{equation}
with rational origin $O$. Put $\rho(T)=T^3+T+1$, choose a root
$\alpha\in\mathbb F_8$, and let $Q$ be the degree-three closed
point with geometric points
$(\alpha,0),(\alpha^2,0),(\alpha^4,0)$.
We remove $Q$ and set $A=\Gamma(E\setminus\{Q\},\OO_E)$;
thus $\infty=Q$ and $\Ninf=3$, while $O$ lies on the affine curve.
 Put $\theta=\iota(x)$ and $\epsilon=\iota(y)$.
Then $\Xi=(\theta,\epsilon)$ and
\[ \epsilon^2+\epsilon=\rho(\theta). 
\]
 Use the normalization $|a|_\infty=2^{-3v_Q(a)}$.
Label $\alpha$ so that $\theta$ reduces to $\alpha^4$ in
the fixed completion. Since $y$ is a uniformizer at $Q$ and
$\rho'(\alpha^4)\ne0$, one has
$|\epsilon|=|\theta-\alpha^4|=2^{-3}$. In particular, the
arithmetic point has reduction

\begin{equation}\label{eq:elliptic-local-normalization}
\overline{\Xi}=(\overline\theta,\overline\epsilon)
 =q_0=(\alpha^4,0).
\end{equation}

The affine Weierstrass equation \eqref{eq:elliptic-curve}
has no solutions over $\mathbb F_2$, so $E(\mathbb F_2)=\{O\}$.
Let $\mathfrak m_O\subset A$ be the maximal ideal of $O$.
The degree map and
$\divisor(y)=Q-3O$ give
\[
\Cl(A)\simeq\mathbb Z/3\mathbb Z,\qquad
\mathfrak m_O^3=(1/y),\qquad
H_A=\mathbb F_8(\theta,\epsilon).
\]
Indeed, the degree-zero divisor class group is $E(\mathbb F_2)$
\cite[Chapter III, Section 3]{Silverman09}; the displayed constant
extension is unramified and splits $Q$ completely
\cite[Chapter 8]{Rosen02} and
\cite[Theorem 15.6]{Hayes92}.
The classes are represented by $A,\mathfrak m_O,\mathfrak m_O^2$.
Write $\sigma$ for the coefficient automorphism fixing
$\theta,\epsilon$ and sending $\alpha$ to $\alpha^2$.

Let $\mathcal F(x,y)=(x^2,y^2)$ be the Frobenius
endomorphism of $E$. Choose the shtuka pole
$q_\psi=\mathcal F(q_0)=(\alpha,0)$ and set $S=(s,w)$ with
\begin{equation}\label{eq:elliptic-V}
m=\frac{\epsilon+1}{\theta+\alpha},\qquad
u=m^2+\alpha+\theta,\qquad
s=u+1,\qquad w=m(\alpha+u)+1+u.
\end{equation}
Let $\ell$ and $k$ be the slopes of the chords through $(S,q_\psi)$
and $(S^{(1)},\Xi)$, respectively. Define $f_\psi$ as the linear
form of the second chord divided by that of the first:
\begin{equation}\label{eq:elliptic-shtuka}
\ell=\frac{w}{s+\alpha},\qquad
k=\frac{w^2+\epsilon}{s^2+\theta},\qquad
f_\psi(x,y)=\frac{y+\epsilon+k(x+\theta)}
                   {y+\ell(x+\alpha)}.
\end{equation}
Both linear forms have leading term $y$ at $O$, so
$f_\psi(O)=1$. The resulting module $\psi$ has free period
lattice by Proposition~\ref{prop:elliptic-module}.
\begin{lemma}\label{lem:elliptic-shtuka}
The geometric data $V_\psi=[S]$, $q_\psi$, and $f_\psi$ determine
a rank-one module over $H_A = \mathbb F_8(\theta,\epsilon)$.
\end{lemma}
\begin{proof}
On this curve, the Lang map $\mathcal F-1$ is an
automorphism. On affine points $P=(x,y)$, it is given by
\[
\mathcal F(P)-P=(x+1,y+x+1),
\]
with inverse 
\begin{equation}\label{eq:elliptic-lang}
 (\mathcal F-1)^{-1}(u,v)=(u+1,v+u).
\end{equation}
For the first identity, the chord joining $(x^2,y^2)$ and
$-(x,y)=(x,y+1)$ has slope
\[
\frac{y^2+y+1}{x^2+x}
 =\frac{x^3+x}{x^2+x}=x+1.
\]
 Both displayed coordinate maps extend to $E$ by fixing $O$.
The addition formulas prove the first identity on a dense open set,
and hence as an identity of morphisms on $E$. Composing the two
maps verifies the inverse. 
The point $(u,m(\alpha+u)+1)$ is $q_\psi-\Xi$.
Thus \eqref{eq:elliptic-V} gives
$(\mathcal F-1)(S)=q_\psi-\Xi$.
The numerator and denominator of \eqref{eq:elliptic-shtuka}
define the chords through $(S^{(1)},\Xi)$ and $(S,q_\psi)$,
respectively. These chords have the same third intersection point
with $E$, because $S^{(1)}+\Xi=S+q_\psi$.
Cancelling the common zero in their ratio gives
\[
\divisor(f_\psi)=[S^{(1)}]-[S]+[\Xi]-[q_\psi].
\]
The relation $(\mathcal F-1)(S)=q_\psi-\Xi$ shows that
$S$ is not a constant point, so neither $S$ nor any of its
Frobenius twists lies above $Q$. Moreover, $S\ne\Xi$, since
otherwise $\mathcal F(\Xi)=q_\psi$. Thus $V_\psi=[S]$ avoids
$\Xi$ and $q_\psi^{(-1)}$, as required by the geometric realization.

Put
$D_n=[S]+\sum_{j=0}^{n-1}[q_\psi^{(j)}]$ for $n\ge0$.
Riemann--Roch gives $h^0(D_n)=n+1$. For $n\ge1$,
\eqref{eq:product-divisor} puts $F_\psi^{[n]}$ in
$H^0(E_{H_A},\OO(D_n))$ and gives it a pole at
$q_\psi^{(n-1)}$ of order one greater than allowed by $D_{n-1}$.
There is no cancellation at this point, since the twists of
$S$ and $\Xi$ avoid $Q$. Induction therefore shows that
$1,F_\psi^{[1]},\ldots,F_\psi^{[n]}$ form a basis of
$H^0(E_{H_A},\OO(D_n))$. Their union is a basis of
$\Gamma(E_{H_A}\setminus Q,\OO(S))$, which is therefore free
over $H_A\{\tau\}$ with generator $1$ for
$\tau(g)=f_\psi g^{(1)}$.

For $a\in A$, expand
$\chi(a)=\sum_i\psi_{a,i}F_\psi^{[i]}$ and define
$\psi_a=\sum_i\psi_{a,i}\tau^i$, as in
\cite[Section 2.2]{HuCauchy}.
Multiplication by $\chi(a)$ commutes with $\tau$, so these
expansions respect addition and multiplication in $A$.
The coordinates of $S$ and $f_\psi$ lie in $H_A$, so the
coefficients $\psi_{a,i}$ also lie in $H_A$.
Evaluation at $\Xi$ gives $\psi_{a,0}=\iota(a)$, and the
pole orders give $\deg_\tau\psi_a=\deg(aA)$ for $a\ne0$.
Thus the data define a rank-one Drinfeld module over $H_A$.
\end{proof}

\begin{proposition}
\label{prop:elliptic-module}
For the module specified by \eqref{eq:elliptic-shtuka},
define the special function and period by
\begin{equation}\label{eq:elliptic-unit-period}
\omega_\psi=\prod_{j\ge0}(f_\psi^{(j)})^{-1},\qquad
\tilde\pi=\Res_\Xi(\omega_\psi dx).
\end{equation}
Then $\omega_\psi\in\mathbb T_A^\times$,
$\sfm(\psi)=A\omega_\psi$, $\Lambda_\psi=A\tilde\pi$, and
$\Hcal_\psi(\tilde\pi)=\omega_\psi dx$.
\end{proposition}
\begin{proof}
The coordinates $s,w$ lie in $\iota(K_\infty)$, and reduction in
\eqref{eq:elliptic-local-normalization} gives
$\overline S=q_0$.
Since $\epsilon$ is a uniformizer of $\iota(K_\infty)$,
$|w|\le|\epsilon|=2^{-3}$. The denominators $s+\alpha$ and
$s^2+\theta$ have absolute value one, so $|\ell|,|k|\le2^{-3}$.
Since $1/y,x/y\in A$, rewrite \eqref{eq:elliptic-shtuka} as
\[
f_\psi=\frac{1+k\,x/y+(\epsilon+k\theta)/y}
              {1+\ell\,x/y+\ell\alpha/y}.
\]
The numerator and denominator each differ from $1$ by an element
of Gauss norm at most $2^{-3}$. Hence
\begin{equation}\label{eq:elliptic-strict-unit}
f_\psi\in\mathbb T_A^\times,\qquad
 \|f_\psi-1\|_{\mathrm G}\le2^{-3}<1.
\end{equation}
 The strict-unit bound and \cite[Proposition 7.3 and
Theorem 7.6]{HuCauchy}, with normalizing scalar $1$, imply that
the product in \eqref{eq:elliptic-unit-period} converges to
$\omega_\psi\in\mathbb T_A^\times$ and that
$\sfm(\psi)=A\omega_\psi$. 
The differential $dx$ has no zeros or poles on $E$, so $\Omega_A=A\,dx$.
The residue isomorphism now gives
$\Lambda_\psi=A\tilde\pi$ and
$\Hcal_\psi(\tilde\pi)=\omega_\psi dx$.
The inverse Frobenius twist of the shtuka pole is
$q_\psi^{(-1)} =(\alpha^4,0)$.  

\end{proof}


\subsection{An explicit value at a degree-three infinite point}
\label{sec:elliptic-closed}

\begin{theorem}
\label{thm:elliptic-explicit}
For the module $\psi$ of
Proposition~\ref{prop:elliptic-module}:
\begin{enumerate}
\item[(i)] the ideal Pellarin value of $\psi$ satisfies
\begin{equation}\label{eq:elliptic-explicit}
\frac{\LL_\psi(1)\omega_\psi}{\tilde\pi}
 =\frac{y+\epsilon+1}{x+\theta}
  +\frac{x^2+1}{y}
  +\frac{(\theta^2+\theta+1)(\epsilon+1)}{\epsilon^2}.
\end{equation}
After multiplication by $dx$, the sum of the first two terms
has poles only at $\Xi$ and the points above $Q$.
The last term, multiplied by $dx$, is the coefficient trace of the
holomorphic correction in $C_\psi$.
\item[(ii)] $\LL_\psi(1)|_\Xi=1$.
\end{enumerate}
\end{theorem}
\begin{proof}
\emph{(i).} Define
\[
G_\alpha(x,y)=\frac{y+\epsilon+1}{x+\theta}
                    +\frac{y+1}{x+\alpha^4},\qquad
\gamma_\alpha=G_\alpha(s,w).
\]
The differential $G_\alpha dx$ has simple poles at $\Xi$
and $q_\psi^{(-1)}$ with residues $1$ and $-1$.
The apparent poles at $-\Xi$ and $-q_\psi^{(-1)}$ are removable:
the corresponding numerators in $G_\alpha$ vanish there.
At $O$, using the uniformizer $z=x/y$, each fraction has expansion
$z^{-1}+O(z)$, so their principal parts cancel.
Adding $\gamma_\alpha dx$ imposes the required zero at $S$,
since $\gamma_\alpha=G_\alpha(s,w)$ and the characteristic is two.
The normalized Cauchy differential is consequently
\begin{equation}\label{eq:elliptic-C}
C_\psi=(G_\alpha+\gamma_\alpha)dx.
\end{equation}
The poles and residues determine $C_\psi$ up to an additive
multiple of $dx$; vanishing at $S$ determines this holomorphic correction.

Substitute \eqref{eq:elliptic-V} in $G_\alpha(s,w)$ and reduce by
$\alpha^3=\alpha+1$ and $\epsilon^2=\epsilon+\rho(\theta)$.
This gives
\[
\gamma_\alpha=\frac{D_0+D_1\alpha+D_2\alpha^2}{\epsilon^2},
\]
where
\begin{align*}
D_0&=(\theta^2+\theta+1)(\epsilon+1),\\
D_1&=\epsilon^2+(\theta+1)\epsilon+\theta^2,\\
D_2&=\epsilon^2+\theta\epsilon+1.
\end{align*}
For the coefficient trace $\Tr=\Tr_{H_A/K}$, one has
$\Tr(1)=1$ and $\Tr(\alpha)=\Tr(\alpha^2)=0$.
Thus $\Tr_{H_A/K}(\gamma_\alpha)=D_0/\epsilon^2$.
For the geometric terms,
\[
\sum_{j=0}^2\frac1{x+\alpha^{2^j}}
 =\frac{\rho'(x)}{\rho(x)},\qquad
 (y+1)\frac{\rho'(x)}{\rho(x)}=\frac{x^2+1}{y}.
\]
Taking the coefficient trace of \eqref{eq:elliptic-C} and using
\eqref{eq:hayes-trace} gives \eqref{eq:elliptic-explicit}.
\emph{(ii).} Finally, $h(A)=3=1$ in characteristic two, so
Corollary~\ref{cor:class-number} gives the characteristic value.
\end{proof}

For a general rank-one source $\varphi$, choose a datum
$J:\varphi\to\psi$ with this free target. The trace in Theorem~\ref{thm:general-value} is the sum of
the three coefficient conjugates of $C_\psi/g_J$:
\begin{equation}\label{eq:elliptic-three-trace}
\LL_\varphi(1)\Gcal_J(z_J)
 =z_J\sum_{j=0}^2
   \frac{G_{\alpha^{2^j}}+\gamma_{\alpha^{2^j}}}
        {g_J^{\sigma^j}}\,dx.
\end{equation}
Here $g_J=h_J/\partial\eta_J$ is computed from the chosen isogeny
by Lemma~\ref{lem:motive-multiplier}.
Taking $J=\mathrm{id}_\psi$, so that $g_J=1$ and
$z_J=\tilde\pi$, recovers \eqref{eq:elliptic-explicit}.

\subsection{Residue-dual data for the finite algorithm}
To apply \eqref{eq:finite-algorithm} to this curve, take $a=1/y$,
whose pole order at $Q$ is one. The space $H^0(E,\Omega_E^1(2Q))$
has the basis $\nu_j=v_j dx$ in the table below. Put
$b=1/y^3$ and $z=x/y$. Then $\deg(bA)=9>6$, as required by
Lemma~\ref{lem:dual-denominator}.  Explicit dual numerators subject to $\langle c_i/b,\nu_j\rangle _{Q}=\delta_{ij}$ are 
\[
\begin{array}{c|c|l}
j&v_j&c_j\\\hline
 1&1&z^8\\
 2&1/y&z^5+z^8\\
 3&x/y&z^3+z^6\\
 4&1/y^2&z^2+z^6\\
 5&x/y^2&z^3+z^4+z^6+z^8\\
 6&x^2/y^2&z^6
\end{array}
\]
The functions $v_j$ and $c_j$ belong to $A$. To verify that $(\nu_j)$ is a basis, use $\Omega_A=A\,dx$
and multiply each $v_j$ by $y^2$. Riemann--Roch and the distinct pole
orders $0,2,3,4,5,6$ show that $1,x,y,x^2,xy,x^3$ form a basis of
$H^0(E,\OO(6O))$. The relation $x^3=y^2+y+x+1$ shows that the
six functions $y^2v_j$ form another basis of this space.

For the duality calculation, $\divisor(b)=9O-3Q$, so
$A/bA\simeq\mathbb F_2[z]/(z^9)$. The global residue theorem reduces
the pairing to
\[
\langle z^k/b,\nu_j\rangle _{Q}
       =\res_O(z^k y^3v_jdx).
\]
Substituting $y=x/z$ and $u=1/x$ in \eqref{eq:elliptic-curve} gives
$u=z^2+(z+z^2)u^2+z^2u^3$. Successive substitution yields
\[
x=z^{-2}+z+z^2+z^6+z^7+z^8+O(z^{10}),\qquad
dx=(1+z^6+O(z^{10}))\,dz.
\]
The residue matrix
$\bigl(\langle z^k/b,\nu_j\rangle _Q\bigr)_{k,j}$, with rows
$k=0,\ldots,8$ and columns $j=1,\ldots,6$, is
\[
\begin{pmatrix}
 0&0&0&0&0&0\\0&0&0&0&0&0\\
 0&0&0&1&0&1\\0&0&1&0&0&1\\
 1&0&1&0&1&0\\1&1&0&0&0&0\\
 0&0&0&0&0&1\\0&0&1&0&0&0\\1&0&0&0&0&0
\end{pmatrix}.
\]
Form a $6\times9$ matrix whose rows are the coordinates
of $c_1,\ldots,c_6$ in the basis $(1,z,\ldots,z^8)$ of $A/bA$.
Its product with the residue matrix is $I_6$, verifying the
claimed duality.
For a rank-one action $\varphi:A\to H\{\tau\}$, put
$\mathscr T=H[Z]/(\varphi_{1/y^3}(Z)/Z)$.
Since $\iota(a)=\epsilon^{-1}$ and $\iota(b)=\epsilon^{-3}$,
substitution in \eqref{eq:finite-algorithm} gives
\[
C_\varphi=\frac{\epsilon^4y}{y+\epsilon}
\sum_{j=1}^6\Tr_{\mathscr T/H}
\left(\frac{\varphi_{c_j/y}(Z)+\epsilon^{-1}\varphi_{c_j}(Z)}{Z}\right)
v_j\,dx.
\]
The residue-dual basis is independent of $\varphi$.
The displayed formula reconstructs $C_\varphi$ from the action
polynomials $\varphi_{1/y^3}$, $\varphi_{c_j/y}$, and
$\varphi_{c_j}$ by traces in the algebra $\mathscr T$ of
dimension $2^9-1=511$ over $H$.

\subsection{Acknowledgments}
ChatGPT and OpenAI Codex (OpenAI) were used during manuscript preparation for structural organization, English-language refinement, and LaTeX formatting. These tools were also employed to support preliminary discussions of proof strategies, logical consistency checks, and alternative argument formulations. All content, revisions, and mathematical reasoning presented in the manuscript have been fully reviewed, verified, and validated by the authors. The authors take exclusive and full responsibility for all mathematical statements, proofs, analytical arguments, references, and every part of the final manuscript, and ensure full adherence to all publication ethics standards.
 
\bibliographystyle{amsplain} 
\bibliography{paper}

\end{document}